\documentclass{SCIYA2017enOL}
\usepackage{enumerate, verbatim}
\usepackage{color, graphicx, enumitem}

\newcommand{\blue}{\textcolor{black}}

\newcommand{\bX}{{\boldsymbol X}}
\newcommand{\bY}{{\boldsymbol Y}}

\newcommand{\bx}{{\boldsymbol x}}
\newcommand{\by}{{\boldsymbol y}}

\newcommand{\BR}{\mathbb{R}}

\newcommand{\bbeta}{{\boldsymbol \beta}}

\newcommand{\btheta}{{\boldsymbol \theta}}

\newcommand{\bvarepsilon}{{\boldsymbol \varepsilon}}

\def\boxit#1{\vbox{\hrule\hbox{\vrule\kern6pt
          \vbox{\kern6pt#1\kern6pt}\kern6pt\vrule}\hrule}}

\online
\begin{document}

\ensubject{fdsfd}

\ArticleType{ARTICLES}
\Year{2017}
\Month{January}%
\Vol{60}
\No{1}
\BeginPage{1} %
\DOI{10.1007/s11425-000-0000-0}
\ReceiveDate{January 1, 2017}
\AcceptDate{January 1, 2017}

\title[]{Nearly optimal Bayesian Shrinkage for High Dimensional Regression}
{Nearly optimal Bayesian Shrinkage for High Dimensional Regression}

\author[1*]{Qifan Song}{{qfsong@purdue.edu}}
\author[1]{Faming Liang}{{fmliang@purdue.edu}}

\AuthorMark{Song, Q.}

\AuthorCitation{Song, Q. and Liang, F.}

\address[1]{Department of Statistics, Purdue University, West Lafayette, Indiana {\rm47906}, U.S.A}

\abstract{ During the past decade, shrinkage priors have received much attention in Bayesian analysis  of high-dimensional data. This paper establishes the posterior consistency for high-dimensional linear regression with a class of shrinkage priors, which has   a heavy and flat tail and allocates a sufficiently large probability mass in a very small neighborhood of zero.  While enjoying its efficiency in posterior simulations, the shrinkage prior can lead to a nearly optimal posterior contraction rate and variable selection consistency as the spike-and-slab prior. Our numerical results show that under the posterior consistency, Bayesian methods can yield much better results in variable selection than the regularization methods such as Lasso and SCAD. This paper also establishes a Bernstein von-Mises type result, which leads to a convenient way of uncertainty quantification  for regression coefficient estimates. }



\maketitle

\section{Introduction}\label{intro}

 The dramatic improvement in data collection and acquisition technologies during the last two decades 
 has enabled scientists to collect a great amount of high-dimensional data. 
 Due to their intrinsic nature, many of the high-dimensional data, such as 
 omics data and SNP data, have a much smaller sample size than their dimension (a.k.a. small-$n$-large-$p$). 
 Toward an appropriate understanding of the system underlying the small-$n$-large-$p$ data, 
 variable selection plays a vital role. In this paper, we consider the problem of variable selection 
 for the high-dimensional linear regression 
\begin{equation}\label{lm}
 \by = \bX\bbeta + \sigma\bvarepsilon,
\end{equation}
 where $\by$ is an $n$-dimensional response vector, $\bX$ is an $n$
 by $p$ design matrix, $\bbeta$ is the vector of regression coefficients, 
 $\sigma$ is the standard deviation, and $\bvarepsilon$ follows N$(0,I_n)$.
 This problem has received much attention in the recent literature. Methods have been 
 developed from both frequentist and Bayesian perspectives. The frequentist methods are 
 usually regularization-based, which enforce the model sparsity through imposing a penalty 
 on the negative log-likelihood function. For example,  Lasso \cite{Tibshirani1996}
 employs a $L_1$-penalty, elastic net \cite{ZouH2005} employs a combination of 
 $L_1$ and $L_2$ penalty,  \cite{FanL2001} employs a smoothly clipped absolute 
 deviation (SCAD) penalty,  \cite{Zhang2010} employs a minimax concave penalty (MCP), 
 and rLasso \cite{SongL2014_2} employs a reciprocal $L_1$-penalty.  
 In general, these penalty functions encourage model sparsity, 
 and tend to shrink the coefficients of false predictors to exactly zero.
 Under appropriate conditions, 
 consistency can be established for both variable selection and parameter estimation. 

 The Bayesian methods encourage sparsity of the {\it posteriori} model through 
 choosing appropriate prior distributions. 
 A classical choice is the spike-and-slab prior,
 $\beta_j \sim r h(\beta_j)+(1-r) \delta_0(\beta_j)$,
 where $\delta_0(\cdot)$ is the degenerated ``spike distribution'' at zero,
 $h(\cdot)$ is an absolutely continuous ``slab distribution'', 
 and $r$ is the prior mixing proportion. Generally, it can be equivalently represented as the
 following hierarchical prior, 
\begin{equation}\label{mixprior}
\begin{split}
   &\xi\sim\pi(\xi),\quad\bbeta_\xi \sim h_{\xi}(\bbeta_\xi), \quad\bbeta_{\xi^c}\equiv 0,
  \end{split}
\end{equation}
for some multivariate density function $h_{\xi}$, where
$\xi$ denotes a subset model,
 $\bbeta_{\xi}$ and $\bbeta_{\xi^c}$ denote 
 the coefficient vectors of the covariates included in and excluded from the model $\xi$,   respectively. 
 The theoretical properties of prior (\ref{mixprior}) have been thoroughly investigated
 \cite{ScottB2010,JohnsonR2012,Jiang2007,LiangSY2013, NarisettyH2014,SongL2014, YangWJ2015,CastilloSHV2015,MartinMW2015}.
 Under proper choices of $\pi$ and $h_{\xi}$, the spike-and-slab prior achieves a (nearly-) optimal
 contraction rate and model selection consistency. 

 Alternative to the hierarchical priors, some shrinkage priors 
 have been proposed for (\ref{lm}) motivated by the equivalence between the regularization 
 estimator and the maximum {\it a posteriori} (MAP) estimator, see e.g. the 
 discussion in \cite{Tibshirani1996}. Examples of such priors include the Laplace 
 prior \cite{ParkC2008, Hans2009}, horseshoe prior \cite{CarvalhoPS2010},
 structuring shrinkage prior \cite{GriffinB2012}, double Pareto shrinkage prior \cite{AemaganDL2013_2},
  Dirichlet Laplace prior \cite{BhattacharyaPPD2015}, and elliptical Laplace prior \cite{GaoVZ2020}.  
 Compared to the hierarchical prior, 
 the shrinkage prior is 
  conceptually much simpler. The former involves specification of priors 
 for a large set of models, while the latter avoids this issue as for which only a single model is considered. Consequently, for the hierarchical prior, 
 a trans-dimensional MCMC sampler is required for simulating of the posterior in a huge space of submodels, and this has constituted the major 
 obstacle for the use of Bayesian methods in high-dimensional variable selection. 
 For the shrinkage prior, there is only a single model used in posterior simulations, and thus some gradient-based MCMC algorithms, such as stochastic gradient Langevin dynamics (SGLD) \cite{Welling2011BayesianLV}, 
 Hamiltonian Monte Carlo \cite{Duane1987, Neal2011},  Riemann manifold Hamiltonian Monte Carlo \cite{GirolamiC2011}, and stochastic gradient 
 Hamiltonian Monte Carlo \cite{SGHMC}, can be 
 easily used in simulations.
  This is extremely attractive for the problems both $n$ and $p$ are very large, for which mini-batch data can be conveniently used to accelerate simulations.  
 
 Despite the popularity and potential advantages of shrinkage priors, few works have been done to study their theoretical properties. There is a lack of general guarantee of posterior consistency for Bayesian shrinkage priors, especially under the high dimensional setting.
 Bayesian community already realized that the Laplace distribution is not a good shrinkage prior for high-dimensional linear regression. 
 \cite{BhattacharyaPPD2015,CastilloSHV2015} showed that the $L_2$ contraction rate of Bayesian Lasso is suboptimal, and 
 we found that under regularity conditions, the posterior of Bayesian Lasso is actually inconsistent in the $L_1$ sense (this result is presented in the 
 supplementary material). To tackle this issue, many other types of shrinkage priors have been proposed, see e.g. \cite{ArmaganCD2011,ArmaganDL2013,BhattacharyaPPD2015,
 CarvalhoPS2010,GhoshC2014,GriffinB2012,GriffinB2011}.
 In the literature, there have been rich theoretical results on Bayesian shrinkage prior for the case of slowly increasing $p$ (i.e., $p=o(n)$) \cite{ArmaganDL2013,Bontemps2011,Ghosal1999} and normal means models \cite{BhattacharyaPPD2015, GhoshC2014,VanKV2014, VanSV2017}.
 For the high-dimensional case, i.e., $p>n$, the non-invertibility and eigen-structure 
 of the Gram matrix $\bX^T\bX$ complicate the analysis. Hence, the results derived from low dimensional models or normal means models don't trivially apply to regression problems.
 It is worth to note that, most of the Bayesian works 
 for normal means models \cite{VanKV2014,
 BhattacharyaPPD2015,Castillov2012} aimed to achieve a minimax contraction rate of $O(\sqrt{s\log(n/s)})$.  A recent preprint \cite{SongC2018} shows that for normal means problem,
 {\it any monotone estimator $\widehat\bbeta$  which asymptotically guarantees no false discovery has
 at best the $L_2$-estimation error rate  $\|\widehat\bbeta-\bbeta^*\|_2=O_p(\sqrt{s\log n})$}. This frequentist assertion implies that
 the existing rate-minimax Bayesian approaches cannot consistently recover the underlying sparsity structure for normal means models [see also Theorem 3 in \cite{VanSV2017_2} and Theorem 3.4 in \cite{BhattacharyaPPD2015}]. 
 For high-dimensional regression models, variable selection consistency remains an unresolved issue for Bayesian shrinkage priors.

 In this paper, we lay down a general theoretical foundation  for Bayesian high dimensional linear regression with shrinkage priors.  
 Instead of focusing on certain types of shrinkage priors, we investigate  
 sufficient conditions of posterior consistency for general shrinkage priors.
 We show that if the prior density has a dominating peak around zero, and 
 a heavy and flat tail, then its theoretical properties are as good as the spike-and-slab prior:
 Its contraction rate is nearly optimal, variable selection is consistent, and posterior follows a BvM-type phenomenon. Specifically, we consider
 two types of shrinkage priors for high dimensional linear regression, namely, 
 polynomially decaying priors and 
 two-Gaussian-mixture priors \cite{GeorgeM1993}.  
 Empirical studies show that the Bayesian method with a consistent shrinkage prior can lead to more accurate 
 results in variable selection than the regularization methods. The general theoretical framework and technical tools developed in this paper 
 have inspired a series of follow-up works, see e.g., R2-D2 shrinkage prior \cite{ZhangRB2017}, Beta prime prior \cite{bai2021beta}, and Bayesian additive 
nonparametric regression \cite{wei2020sparse}.
  
\textcolor{black}{
Finally, we note that there are some other Bayesian works which deal with high-dimensional problems with shrinkage priors. For example,  \cite{PatiBPD2014} employed a Dirichlet-Laplace (DL) prior in dealing with  high-dimensional factor models, but their results only allow the magnitude of true parameters to increase very slowly with $n$;}
 \textcolor{black}{\cite{BhadraDLP2016} studied the prediction risk, instead of  the posterior properties of $\bbeta$, for high-dimensional regression with a horseshoe prior;
   and \cite{rovckova2018spike} established for high-dimensional linear regression the same posterior convergence rate as ours with a two-group Laplace prior, but failed to establish consistency of variable selection. }

The rest of this paper is organized as follows. 
Section \ref{main} presents the main theoretical results, where we lay down the theory of 
posterior consistency for high-dimensional linear regression with shrinkage priors. 
Section \ref{ext} studies posterior consistency for several commonly used shrinkage priors.
Section \ref{perform} discusses some important practical issues on Bayesian computation,
and illustrate the performance of Bayesian variable selection using a toy example.
Section \ref{simu} presents some simulation studies and a real data example.
Section \ref{diss} concludes the paper with a brief discussion. 
The Appendix gives the proofs of the main theorems.

\section{Main Theoretical Results}\label{main}

 {\it Notation.} In what follows, we rewrite the dimension $p$ of the model (\ref{lm}) 
 by $p_n$ to indicate that the number of covariates can increase with the sample size $n$. 
 We use superscript $^*$ to indicate true parameter values, e.g. $\bbeta^*$ and $\sigma^*$.
 For simplicity, we assume that the true standard deviation $\sigma^*$ 
 is unknown but fixed, and it doesn't change as $n$ grows.
 For vectors, we let $\|\cdot\|$ or $\|\cdot\|_2$ denote the $L_2$-norm; let $\|\cdot\|_1$ denote the $L_1$-norm; 
 let $\|\cdot\|_{\infty}$ denote the $L_\infty$ norm, i.e. the maximum absolute 
 value among all entries of the vector; and let $\|\cdot\|_0$ denote the $L_0$ norm, i.e. the number of non-zero entries. 
 As in (\ref{mixprior}), we let $\xi\subset\{1,2,\dots,p_n\}$ denote a subset model, and let $|\xi|$ denote 
 the size of the model $\xi$. 
 We let $s$ denote the size of the true model, i.e., $s=\|\bbeta^*\|_0=|\xi^*|$.
 We let $\bX_{\xi}$ denote the sub-design matrix corresponding to the model $\xi$, 
 and let $\lambda_{\max}(\cdot)$ and $\lambda_{\min}(\cdot)$ denote the largest
 and smallest eigenvalues of a square matrix, respectively.  
 We let $1(\cdot)$ denote the indicator function. 
 For two positive sequences $a$, and $b$, 
$a\prec b$ means $\lim a/b = 0$, $a\asymp b$ means 
$0<\liminf a/b\leq \limsup a/b<\infty$, and $a\preccurlyeq b$ means $a\prec b$ or $a\asymp b$. 
 We use $\{\epsilon_n\}$ to denote the Bayesian contraction rate which satisfies 
$\epsilon_n\prec 1$.

 \subsection{Posterior Consistency}

The posterior distribution for model (\ref{lm}) follows a general form: 
\[
 \pi(\bbeta,\sigma^2|D_n)\propto f(\bbeta,\sigma^2; D_n)\pi(\bbeta, \sigma^2),
\]
where $f(\bbeta,\sigma^2; D_n)\propto \sigma^{-n}\exp(-\|\by-\bX\bbeta\|^2/2\sigma^2)$ 
 is the likelihood function of the observed data $D_n=(\bX,\by)$, and 
 $\pi(\bbeta,\sigma^2)$ denotes the prior density of $\bbeta$ and $\sigma^2$. 
Consider a general shrinkage prior:
$\sigma^2$ is subject to an inverse-gamma prior $\sigma^2\sim\mbox{IG}(a_0,b_0)$,
 where $a_0$ and $b_0$ denote the prior-hyperparameters;
 and conditioned on $\sigma^2$, $\bbeta$ has independent prior for each entry, with an absolutely continuous 
 density function of the form 
\begin{equation}\label{prior0}
 \pi(\bbeta|\sigma^2) = \prod_j [g_\lambda(\beta_j/\sigma)/\sigma],
\end{equation}
where $\lambda$ is some tuning parameter(s). It is to easy to derive that
\begin{equation}\label{posterior}
 \begin{split}
 \log \pi(\bbeta,\sigma^2|D_n) =& C+\sum_{j=1}^{p_n}\log g_{\lambda}\left(\frac{\beta_j}{\sigma}\right)
   -(n/2+p_n/2+a_0+1)\log(\sigma^2)
  -\frac{2b_0+\|\by-\bX\bbeta\|^2}{2\sigma^2},
 \end{split}
\end{equation}
for some additive constant $C$.

The shape and scale of the pdf $g_\lambda$ play a crucial role for posterior consistency.
Intuitively, we may decompose the parameter space $\BR^{p_n}$ into three subsets:
neighborhood set $B_1=\{\|\bbeta-\bbeta^*\|\leq\epsilon_n\}$, ``overfitting'' set 
$B_2=\{\|\bX(\bbeta-\bbeta^*)-\bvarepsilon\|\lesssim \sigma^*\sqrt n\}\backslash B_1$
and the rest $B_3$. Heuristically, the likelihood $f(\bbeta_2)\gtrsim f(\bbeta_1) \gtrsim f(\bbeta_3)$ for any $\bbeta_i\in B_i$, $i=1,2,3$.
Therefore, to drive the posterior mass toward the set $B_1$, it is sufficient to require that
$\pi(B_1)\gg\pi(B_2)$ and the ratio $\pi(B_1)/\pi(B_3)$ is not too tiny. 
In other words, the prior distribution should 1) assign at least a minimum probability mass around $\bbeta^*$, and 2) assign a 
tiny probability mass on the overfitting set.
However, under the high dimensional setting, the ``overfitting'' set is geometrically intractable (and it expands to infinite)
due to the 
arbitrariness of the eigen-structure of the design matrix. Therefore, analytically, it is difficult to directly study the prior on the ``overfitting'' set.
One possible way to control the prior on the ``overfitting'' set is to impose a strong prior concentration for each $\beta_j$
such that the most of the prior mass is allocated on the ``less-complicated'' models under certain complexity measure.
Under regular identifiability conditions, the overfitting models are always complicated, so the prior probability mass on the ``overfitting'' models
should be small, but it is worth noting that the overfitting models are a subset of all complicated models and strong prior concentration 
is only a sufficient condition.
When the geometry of the overfitting set is easier to handle, e.g. under $p_n=o(n)$ or in the normal means models, the overfitting 
set can be a neighboring set of $\bbeta^*$,   potentially annulus-shaped. In this case, it is absolutely unnecessary
to require a strong prior concentration on the neighboring set of 
$\bbeta^*$. That is, we only need to impose conditions on the local shape of the prior around $\bbeta^*$, see \cite{Ghosal1999,ChenC2008,VanSV2017_2}.
This is also the key difference between high dimensional models and slowly increasing models/normal means models.

Before rigorously studying the properties of the posterior distribution, we first 
 state some regularity conditions on the eigen-structure of the design matrix 
 $\bX$:
\begin{enumerate}[label=$A_1$(\arabic*)]
 \item\label{covariate} All the covariates are uniformly bounded. For simplicity, we assume that 
  $\bx_j \in [-1,1]^n$ for $j=1,2,\ldots,p_n$, where 
  $\bx_j$ denotes the $j$-th column of $\bX$.
  
 \item\label{dim} The dimensionality is high: $p_n\succcurlyeq n$.
 \item\label{smalleig} There exist some integer $\bar p$ (depending on $n$ and $p_n$) and a fixed 
  constant $\lambda_0$ such that $\bar p\succ s$ and 
 $\lambda_{\min}(\bX_{\xi}^T\bX_{\xi})\geq n\lambda_0$ for any subset model $|\xi|\leq \bar p$.
\end{enumerate}

\noindent{\bf Remark:}
\ref{covariate} implies that $\lambda_{max} (\bX^T\bX)= \mbox{tr}(\bX^T\bX)\leq np$.
\ref{smalleig} has often been used in the literature to overcome the 
 non-identifiability issue of $\bbeta$, see e.g., \cite{NarisettyH2014,Zhang2010,SongL2014_2}.
  This condition is also equivalent to the lower bounded compatibility number condition used in \cite{CastilloSHV2015}.
In general, $\bar{p}$ should be much smaller than $n$. For example, for an $n\times n$-design matrix with all entries
iid distributed, the Marchenko-Pastur law states that the empirical distribution of the eigenvalues of 
the corresponding sample covariance matrix converges to 
$\mu(x) \propto \sqrt{(2-x)/x}1(x\in[0,2])$. The random matrix theory typically allows  $\bar p \asymp n/\log p_n$
 with a high probability 
when the rows of $\bX$ are independent isotropic sub-Gaussian random vectors, refer  
 to Lemma 6.1 of \cite{NarisettyH2014} and Theorem 5.39 of \cite{Vershynin2012}.

The next set of assumptions concern the sparsity of $\bbeta^*$ and the magnitude of nonzero entries of $\bbeta^*$.
\begin{enumerate}[label=$A_2$(\arabic*)]
 \item \label{maxsparse} $s\log p_n\prec n$, where $s$ is the size of the true model.
 \item \label{maxsig} $\max\{|\beta_j^*/\sigma^*|\}\leq\gamma_3 E_n$ for some fixed $\gamma_3\in(0,1)$, and $E_n$ is nondecreasing with respect to $n$. 
\end{enumerate}
\noindent{\bf Remark:} The condition \ref{maxsparse} is regularly used in the literature of 
 high dimensional statistics, which restricts the size of the true model to be of the order $o(n/\log p_n)$. 
The condition \ref{maxsig} constrains the growth of the nonzero true regression coefficients such that
$\max\{|\beta_j^*|\}\preccurlyeq E_n$.
\textcolor{black}{
 Together with the second condition in (\ref{priorconcentrationeq}), it ensures that the prior probability around the true model doesn't decay too fast, which echos the heuristics discussed in the previous paragraph that the shrinkage prior shall assign at least a minimum probability mass around $\bbeta^*$. Note that such an upper bound condition is fairly common in the literature of Bayesian asymptotics.  For example, \cite{GhosalGV2000} established a general posterior convergence rate, which requires that the prior mass over a small $f$-divergence ball of the true density $p_0$ is not too small. For linear regression models, Theorem 1 of \cite{ArmaganDL2013}, Theorem 3.1 of \cite{BhattacharyaPPD2015} and condition (7a) of \cite{YangWJ2015} imposed a similar upper bound condition on $\bbeta^*$. A similar condition has also been used in \cite{Jiang2007}, \cite{KleijnV2006} and \cite{GhosalV2007}. We note that it is also possible to establish posterior consistency without such an upper bound condition for certain types of shrinkage priors.  Noticeable examples include \cite{rovckova2018spike} which used a two-component mixture Laplace prior, \cite{CastilloSHV2015, GaoVZ2020} which used a Dirac-and-Laplace prior,  and \cite{MartinMW2015} which used  a $g$-prior centered at the least-square estimator. More discussions on this issue can be found after Corollary \ref{colheavy}.}

The next theorem provides sufficient conditions for posterior consistency. Hereafter, we let  
$\epsilon_n=M\sqrt{s\log p_n/n}$ denote the contraction rate, where $M$ is a fixed positive constant.

\begin{theorem}[Posterior Consistency] \label{thmmain} 
Consider the linear regression model (\ref{lm}), where the design matrix $\bX$ and the true $\bbeta^*$ satisfying conditions $A_1$ and  $A_2$, 
$\sigma^2$ is subject to an inverse-Gamma prior IG($a_0,b_0$), and the prior of $\bbeta$ is given by  (\ref{prior0}).
If $g_\lambda$ satisfies the conditions
\begin{equation} \label{priorconcentrationeq}
\begin{split}
  & 1-\int_{-a_n}^{a_n}g_\lambda(x)dx\leq p_n^{-(1+u)},\\
  & -\log \left(\inf_{x\in[-E_n,E_n]}g_\lambda(x)\right) =O(\log p_n),
  \end{split}
\end{equation}
where $u>0$ is a constant, $a_n\asymp \sqrt{s\log p_n/n}/p_n$, and the constant $M$ is sufficiently large, then the posterior consistency holds: 
\begin{equation}\label{postconsist12}
\begin{split}
&P^*\left(\pi(\|\bbeta-\bbeta^*\|\geq c_1\sigma^*\epsilon_n|D_n)\geq e^{-c_2n\epsilon_n^2}\right)\leq e^{-c_3n\epsilon_n^2}, \quad \mbox{and} \\
&P^*\left(\pi(\|\bbeta-\bbeta^*\|_1\geq c_1\sqrt{s}\epsilon_n\sigma^*|D_n)\geq e^{-c_2\blue{n\epsilon_n^2}}\right)\leq e^{-c_3\blue{n\epsilon_n^2}},
\end{split}
\end{equation}
for some positive constants $c_1$, $c_2$ and $c_3$.
\end{theorem}

The proof of this theorem is given in the Appendix. 
The results in (\ref{postconsist12}) imply that 
$\lim_{n\to \infty} E(\pi(\|\bbeta-\bbeta^*\|\geq c_1\sigma^*\epsilon_n|D_n)=0$ and 
$\lim_{n\to \infty} E(\pi(\|\bbeta-\bbeta^*\|_1\geq c_1\sigma^*\sqrt{s}\epsilon_n|D_n)=0$; that is, 
the $L_2$- and $L_1$-contraction rates of the posterior distribution of $\bbeta$ are $O(\sqrt{s\log p_n/n})$ and $O(s\sqrt{\log p_n/n})$, respectively.
These contraction rates are nearly optimal by recalling that the minimax $L_2$ rate is $O(\sqrt{s\log(p_n/s)/n})$ \cite{RaskuttiWY2011}, 
and they are no worse than the rates achieved with the spike-and-slab prior \cite{CastilloSHV2015}.
In other words, there is no performance loss due to the use of shrinkage priors.

 The conditions (\ref{priorconcentrationeq}) in the above theorem are consistent with our heuristic arguments in previous paragraphs.
 The first equation of (\ref{priorconcentrationeq}) concerns prior concentration, which requires that
 the prior density of $\beta_j/\sigma$ has a dominating peak inside a tiny interval $\pm a_n$. 
 Such a steep prior peak plays the role of ``spike'' as in spike-and-slab prior modeling. 
 In the literature, \cite{CastilloSHV2015} assigned on the spike a prior probability $\pi(\xi_j=1) = O(p_n^{-u})$ with $u>1$, 
 \cite{NarisettyH2014} employed an SSVS-type prior \cite{GeorgeM1993} under which 
 the prior probability $\pi(\xi_j = 1) \propto 1/p_n$, and \cite{YangWJ2015}
 assigned on the spike a prior probability $\pi(\xi_j=1) = O(p_n^{-u})$ with $u>0$.
 All these prior specifications are comparable to our condition $\pi(|\beta_j/\sigma|>a_n)=O(p_n^{-(1+u)})$ with $u>0$.
 Note that \cite{NarisettyH2014} and \cite{YangWJ2015} seem to require less prior concentration, they both
 imposed additional prior concentration conditions to bound the model size such that $\pi(|\xi|>O(n/\log p_n))=0$.
 \textcolor{black}{
 It worth noting that all our theorems 
 require the prior distribution to have a tiny scale by
 imposing a very small bound on $a_n$. The scale of the shrinkage prior affects the convergence rate of the posterior through its logarithm only. In other words, no matter how small the scale of the prior distribution is, it doesn't affect much the convergence rate of the posterior as long as $\log(1/a_n)$ is of order $\log(p_n)$. One established example is
 the horseshoe prior, see Theorem 3.3. of \cite{VanKV2014} for the convergence theory of the posterior.} The second equation of (\ref{priorconcentrationeq}), as discussed previously, essentially requires that the prior density around the true nonzero regression coefficient $\bbeta_j^*/\sigma^*$ is at least $\exp\{-O(\log p_n)\}$, i.e. $g_\lambda(\bbeta_j^*/\sigma^*)\geq\exp\{-c\log p_n\} $ for some positive constant $c$.
 Finally, we note that this prior concentration condition is only sufficient.
 In practice, a moderate degree of concentration can often lead to satisfactory results.
 
Other than the regression coefficients, similar results to (\ref{postconsist12}) can be derived for the fitting error $\|\bX\bbeta-\bX\bbeta^*\|$.


\begin{theorem} \label{predthem}
 If the conditions of Theorem \ref{thmmain} hold, then
 \begin{equation}\label{postconsistpred}
\begin{split}
&P^*\left(\pi(\|\bX\bbeta-\bX\bbeta^*\|\geq c_1\sigma^*\sqrt{n}\epsilon_n|D_n)\geq e^{-c_2n\epsilon_n^2}\right)\leq e^{-c_3n\epsilon_n^2},
\end{split}
\end{equation}
for some positive constants $c_1$, $c_2$ and $c_3$.
\end{theorem}

{\noindent \bf Remark}: Theorem \ref{predthem} actually holds without Condition $A_1$(3). To intuitively understand the redundancy of Condition $A_1(3)$, let us consider the fitted error under any selected subset model $\xi\supseteq\xi^*$, i.e., $\bX_{\xi}(\bX_\xi^T\bX_{\xi})^{-1}\bX_{\xi}^T\bvarepsilon$. Without any assumption on the eigen-structure of $\bX$, this term can be bounded in probability since the eigenvalues of $\bX_{\xi}(\bX_\xi^T\bX_{\xi})^{-1}\bX_{\xi}^T$ are 0 or 1. However, to prove Theorem \ref{thmmain}, we need to bound the estimation error $(\bX_\xi^T\bX_{\xi})^{-1}\bX_{\xi}^T\bvarepsilon$, hence an eigen-structure assumption such as Condition $A_1$(3) is necessary.


To conclude this subsection, we state that an appropriate shrinkage prior can lead to almost the same posterior consistency result as the spike-and-slab prior.

\subsection{Variable Selection Consistency} 

In this subsection, we perform a theoretical study on how to retrieve the sparse structure of $\bbeta^*$ with a shrinkage prior.
To achieve this goal, it is necessary to ``sparsify'' the continuous posterior distribution induced by the continuous prior.
In the literature, this is usually done by 1) hard (or adaptive) thresholding
on $\beta_j$ or on the shrinkage weight $1/(1+\lambda_j^2)$ \cite{LiP2017,TangXGG2016, CarvalhoPS2010, IshawaranR2005}, 
or 2) decoupling shrinkage and selection methods \cite{HahnC2015,XuSMQH2017}.
Note that the approaches in the latter class intend to incorporate the dependency between covariates
into the sparse posterior summary.
All the aforementioned approaches depend solely on the magnitude of the Bayesian estimates of $\beta_j$'s, without accounting for the degree of prior concentration. 

We propose to use a prior-dependent hard 
thresholding method, which sets
$\tilde \beta_j = \beta_j1(|\beta_j|>\eta_n)$ for some threshold $\eta_n$.
This induces a sparse pseudo posterior $\pi(\tilde\bbeta|D_n)$, which thereafter can be used
to assess the model uncertainty and conduct variable selection as if it was induced by a spike-and-slab prior. 
The correlation structure of $\pi(\tilde\bbeta|D_n)$ will reflect the dependency knowledge in $\bX$. 

First of all, Theorem \ref{thmmain} trivially implies that 
$E\pi(|\beta_j-\beta_j^*|\geq c_1\sigma^*\epsilon_n, \mbox{for all }j=1,\dots,p_n|D_n)=o_p(1)$. 
Therefore, if $\min_{j\in\xi^*}|\beta_j^*|>2c_1\sigma^*\epsilon_n$ and $\eta_n=c_1\sigma^*\epsilon_n$,
then $E\pi(1(\tilde\beta_j=0)\neq1(\beta_j^*=0)\mbox{ for all }j|D_n)=o_p(1)$ and 
 $\pi(\tilde\bbeta)$ can consistently select the true model.
However, one potential issue of using $c_1\sigma^*\epsilon_n$ for thresholding is that it greatly alters
the theoretical characteristic of $\pi(\bbeta|D_n)$
in the sense that the $L_2$-contraction rate of $\pi(\tilde\bbeta|D_n)$ can be as large as $s\sqrt{\log p_n/n}$
but not $\sqrt{s\log p_n/n}$.

This motivates us to consider another choice of $\eta_n$.
As discussed previously, (\ref{priorconcentrationeq}) implies a ``spike''  between [$-a_n$, $a_n$] for the prior of $\beta/\sigma$,
which plays the same role as  the Dirac measure in the spike-and-slab prior. Hence, from the point of view of prior specification,
$a_n$ distinguishes zero and nonzero coefficients, and it is natural to consider
$\tilde \beta_j = \beta_j1(|\beta_j/\sigma|>a_n)$. The posterior $\pi(\tilde \bbeta,\sigma^2|D_n)$ thus
implies the selection rule as $\xi(\bbeta,\sigma^2) = \{j; |\beta_j/\sigma|>a_n\}$.
\textcolor{black}{
This hard-thresholding rule of Bayesian variable selection can be viewed as a counterpart of the selection rule $\{j: |\beta_j/\sigma|>0\}$ used in spike-and-slab modeling. It is also closely related with the idea of ``generalization dimension'' \cite{BhattacharyaPPD2015, rovckova2018spike}. Theorem 3.4 of \cite{BhattacharyaPPD2015} defines  $supp_{\delta}(\bbeta)=\{j: |\beta_j/\sigma|\geq\delta\}$ as the set of variables selected based on a nonsparse posterior sample $\bbeta$, where $\sigma=1$ is known, $p_n=n$ ($\bX=I$), and $\delta$ satisfies the condition $\pi(|\beta_j|\geq\delta)\leq C\log(n/s)/\Gamma(n^{-1-u})\asymp \log(n/s)/(n^{1+u})$ for some $u>0$.
This choice of $\delta$ matches our threshold $a_n$, which is the quantile of the prior distribution satisfying $\pi( |\beta_j/\sigma|\geq a_n)\leq p_n^{-1-u}$ for some $u>0$.
}

The following theorem establishes 
variable selection consistency of the above hard-thresholding rule, \textcolor{black}{while \cite{BhattacharyaPPD2015, rovckova2018spike} proved only that the selected model has a bounded size.}

\begin{theorem}\label{variableselection} (Variable selection consistency) 
 Suppose that the conditions of Theorem \ref{thmmain} hold under $a_n\prec$ $\sqrt{\log p_n}/(\sqrt{n}p_n)$ and
 $u>1$. Let $l_n$ be a measure of flatness of the function $g_\lambda(\cdot)$,
\[l_n=\max_{j\in\xi^*}\sup_{\substack{x_1,x_2\in \beta_j^*/\sigma^*\pm c_0\epsilon_n\\|x_1|,|x_2|\geq a_n}}\frac{g_\lambda(x_1)}{g_\lambda(x_2)}\]
where $c_0$ is some large constant.
 If $\min_{j\in\xi^*}|\beta_j^*|> M_1\sqrt{\log p_n/n}$ for some sufficiently large $M_1$ and $s\log l_n\prec\log p_n$, then 
\begin{equation}\label{postconsistvs}
P^*\{\pi[\xi(\bbeta,\sigma^2)=\xi^*|D_n]>1-o(1)\}>1-o(1).
\end{equation}
\end{theorem}

 This theorem is a simple corollary of Theorem \ref{vs} in the Appendix. 
 It requires a smaller value of $a_n$ and a larger  value of $u$, i.e. a narrower and more concentrated prior peak, compared to Theorem \ref{thmmain}.
 Besides the prior concentration and tail thickness, the condition $s\log l_n\prec\log p_n$ also requires tail flatness
 such that the prior density around the true value $\beta^*/\sigma^*$ is not rugged.
 This flatness facilitates an analytic study for the posterior $\pi(\xi(\bbeta,\sigma^2)|D_n)$.
 Generally speaking, for smooth $g_\lambda$, the flatness measure approximately follows 
 $\log l_n \asymp \max_{j\in\xi^*}\epsilon_n[\log g_\lambda]'(\beta_j^*/\sigma^*)\rightarrow 0$,
 where $[\log g_\lambda]'$ is the first derivative of $\log g_\lambda$.
 In the extreme situation, we can utilize an exactly flat tail such that $\log l_n\equiv 0$. 
 An example could be $g_\lambda(x) \propto \exp\{-p_\lambda(x)\}1_{x\in[-E_n,E_n]}$,
 where $p_\lambda(x)$ has a shape like a non-concave penalty function such as SCAD.
 If $\log l_n$ is not exactly 0, then the condition $s\log l_n\prec\log p_n$ imposes an additional
 constraint on the sparsity $s$ other than $s\prec n/\log p_n$. More
 discussions on $l_n$ can be found in Section \ref{ext}.
 
 The result of this theorem also implies a stronger posterior contraction for the false covariates
 such that $|\beta_j/\sigma|$ is bounded in posterior by $a_n$. 

\subsection{Shape Approximation of the Posterior Distribution}

Another important aspect of Bayesian asymptotics is the shape of the posterior distribution.
The general theory on the posterior shape is the Bernstein von-Mises (BvM) theorem.
It claims that the posterior distribution 
of the parameter $\theta$ in a regular finite dimensional model is 
approximately a normal distribution as $n\rightarrow\infty$, i.e.,
\begin{equation}\label{bvm}
||\pi(\cdot|D_n)-N(\cdot; \hat\theta_{\tiny\mbox{MLE}}, (n\hat I)^{-1})||_{TV}\rightarrow 0,
\end{equation}
regardless of the choice of the prior $\pi(\theta)$, where $\pi(\cdot|D_n)$ is the posterior distribution given data $D_n$,
$N(\cdot; \mu, \Sigma)$ denotes a (multivariate) normal distribution, 
$\hat\theta_{\tiny\mbox{MLE}}$ stands for the maximum likelihood estimator of $\theta$, 
 $I$ is Fisher's information matrix, and $||\cdot||_{TV}$ denotes
the total variation distance between two measures. 
The BvM theorem provides an important link between the frequentist limiting distribution and 
the posterior distribution, and it can be viewed as a frequentist justification for Bayesian 
credible intervals. To be specific, the Bayesian credible intervals are asymptotically equivalent to
the Wald confidence intervals, and also have the long-run relative frequency interpretation.

The BvM theorem generally holds for fixed dimensional problems.  
For linear regression with known $\sigma^*$, the posterior normality always holds under (improper) uniform prior, as 
$\pi(\bbeta|D_n ) \sim N( (\bX^T\bX)^{-1}\bX^T\by$, $ \sigma^{*2}(\bX^T\bX)^{-1})$,
as long as $p\leq n$ and the matrix $\bX$ is of full rank.

Under the scenario $p_n \gg n$, 
 all false coefficients are bounded in posterior by a threshold value by Theorem \ref{variableselection}.
Combining this with the fact that 
 $f(\bbeta_{\xi^*}, \bbeta_{(\xi^*)^c};\bX,\by)$ $\approx f(\bbeta_{\xi^*},\bbeta_{(\xi^*)^c}=0;\bX_{\xi^*},\by)$ when $\|\bbeta_{(\xi^*)^c}\|_\infty$ is sufficiently small, 
we have that
\begin{equation*}
\begin{split}
\pi(\bbeta|D_n) &\propto  L(\bbeta_{\xi^*},\bbeta_{(\xi^*)^c}|\bX,\by)\pi(\bbeta_{\xi^*}, \bbeta_{(\xi^*)^c})
\approx L(\bbeta_{\xi^*};\bX_{\xi^*},\by)\pi(\bbeta_{\xi^*})\pi(\bbeta_{(\xi^*)^c}).
\end{split}
\end{equation*}
If $\pi(\bbeta_{\xi^*})$ is sufficiently flat around $\bbeta^*_{\xi^*}$ and acts like a uniform prior, then the low dimensional term $L(\bbeta_{\xi^*};\bX_{\xi^*},\by)$ $\pi(\bbeta_{\xi^*})$ follows
a normal BvM approximation. More rigorously, we have the next theorem.

\begin{theorem}\label{bvmthm} (Shape Approximation) 
Assume the conditions of Theorem \ref{variableselection} hold, $\lim s\log l_n = 0$,
and $\blue{a_n\prec(1/p_n)}\sqrt{1/(ns\log p_n)}$.
Let $\btheta = (\bbeta_{\xi^*},\sigma^2)^T$, then with dominating probability,
\begin{equation}\label{bvm2}
\begin{split}
&\pi(\bbeta,\sigma^2|D_n) \mbox{ converges in total variation to}\\
&\phi(\bbeta_{\xi^*};\hat\bbeta_{\xi^*}, \sigma^2(\bX_{{\xi^*}}^T\bX_{{\xi^*}})^{-1})\prod_{j\notin\xi^*} \pi(\beta_j|\sigma^2)
ig\left(\sigma^2, \frac{n-s}{2}, \frac{\hat\sigma^2(n-s)}{2}\right),
\end{split}
\end{equation}
where $\phi(\cdot)$ is a multivariate normal density function,
$ig(\cdot)$ is an inverse-gamma density function,
 $\pi(\beta_j|\sigma^2)$ is the conditional prior distribution of $\beta_j$, and 
$\hat\bbeta_{\xi^*}$ and $\hat\sigma^2$ are, respectively, the maximum likelihood estimates (MLEs) of $\bbeta_{\xi^*}$ and $\sigma^2$ given data $(\by,\bX_{\xi^*})$.
\end{theorem}

 Refer to Theorem \ref{BvM} for the proof of this theorem. 
 Its condition is slightly stronger than that of Theorem \ref{variableselection}. 
 It requires that $a_n$ is smaller and the prior log-density $\log g_\lambda(\cdot)$ 
 is almost constant 
 around the true value of $\beta_j^*/\sigma^*$. 
The following corollary can be easily derived from the above theorem.
\begin{corollary}\label{bvmthm2}
Under the condition of Theorem \ref{bvmthm}, 
for any $j\in\xi^*$, the marginal posterior of $\beta_j$
converges to normal distribution $\phi(\beta_j, \hat\beta_j, \sigma^{*2}\sigma_j)$,
where $\hat\beta_j$ is the $j$th entry of $\hat\bbeta_{\xi^*}$,
$\sigma_j=[(\bX_{{\xi^*}}^T\bX_{{\xi^*}})^{-1}]_{j,j}$.
Furthermore, if $s\prec \sqrt n$,
the posterior $\pi(\bbeta_{\xi^{*c}},\bbeta_{\xi^*},\sigma^2|D_n)$ converges in total variation to
\[
\prod_{j\notin\xi^*} \pi(\beta_j|\sigma^2)
\phi\left(\btheta ;\hat\btheta, (n\hat I)\right), \mbox{ with probability approaching 1,}
\]
where $\btheta=(\bbeta_{\xi^*},\sigma^2)^T$,
$\hat\btheta=(\hat\bbeta_{\xi^*}$, $\hat\sigma^2)^T$, and 
$(n\hat I)^{-1}=\mbox{diag}\left(\hat\sigma^2(\bX_{{\xi^*}}^T\bX_{{\xi^*}})^{-1}, 2\hat\sigma^4/n
 \right)$.
 In other words, the BvM theorem holds for the parameter component $(\bbeta_{\xi^*},\sigma^2)$.
\end{corollary}

Theorem \ref{bvmthm} is comparable to the result developed under the spike-and-slab prior \cite{CastilloSHV2015}. Under the spike-and-slab prior, the posterior density of $\bbeta$ 
 can be rewritten as a mixture,
\begin{equation}\label{bvmps0}
  \pi(\bbeta|D_n) =\sum_{\xi\subset{\{1,\dots,p\}}} \pi(\xi|D_n) \pi(\bbeta_{\xi}|\bX_\xi,\by)
  1\{\bbeta_{\xi^c}=0\},
\end{equation}
where $\pi(\bbeta_{\xi}|\bX_\xi,\by)\propto h_{\xi}(\bbeta_{\xi})f(\bbeta_{\xi};\bX_{\xi},\by)$, and  
$h_{|\xi|}(\cdot)$ is defined in (\ref{mixprior}). If
$\pi(\xi^*|D_n)\rightarrow 1$, $\pi(\bbeta|D_n)$ converges to $\pi(\bbeta_{\xi^*}|\bX_{\xi^*},\by)
1\{\bbeta_{{\xi^*}^c}=0\}$. 
Furthermore, if $\pi(\bbeta_{\xi^*})$ is sufficiently flat and BvM holds for the  low-dimensional term $\pi(\bbeta_{\xi^*}|\bX_{\xi^*},\by)$, then
it leads to a posterior normal approximation as
\begin{equation}\label{bvmps}
\pi(\bbeta|D_m)\approx\mbox{N}(\bbeta_{\xi^*}; \hat\bbeta_{\xi^*}, (\bX_{\xi^*}^T\bX_{\xi^*})^{-1})\otimes \delta_{0}(\bbeta_{(\xi^*)^{c}}),
\end{equation}
where $\otimes$ denotes product of measure.
   


 Theorem \ref{bvmthm} and Corollary \ref{bvmthm2} extend the BvM-type result from the spike-and-slab 
 prior to the shrinkage prior. They show that the 
 marginal posterior distribution for the true covariates follows 
 the BvM theorem as if under the low dimensional setting, 
 while the marginal posterior for the false covariates can be approximated by its
 prior distribution. Since the prior distribution is already highly concentrated,
 the posterior of the false covariates being almost the same as the prior does not contradict our contraction results.
 Note that Bayesian procedure can be viewed as a process of updating the probabilistic knowledge
 of parameters. The concentrated prior distribution reflects our prior belief that almost all 
 the predictors are inactive, and (\ref{bvm2}) can be interpreted as that
the Bayesian procedure correctly identifies the true model $\xi^*$ and updates the distribution of $\bbeta_{\xi^*}$ using the data, but it obtains no evidence to support $\beta_j\neq 0$ for any $j\notin \xi^*$ and thus doesn't update their concentrated prior distributions.

Let $CI_i(\alpha)$ denote the posterior quantile credible interval of the $i$th covariate. 
If $\pi(\beta|\sigma^2)$ is a symmetric distribution, 
then Corollary \ref{bvmthm2} implies that
\begin{equation}\label{ci}
\begin{split}
&\lim P^*(\beta_i^*\in CI_i(\alpha) ) = 1-\alpha,\, \mbox{if $i\in\xi^*$  and }\\
&\lim P^*( 0 \in CI_i(\alpha)) =1,\, \mbox{if $i\notin\xi^*$, }
\end{split}
\end{equation}
 for any $1>\alpha>0$. This result implies that for the false covariates, the Bayesian credible interval  
 is super-efficient:  Asymptotically, it can be very narrow (as the prior is highly concentrated), 
  but has always 100\% probability coverage. This is much different from the confidence interval. 
 
It is important to note that both Theorem \ref{bvmthm} and its counterpart (\ref{bvmps}) rely on selection consistency (and beta-min condition),  which drives Bayesian post-selection inference. Therefore, the frequentist coverage of the 
Bayesian credible interval (first equation of (\ref{ci})) does not hold uniformly for all nonzero $\beta_i$ values, but
only hold for those bounded away from 0. If the beta-min condition is violated, one can 
rewrite the posterior with the shrinkage prior as a mixture distribution similar to (\ref{bvmps0}).
Hence, the corresponding posterior inference will be model-average-based.

The above asymptotic studies are completely different from the frequentist 
sampling distribution-based inference tools such as de-biased Lasso \cite{VanBRD2014,ZhangZ2014}.
The de-biased Lasso method established asymptotic normality as
 \begin{equation}\label{dlasso2}
\sqrt{n}(\hat\bbeta-\bbeta^*)\stackrel{d}{\rightarrow}\mbox{N}(0, \sigma^{*2}S\bX^T\bX S^{T}/n),
\end{equation}
for any $\bbeta^*$, even when it is arbitrary closed to zero, and
 $S$ is some surrogate inverse matrix of the sample covariance.
Different from our posterior consistency result, the asymptotic distribution in the right hand 
side of (\ref{dlasso2}) is a divergent distribution when $p_n\gg n$.

\textcolor{black}{In the literature, there is a different line of researches about the validity of Bayesian credible intervals, which do not require selection consistency, see e.g. \cite{VanSV2017_2, belitser2020empirical}. These works are usually based on the first-order Bayesian convergence rate only. As a consequence, these credible intervals/balls involve an unknown multiplicative constants (e.g., $c_1$ and $M$) that appear in the posterior convergence rate (\ref{postconsist12}) and their coverage always converges to 1, rather than the nominal level $1-\alpha$.}



  We conjecture that if consistent point estimation and 
  inference of credible intervals are made simultaneously, the credible intervals will 
  be super-efficient for the false covariates due to 
  the sparsity constraints (i.e. the prior distribution) imposed on the regression coefficients. 
  These constraints ensure  posterior consistency and thus reduces the 
  variability of the coefficients of the false covariates. Based on this understanding, 
  it seems that under the framework  
  of consistent high-dimensional Bayesian analysis, 
   a separate post-selection inference procedure (without sparsity constraints) 
  is necessary to induce the correct second-order inference. For example, it can be done 
  in a sequential manner (refer the idea to \cite{LockhartTTT2014} and \cite{TibshiraniTLT2014}): Attempting to add each of the unselected variables 
  to the selected model, and calculating the corresponding credible interval 
  for the unselected variable. 

\section{Consistent Shrinkage Priors}\label{ext}

In the previous section, we establish general theory for shrinkage priors based on
abstract conditions. In this section, we will apply the theory to several types of 
shrinkage priors, and study the corresponding posterior asymptotics.

 
 The condition (\ref{priorconcentrationeq}) requires certain balance between
 prior concentration and tail thickness. First of all, it is easy to see that
 the Laplace prior fails to satisfy condition (\ref{priorconcentrationeq}) unless the tuning
 parameter  $\lambda_n\sim p_n\log p_n/\sqrt{\frac{s\log p_n}{n}}$ and 
 the true coefficients are as tiny as $|\beta_j^*|=O(\sqrt{s\log p_n/n}/p_n)$ for all $j\in \xi^*$.
Therefore, we first consider the prior specification that has a heavier tail than the exponential distribution.

\subsection{Polynomial-tailed Distributions}

We assume that the prior density of $\bbeta$ has the form  
$\pi(\bbeta|\sigma^2)=\prod_{i=1}^{p_n} \frac{1}{\lambda_n\sigma}g(\beta_i/\lambda_n\sigma)$, 
where $\lambda_n$ is a scale hyperparameter, and the density $g(\cdot)$ is symmetric and polynomial-tailed, i.e.
$g(x)\asymp x^{-r} $ as $|x|\rightarrow \infty$ for some positive $r>1$. 
Under the above prior specification, we adapt  Theorem \ref{thmmain} as follows:

\begin{theorem}\label{betaprior}
Assume conditions $A_1$ and $A_2$ hold for the linear regression model, and 
a polynomial-tailed prior is used. 
If $\log(E_n)=O(\log p_n)$, and the scale parameter $\lambda_n$ satisfies
 $\lambda_n\leq a_np_n^{-(u+1)/(r-1)}$ and  $-\log\lambda_n = O(\log p_n)$
 for some $u>0$,  then 
\begin{itemize}
 \item the posterior consistency (\ref{postconsist12})
holds when $a_n\asymp \sqrt{s\log p_n/n}/p_n$;
 \item the model selection consistency (\ref{postconsistvs}) holds when
 $a_n\prec \sqrt{\log p_n}/\sqrt{n}p_n$, 
 $\min_{j\in\xi^*}|\beta_j^*|\geq M_1\sqrt{\log p_n/n}$ for sufficiently large $M_1$,
 $s\log l_n\prec\log p_n$ and $u>1$;
 \item the posterior approximation (\ref{bvm2}) holds if
$a_n\prec \sqrt{1/(ns\log p_n)}/p_n$,
 $\min_{j\in\xi^*}|\beta_j^*|\geq M_1\sqrt{\log p_n/n}$ for sufficiently large $M_1$, $s\log l_n\prec 1$, and $u>1$.
\end{itemize}
\end{theorem}

Note that polynomially decaying distributions that we most commonly used satisfy
$g(x) = Cx^{-r}L(x)$, where $\lim_x L(x) =1$ with the rate
\begin{equation}
 |L(x)- 1| = O( x^{-t}) \mbox{ for some }t\geq 0. \label{L}
\end{equation}
It is not difficult to see that if $\min_{j\in\xi^*}|\beta_j^*|>M_2\epsilon_n$ for some large $M_2$,
$\lambda_n=O(\epsilon_n)$, then
$s\log l_n \asymp s\epsilon_n/\min_{j\in\xi^*}|\beta_j^*|. $
Therefore, Theorem \ref{betaprior} can be refined as
follows.


\begin{corollary}\label{colheavy}
 Consider the polynomial-tailed prior distributions satisfying (\ref{L}). Assume condition $A_1$
 holds,  $s\log p_n\prec n$, and $\log(\max_{j\in\xi^*}|\beta_j^*|) = O(\log p_n)$.
 Let the choice of $\lambda_n$ satisfy $-\log \lambda_n = O(\log p_n)$, then
\begin{itemize}
 \item If $\lambda_n = O\{\sqrt{s\log p_n/n}\big/p_n^{(u+r)/(r-1)}\}$ with $u>0$,
 then posterior consistency holds with a nearly optimal contraction rate;
 \item If $s\sqrt{s\log p_n/n}/\min_{j\in\xi^*}|\beta_j^*|\prec\log p_n$, $\lambda_n\prec\sqrt{\log p_n/ n}\big/p^{(u+r)/(r-1)}$ with $u>1$, 
 $\min_{j\in\xi^*}|\beta_j^*|\geq M_1\sqrt{s\log p_n/n}$ for sufficiently large $M_1$, then the variable selection consistency holds;
 \item If $s\sqrt{s\log p_n/n}/\min_{j\in\xi^*}|\beta_j^*|\prec 1$, 
  $\lambda_n\prec \sqrt{1/n\log p_n}\big/p_n^{(u+r)/(r-1)}$ with $u>1$, and $\min_{j\in\xi^*}|\beta_j^*|\geq M_1\sqrt{s\log p_n/n}$ for sufficiently large $M_1$, 
 then the posterior shape approximation holds.
\end{itemize}
\end{corollary}

Theorem \ref{betaprior} and Corollary \ref{colheavy} show that a nearly optimal contraction rate
can be achieved for high-dimensional linear regression by adopting a polynomial-tailed prior with an appropriate value of $\lambda_n$.
As suggested by Corollary  \ref{colheavy}, it is sufficient to choose the scale parameter as
$\log\lambda_n\sim -c \log p_n$ for some $c\ll (u+r)/(r-1)$, since $n=O(p_n)$ and $s=o(p_n)$.
Compared to the choice $\lambda_n = (s/p)\sqrt{\log(p/s)}$ under normal means models \cite{VanKV2014,GhoshC2014},
we note that a stronger prior concentration is required for regression models.
Our results allow the maximum magnitude of nonzero coefficients to increase up to a polynomial
of $p_n$. In contrast, the DL prior allows  $|\beta_j^*|$ to increase with a logarithmic order of $n$ only \cite{BhattacharyaPPD2015}.
It is worth to note that the bounded condition on $|\beta_j^*|$ is not necessary when a polynomially decaying prior under normal means models, i.e., when $\bX=I$ \cite{VanKV2014,GhoshC2014,song2020bayesian}. However, under general regression settings,
such condition may be necessary due to the dependency among covariates.
One should also notice that selection consistency or posterior normality requires stronger beta-min condition (minimal $\beta^*$ is greater than the order of $\sqrt{s\log p_n/n}$) and an 
additional condition on the true sparsity $s$ (e.g., if $\min_{j\in\xi^*}|\beta_j^*|>C$ for some
constant $C$, selection consistency and posterior normality require $s^3\prec n\log^2p_n$ and $s^3\prec n/\log p_n$ respectively). 
The reason we need such unpleasant conditions is that the polynomially decaying prior 
modeling utilizes only one scale hyperparameter.
Although this simplifies the modeling part, 
we lose control on the shape or tail flatness of the prior distribution.
If we utilize both scale and shape hyperparameters in  prior modeling, the conditions can be improved, as seen in  Section \ref{ext3}.

For the convenience of posterior sampling, one way to construct polynomially decaying 
prior is to design a hierarchical scale mixture Gaussian distribution as 
\begin{equation}\label{prior4th}
 \beta_j\sim \mbox{N}(0, \lambda_j^2\sigma^2), \quad \lambda_j^2 \sim \pi_{s_n}(\lambda_j^2),\quad
 \mbox{independently for all }j,
\end{equation}
where $s_n$ is the scale hyperparameter of the mixing distribution $\pi_{s_n}(\cdot)$, i.e., 
$\pi_{s_n}(\cdot) = \pi_1(\cdot/s_n)/s_n$. Equivalently,
 $\sqrt{s_n}$ is the scale parameter of the marginal prior of $\beta_j$. The scale mixture
Gaussian distribution can also be viewed as a local-global shrinkage prior, where $\lambda_j^2$'s 
are local shrinkage parameters, and $s_n$ is a deterministic global shrinkage parameter. 
As shown in the next lemma, the tail behavior of the marginal distribution of $\beta_j$ is determined
by the tail behavior of $\pi_1$.

\begin{lemma}
If the mixing distribution $\pi_{s_n}(\cdot)$ is a polynomial-tailed distribution satisfying
$\pi_1(\lambda^2) = C\lambda^{-2\tilde r}\tilde L(\lambda^2) $ 
 and $|\tilde L(\lambda^2)-1|=O((\lambda^2)^{-\tilde t})$, 
 then the marginal prior distribution of $\beta_j$ induced by (\ref{prior4th}) is 
 polynomial-tailed with order $2\tilde r-1$ and satisfies $|L(x)-1|=O(x^{-2\tilde t})$, 
 where $L$ is defined in (\ref{L})).
\end{lemma}

The proof of this lemma is trivial and hence omitted in this paper.


 Combining the above lemma and  Corollary \ref{colheavy}, 
 it is sufficient to assign $\lambda_j^2$ a polynomial-tailed distribution and properly choose the 
 scale parameter $s_n$ such that $\sqrt{s_n}$ is decreasing and satisfies the conditions in Corollary \ref{colheavy}.
 \cite{GhoshC2014} studied the posterior convergence of the normal means models with 
 a scale mixture Gaussian prior (\ref{prior4th})
  and achieved  a minimax contraction rate.
  However, their result is only applicable to the case that the polynomial order $\tilde r$ of 
 $\pi_1(\lambda_j^2)$ is between 1.5 and 2. 
 Our result is more general 
 and valid for any $\tilde r>1$.

 In what follow, we list some examples of polynomially decaying prior distributions which can be represented 
 as a scale mixture Gaussian. All these priors satisfy condition (\ref{L}):

\begin{itemize}
 \item Student's $t$-distribution,  for which the mixing distribution of $\lambda^2$ is inverse gamma $\mbox{IG}(a_1, s_n)$ with $a_1>0$.

\item Normal-exponential-gamma (NEG) distribution \cite{GriffinB2011}, for which the mixing distribution is
$\pi(\lambda^2)=\nu s_n^{-1}(1+s_n^{-1}\lambda^2)^{-\nu-1}$ with $\nu>0$.

 \item Generalized double Pareto distribution \cite{ArmaganDL2013} with the density 
   $g(x) = (2\lambda_n)^{-1}$ $(1+|x|/(a_1\lambda_n))^{-(a_1+1)}$, for which the mixing distribution can be represented 
   as a gamma mixture of exponential distributions with $a_1>0$. 

 \item Generalized Beta mixture of Gaussian distributions \cite{ArmaganCD2011}, 
  for which the mixing distribution is inverted Beta: 
    $\lambda_j^2/s_n\sim \mbox{Inverted Beta}(a_1,b_1)$ with $a_1>0$.
    Note that the horseshoe prior is a special case of  generalized Beta mixture Gaussian distributions 
    with $a_1=b_1 = 1/2$.
\end{itemize}

In addition, Theorem \ref{betaprior} implies a simple way to remedy the inconsistency of Bayesian Lasso
by imposing a heavy tail prior on the hyperparameter:
 $\beta/\sigma\sim \mbox{DE}(\lambda_j)$, $\lambda_j^{-1}\sim \pi_{s_n}$, 
where $\mbox{DE}(\lambda)$ denotes the double exponential distribution $\lambda\exp\{-\lambda x\}/2$, and the mixing distribution 
$\pi_{s_n}$ of $\lambda_j^{-1}$ has a polynomial tail with the scale parameter $s_n$.

In the above analysis, we choose the scale parameters $\lambda_n$ or $s_n$ to decrease deterministically 
as $n$ increases. Hence, in practice, certain tuning procedures are recommended as described 
in Section \ref{perform}. 
Such hyperparameter tuning occurs in most Bayesian procedures under the spike-and-slab prior as well. \textcolor{black}{Note the such a tuning procedure usually requires multiple simulations under different levels of $\lambda_n$.}
In the literature, an adaptive Bayesian way to choose $\lambda_n$ is
to assign a hyper-prior on $\lambda_n$. \cite{VanSV2017} studied the horseshoe prior for the
normal means models, and they showed that the posterior consistency remains if $\lambda_n$ is subject to
a hyper-prior which is truncated on $[1/n,1]$.
However, the results derived for normal means models may not be trivially applicable to regression models.
Note there is a $\sqrt n$ difference between  regression models and normal means models,
in terms of $L_2$-norm for the columns in the design matrix. The
result of \cite{VanSV2017} suggests to truncate the prior of $\lambda_n$
on [$n^{-3/2}, n^{-1/2}$] for regression models.
A toy example shown in Figure \ref{shoe} indicates that such truncation still leads to many false discoveries.
\textcolor{black}{Another popular choice is to impose the global shrinkage parameter a half Cauchy prior $\lambda_n\sim \mathcal C^+(0,1)$. 
Simulation studies have been conducted with this hierarchical prior in the supplementary material. 
The numerical results show that this hierarchical prior leads to insufficient prior shrinkage and less accurate posterior concentration.}
Finally, our posterior shape approximation result relies on the fact that 
$\beta_j$'s are {\it a priori} independent  conditioned on $\sigma^2$. If a hyper-prior on $\lambda_n$ is used, then the conditional {\it a priori} independence does not hold any more, and the BvM result (\ref{bvm2}) fails.

\subsection{Two-Component Mixture Gaussian Distributions}\label{ext3}

Another prior that has been widely used in Bayesian linear regression analysis is the two-component 
 mixture Gaussian distribution, see e.g., \cite{GeorgeM1993} and \cite{NarisettyH2014}:
 \begin{equation}\label{prior5th}
  \beta_j/\sigma \sim (1-\xi_j)\mbox{N}(0,\sigma_0^2)+\xi_j\mbox{N}(0,\sigma_1^2),
\quad \xi_j\sim\mbox{Bernoulli}(m_1).
\end{equation}
The component $\mbox{N}(0,\sigma_0^2)$ has a very small $\sigma_0$ and can be viewed as an approximation 
 to the point mass at 0. In the literature, 
 the interest in this prior has been focused only 
 on the consistency of variable selection, i.e., $\pi(\{j: \xi_j=1\}=\xi^*|D_n)$.
 Here, we treat it as an absolutely continuous prior and study the posterior properties of $\bbeta$ 
 in the next theorem.

\begin{theorem}\label{thmssvs}
 Suppose that the two-component mixture Gaussian prior (\ref{prior5th}) is used for the high-dimensional 
 linear regression model (\ref{lm}), and that the following conditions hold: condition $A_1$, condition $A_2$, 
 $E_n^2/\sigma^2_1+\log \sigma_1\asymp \log p_n$, 
 $m_1=1/p_n^{1+u}$ and $\sigma_0\leq a_n/\sqrt{2(1+u)\log p_n}$
for some $u>0$.  Then
\begin{itemize}
 \item the posterior consistency (\ref{postconsist12}) holds when $a_n\asymp\sqrt{s\log p_n/n}/p_n$;
 \item the model selection consistency (\ref{postconsistvs}) holds when
 $a_n\prec \sqrt{\log p_n}/\sqrt{n}p_n$, 
  $sE_n\sqrt{s\log p_n/n}/\sigma_1^2\prec\log p_n$,
 $\min_{j\in\xi^*}|\beta_j^*|\geq M_1\sqrt{\log p_n/n}$ for sufficiently large $M_1$
 and $u>1$;
 \item the posterior approximation (\ref{bvm2}) holds if
$a_n\prec \sqrt{1/(ns\log p_n)}/p_n$, $sE_n\sqrt{s\log p_n/n}/\sigma_1^2\prec 1$,
 $\min_{j\in\xi^*}|\beta_j^*|$ $\geq M_1\sqrt{\log p_n/n}$ for sufficiently large $M_1$ and $u>1$.
\end{itemize}
\end{theorem}

The two normal mixture distribution contains three hyperparameters $m$, $\sigma_0^2$ and $\sigma_1^2$.
Hence, we have more control on the prior shape compared to the polynomially decaying priors,
and the theoretic properties are improved slightly comparing to Corollary \ref{colheavy}.
Specifically, Theorem \ref{thmssvs} allows us to choose $\sigma_1=E_n=p_n^c$ for some $c>1$
and thus 
 $sE_n\sqrt{s\log p_n/n}/\sigma_1^2\prec 1$ always holds, i.e, there will be no additional conditions on the upper bound of the model size $s$; and  Theorem \ref{thmssvs} only requires that $\min_{j\in\xi^*}|\beta^*_j|$ is larger than the order of $\sqrt{\log p/n}$.



\section{Bayesian Computation and an Illustrative Example}\label{perform}

In this section, we will first discuss some important practical issues, including 
posterior computation, model selection and hyperparameter tuning, and then
 we will use some toy examples to illustrate the performance of the shrinkage priors.
 For convenience, we will call the Bayesian method, whose consistency is guaranteed by Theorem \ref{thmmain} with a shrinkage prior, a Bayesian consistent shrinkage (BCS) method in what follows. In particular, we will use the student-$t$ prior, as an example of the shrinkage 
 prior, and compare it with the Laplace prior. 

 The scale mixture Gaussian priors (\ref{prior4th}), under a proper hierarchical representation,
 usually lead to posterior conjugate Gibbs updates. For example, for the student-$t$ prior, 
 the posterior distribution can be updated in the following way: 
\begin{equation}\label{gibbs}
\begin{split}
\sigma^2|\bbeta, \lambda_1,\ldots,\lambda_{p_n} &\sim \mbox{IG}(a_0+\frac{n+p_n}{2}, b_0+\frac{\|\by-\bX\bbeta\|^2}{2}+\sum_j\frac{\beta_j^2}{2\lambda_j^2}),\\
\bbeta |\sigma^2,\lambda_1,\ldots,\lambda_{p_n}   &\sim \mbox{N}(K^{-1}\bX^T\by/\sigma^2, K^{-1}),\\ 
f(\lambda_j^2|\bbeta,\sigma^2) &\propto \frac{1}{\lambda_j}\exp\left\{-\frac{\beta_j^2}{2\lambda_j^2\sigma^2}\right\}\pi(\lambda_j^2), \quad j=1,\ldots, p_n,
\end{split}
 \end{equation}
where $K=(\bX^T\bX+\Lambda)/\sigma^2$, $\Lambda = \mbox{diag}(1/\lambda_j^2)$, and $\pi(\lambda_j^2)$ denotes the density 
 function of an inverse gamma distribution, i.e., $\lambda_j^2\sim \mbox{IG}(a_1, s_n)$. 

The step of updating $\bbeta$ is computationally difficult due to the inverse of a $p_n \times p_n$ matrix.
However, the special structure of the covariance matrix $K^{-1}$ allows for a blockwise update 
 of $\bbeta$ \cite{IshawaranR2005}. For example,
 if we partition $\bbeta$ into two blocks $\bbeta^{(1)}$ and 
 $\bbeta^{(2)}$, and partition $\bX=[\bX_1,\bX_2]$ and $\Lambda=\mbox{diag}(\Lambda_1,\Lambda_2)$ accordingly, 
 then the conditional distribution of $\bbeta^{(1)}$ is given by 
\begin{equation}\label{bgibbs}
 \bbeta^{(1)}|\bbeta^{(2)}\sim \mbox{N}((\bX_1^T\bX_1+\Lambda_1)^{-1}\bX_1^T(\by-\bX_2\bbeta^{(2)}),\sigma^2(\bX_1^T\bX_1+\Lambda_1)^{-1}),
\end{equation} 
 which requires only an inverse of a lower dimensional matrix. 
The computational complexity  of
updating $\bbeta$ in (\ref{gibbs}) is $O(p_n^3)$, while that in (\ref{bgibbs}) is $O((d^3+n(p_n-d))p_n/d)$, 
where $d$ is the block size and the term $n(p_n-d)$ comes from computing the product $\bX_2\bbeta^{(2)}$.
The optimal order of $d$ is $O(\sqrt[3]{np_n})$, which yields a computational complexity of $O(n^{2/3}p_n^{5/3})$ for
one update of the entire vector $\bbeta$. Further improvement in computation is possible when we
 incorporate the idea of the skinny Gibbs sampler \cite{Narisetty2016}.

Posterior model selection based on BCS has been discussed in Sections \ref{main} and \ref{ext}
from the theoretical aspect. However, in practice, the selection rule $\xi(\bbeta)=\{j: |\beta_j/\sigma|>a_n\}$
cannot be directly used since $a_n$ is not an explicit hyperparameter of the prior distribution. 
Recall that $a_n$ represents the boundary of the prior spike region, and it is implicitly 
defined through the condition (\ref{priorconcentrationeq}) as $\pi(|\beta_j/\sigma|>a_n) = p_n^{-1-u}$. 
Since $u$ is unknown, we suggest to choose the threshold $a$ in the rule $\pi(|\beta_j/\sigma|>a) = 1/p_n$, i.e., let $u=0$. 
This rule can be interpreted as that 
the expected {\it a priori} model size 
is equal to 1. Such a rule has often been 
in the literature of Bayesian model selection, 
see e.g. \cite{NarisettyH2014}. 
Obviously, $a_n\leq a$, and thus it leads to a conservative selection. 
However, if $a\ll \min_{j\in\xi^*}|\beta_j^*|$, it is not difficult to see that the Bayesian
selection consistency remains, when $\min_{j\in\xi^*}|\beta_j^*|$ satisfies beta min condition.
 In the simulation studies of this paper, we choose the Bayesian estimator for the model as
 $\hat\xi=\{j: q_j\triangleq\pi(|\beta_j/\sigma|>a|D_n)>t\}$, where $t=0.5$ and $q_j$  
 plays the role of posterior inclusion probability. 
 It is worth to mention that one may also use a data-driven method to determine the value of $t$,
 and make the variable selection rule more robust across different sparsity regimes. 
 For instance, we can conduct a multiple hypothesis test based on the marginal inclusion probabilities
 $q_j$'s for the hypotheses $H_{j0}: \beta_j=0$
 versus $H_{j1}: \beta_j \ne 0$ $j=1,\ldots, p_n$ based on posterior summaries. 
 This can be done using an empirical Bayesian approach as developed in \cite{Efron2004, LiangZ2008}.  
 

 Another important practical issue is how to select
  hyperparameters. The theory developed in Section \ref{main} and Section \ref{ext}
 provides only sufficient conditions for the asymptotic order of hyperparameters. 
 For example, by Theorem \ref{colheavy}, one can set the scale parameter 
 $\lambda_n =1/[\sqrt{n\log p_n}p_n^{\gamma}]$ with any sufficiently large value of $\gamma$
 for the $t$-prior.  Asymptotically, 
 an excessively large value of $\gamma$ 
 doesn't affect the rate of convergence,
 but affects only the multiplicative constants,  such as $M$ and $c_1$, in the statement of Theorem \ref{thmmain}. 
 However, in finite-sample applications, it is crucial to select a properly scaled parameter such that 
 the posterior is neither over- nor under-shrunk.
 In this work, we let $\lambda_n = 1/[\sqrt{n\log p_n}p_n^{\hat\gamma}]$ and choose $\hat\gamma$ to minimize the posterior mean of a ``BIC-like score'': $\int bic(\bbeta,\sigma^2) d\pi(\bbeta,\sigma^2|D_n,\gamma)$,
 where $bic(\bbeta) = n\log (\|Y-\bX^T\tilde\bbeta\|^2/n)+\|\tilde \bbeta\|_0\log n$, $\tilde\bbeta=(\tilde\beta_1,\dots,\tilde\beta_{p_n})$, $\tilde\beta_j=\beta_j1(|\beta_j/\sigma|>a)$,
 and $\pi(\bbeta,\sigma^2|D_n,\gamma)$ is the posterior distribution of $(\bbeta, \sigma^2)$ given the hyperparameter $\gamma$.
 In practice, one can run multiple posterior simulations with
 different values of $\gamma$, and then choose the one that yields the smallest posterior sample mean of the ``BIC-like'' score. 
 Since the multiple runs can be made in parallel on a high-performance computer, 
 such a parameter tuning strategy doesn't add much on computational time. 
 Since investigating the theoretical properties of tuning parameter selection is beyond the scope of this work, such study will be conducted elsewhere.

 We illustrate the performance of BCS using a simulated example, where $p=200$, $n=120$, and
 the non-zero coefficients are $(\beta_1,\beta_2,\beta_3,\beta_4)=(1,1,1,1)$. 
For the Laplace prior, we set the hyper-parameter $\lambda = \sqrt{n \log(p_n)}$ 
 at which the Lasso estimator is known to be consistent, see e.g. \cite{ZhangH2008}. 
For the student-$t$ prior, we set the degree of freedom to be 3 
 with the scale parameter $s_n=\lambda_n^2=1/[n\log p_np_n^{-2\gamma}]$, 
 where $\gamma$ ranges from -0.25 to 1.1, the best $\hat\gamma$ is selected as described in the above.
For both priors, we let $\sigma^2$ be subject to an inverse gamma distribution with 
$a_0=b_0=1$.

The numerical results are summarized in Figure \ref{toy}. The first plot shows the
 posterior sample mean of the BIC-like score with different values of $\gamma$. It shows that when $\gamma$ is larger than 0.8,
 the tuning parameter $\lambda_n$ is too small, the posterior begins to miss true covariates due to over-shrinkage, and
 thus the posterior mean of the BIC-like score rapidly increases to a very large value.
 The second and third plots are the posterior boxplots of $\pi(\beta_j|D_n)$ of Bayesian Lasso, 
 and BCS under the optimal setting of $\hat\gamma$.  
 To make the boxplots more visible, we only include the coefficients of the first 50 covariates, including four true covariates.
 The comparison shows that BCS led to a consistent inference of the model in the 
 sense that the coefficients of the false covariates were shrunk to zero, and the coefficients of the true covariates were 
 distributed around their true values.
 In contrast, Bayesian Lasso over-shrunk the coefficients of true covariates, 
 and under-shrunk the coefficients of false covariates. 
 This is due to the fact the Laplace prior failed to achieve the balance between prior concentration and 
 tail thickness.
 But it is worth to note that the posterior Bayesian Lasso can still separate the true and false covariates, and 
 thus it can be used for model selection.  
 
\begin{figure}[htbp]
 \begin{center}
  \includegraphics[width=13cm]{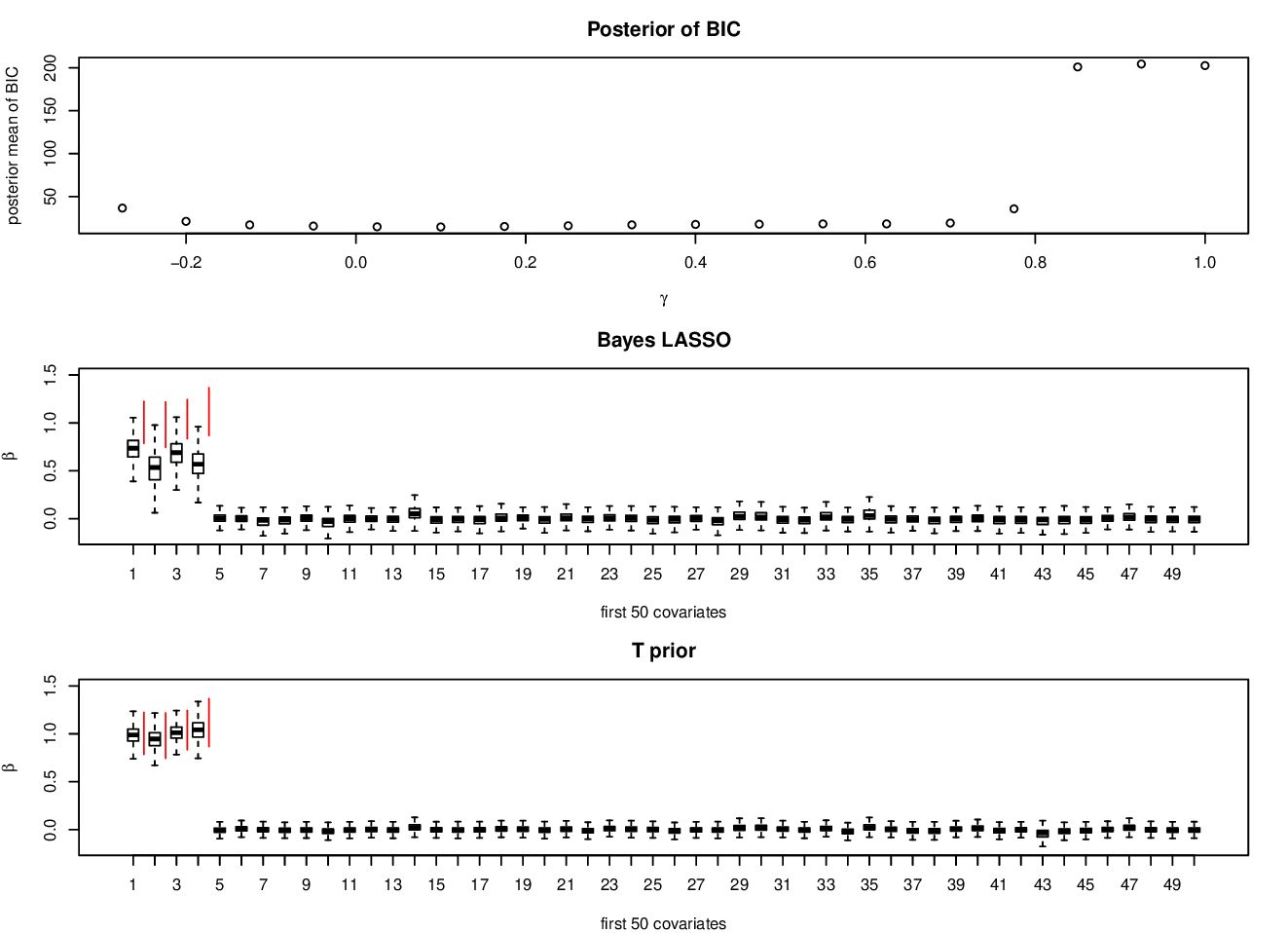}
  \caption{
   Upper: posterior mean of the BIC-like score for different value of $\gamma$;
   middle: box-plots of the posterior samples by Bayesian Lasso;
   lower: box-plots of the posterior samples by BCS.
  }\label{toy}
 \end{center}
\end{figure}

\begin{figure}[htbp]
 \begin{center}
  \includegraphics[width=12cm]{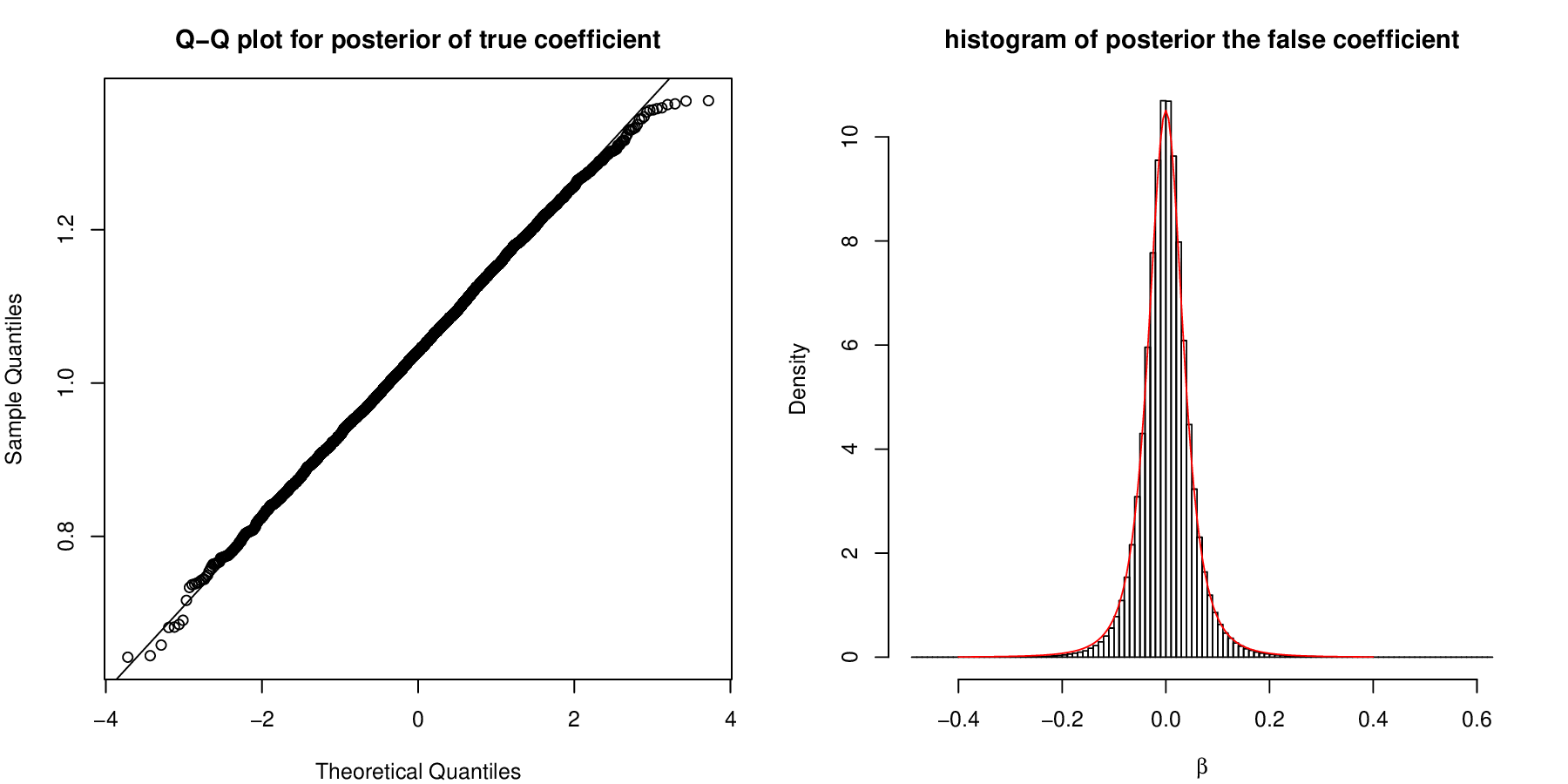}
  \caption{Shape of the posterior distribution  by BCS: the left plot shows the QQ plot for one true covariate, and the right plot shows a histogram of the posterior samples of $\beta_j/\sigma$, $j\notin\xi^*$ (i.e.,  
 false covariates), where the red curve represents the density of the student-$t$ prior.
  }\label{toytoy}
 \end{center}
\end{figure}

In addition, we drew in Figure \ref{toy} four red vertical segments
which represent the 99\% oracle confidence intervals of the 
true coefficients by assuming that the true model is known. In Figure \ref{toytoy}, 
we examined the shape of posterior samples resulted from BCS.
The plots are consistent with the established BvM Theorem (\ref{bvm2}).

 Figure \ref{toy} shows that for this example a wide range of $\gamma$, from -0.1 to 0.6,
 yielded similar posterior means for the  BIC-like score, which implies that the true model is correctly selected 
 under $\gamma$ within this range. 
 The BIC-like score posterior mean criterion tends to select a smaller value of $\gamma$ within this range, 
 since a smaller $\gamma$ reduces the shrinkage effect on the true covariates.
 But as shown in the supplementary material, the performance of BCS is actually quite stable with any $\gamma$ in this range.
 This also implies that BCS is tolerant to stochastic tuning errors.

As discussed previously, the Bayesian interval estimates obtained by BCS will be super-efficient for false covariates. 
Their coverages highly rely on the selection consistency, and have completely different 
performance compared to frequentist confidence intervals. 
The frequentist de-biased Lasso estimator 
is defined as $\hat\bbeta = \hat\bbeta_{\tiny \mbox{LASSO}}+ \frac{1}{n}S\bX^T(\by-\bX \hat\bbeta_{\tiny \mbox{LASSO}})$,
where $S$ is the surrogate inverse matrix of the sample covariance. 
This de-bias step applies an OLS-type bias correction to the Lasso estimator. 
In the ideal case that $p_n\leq n$ and $\frac{1}{n}S=(\bX^T\bX)^{-1}$, the de-biased Lasso estimator reduces to the OLS estimator.
Therefore, the marginal confidence intervals of all covariates, including both true and false, 
have the same length scale (Detailed illustration can be found in  the supplementary material).

\section{Numerical Studies}\label{simu}

This section examines the performance of BCS 
in variable selection and uncertainty assessment for the regression coefficient estimates.
The method is tested on two simulation examples and a real data example.

In the simulation study, two design matrices were considered for the model (\ref{lm}):
$(n,p)=(80,201)$ and (100,501), where the intercept term has been included.
The true values of the parameters are $\sigma^*=1$, $\bbeta=(0,1,1.5, 2,0,\dots,0)^T$, where the first 0 
corresponds to the intercept term. 
The  design matrices were generated from the multivariate normal distribution $\mbox{N}(0,\Sigma)$ 
 with the covariance structure 
1) independent covariates: $\Sigma = I$; or 2) pairwisely dependent covariates: 
 $\Sigma_{ii}=1.0$ for all $i$, $\Sigma_{i,j}=0.5$ for $i\neq j$. 
The methods under comparison include BCS, Bayesian Lasso, Lasso, and SCAD. 
 For Bayesian Lasso, we set the scale parameters  
 to be $\lambda=\sqrt{n\log p_n}$. 
 For BCS, the tuning parameter $s_n$ is selected by the posterior mean of the BIC-like score  as discussed in Section \ref{perform}.
 For the setting of the Gibbs sampler, we set the total iteration number to $N=40,000$ in addition to 5000 iterations for the burn-in process. The posterior samples were collected at every 40 iterations.
 The R-package {\it glmnet} \cite{FriedmanHT2010} and {\it ncvreg} \cite{BrehenyH2011}
 were used for implementing Lasso and SCAD, \textcolor{black}{where the tuning parameter $\lambda$ was chosen to minimize the 10-fold cross-validation error. For Lasso, this is to set $\lambda=  \texttt{lambda.min}$ in {\it glmnet}. Since LASSO is known to select many false variables, we have also tried to set $\lambda=\texttt{lambda.1se}$, which is to choose the largest value of $\lambda$ such that the cross-validation error is within one standard deviation of the minimum cross-validation error.}
 The R-package {\it hdi} \cite{DezeureBMM2014} was used  for implementing 
 de-biased Lasso. All the results reported below were based on 112 simulated replicates.

\subsection{Simulation I: $n$=80, $p$=201}
We evaluated accuracy of the estimates obtained from various methods in 
$L_1$-error, which is defined as $\sum_{j\in\xi^*}|\beta^*_j-\hat\beta_j|$ for 
the true covariates and $\sum_{j\notin\xi^*}|\hat\beta_j|$ 
for the false ones.  
For the Bayesian methods, the posterior mean was used as the point estimator, although 
which is not the optimal choice for minimizing the $L_1$-error.
We evaluated the accuracy of variable selection using the average number of
selected true covariates $|\hat\xi\cap\xi^*|$ (the perfect value is 3), and the
average number of selected false covariates $|\hat\xi\cap(\xi^*)^{c}|$ (the perfect value is 0).
For each covariate, we also compared the marginal
credible intervals produced by the Bayesian methods and the 
confidence intervals produced by de-biased Lasso under a nominal level of 95\%.
For simplicity, the credible intervals were constructed based on the empirical quantiles from posterior samples
instead of the highest density region.

\begin{table}
\caption{Comprehensive comparison of BCS, Bayesian Lasso (Bay-Lasso), LASSO with \texttt{lambda.min} (LASSO$_1$), LASSO with \texttt{lambda.1se} (LASSO$_2$), SCAD and
de-biased Lasso for the datasets with independent covariates, $n=80$ and $p=201$.}\label{Tg1}
\begin{center}
\begin{tabular}{ccccccc}
 \hline
 & \multicolumn{6}{c}{Methods}\\ \cline{2-7}
                                         & BCS            & Bay-Lasso        & debiased-Lasso & LASSO$_1$ & LASSO$_2$ &SCAD\\ \hline
$L_1$ error of $\bbeta_{\xi^*}$          &0.3380          &2.1115            & 0.3503  &0.6678 &1.0850 &0.2811\\ 
Standard error                           &0.0149          &0.0281            & 0.2537  &0.0211 &0.0275 &0.0133\\ \hline
$L_1$ error of $\bbeta_{(\xi^*)^{c}}$    &2.3137          &4.5533            & 23.305  &0.8402 &0.1650 &0.2180\\ 
Standard error                           &0.0758          &0.0360            & 0.2537  &0.0950 &0.0319 &0.0324\\\hline
\hline
$|\hat\xi\cap\xi^*|$                     & 3              &2.3036            &--- & 3       & 3      & 3    \\ 
Standard error                           & ---            &0.0505            &--- & ---     & ---    &  ---   \\\hline
$|\hat\xi\cap(\xi^*)^{c}|$               & 0              &0                 &--- & 14.161  & 3.1964 &4.3304\\ 
Standard error                           & ---            &---               &--- & 1.2841  & 0.5041 &0.5176\\\hline
\hline                                                     
Coverage of $\xi^*$                      &0.9067          & 0.0595           &0.9613&---&---&---\\
Average length                           &0.4996          & 0.8471           &0.5798&---&---&---\\ \hline
Coverage of $(\xi^*)^{c}$                &1.0000          & 1.0000           &0.9492&---&---&---\\ 
Average length                           &0.1371          & 0.3322           &0.5490&---&---&---\\ \hline
\end{tabular}
\end{center}
\end{table}

\begin{table}
\caption{Comprehensive comparison of BCS, Bayesian Lasso (Bay-Lasso), LASSO with \texttt{lambda.min} (LASSO$_1$), LASSO with \texttt{lambda.1se} (LASSO$_2$), SCAD and
de-biased Lasso for the datasets with dependent covariates, $n=80$ and $p=201$.}\label{Tg2}
\begin{center}
\begin{tabular}{ccccccc}
 \hline
 & \multicolumn{6}{c}{Methods}\\ \cline{2-7}
                                       & BCS             & Bay-Lasso  & debiased-Lasso & LASSO$_1$ & LASSO$_2$ &SCAD\\ \hline
$L_1$ error of $\bbeta_{\xi^*}$        &0.5040           &2.7798      & 0.4469  &0.8735 &0.9516 &0.3593\\ 
Standard error                         &0.0342           &0.0302      & 0.0192  &0.0265 &0.0242 &0.0170\\ \hline
$L_1$ error of $\bbeta_{(\xi^*)^{c}}$  &0.3805           &4.8558      & 24.795  &1.3638 &0.4509 &0.1587\\ 
Standard error                         &0.0782           &0.0302      & 0.2448  &0.1083 &0.0274 &0.0198\\\hline
\hline                                                          
$|\hat\xi\cap\xi^*|$                   & 2.9             &1.7500      &--- & 3     &3      & 3    \\ 
Standard error                         & 0.03            &0.0546      &--- & ---   &---    & ---   \\\hline
$|\hat\xi\cap(\xi^*)^{c}|$             & 0.01            &0           &--- & 16.179&6.7232 & 2.3125\\ 
Standard error                         & 0.01            &---         &--- & 0.9801&0.3418 & 0.2568\\\hline
\hline
Coverage of $\xi^*$                    &0.9000           & 0.0327     &0.8988&---&---&---\\
Average length                         &0.6970           & 0.9279     &0.6046&---&---&---\\ \hline
Coverage of $(\xi^*)^{c}$              &1.0000           & 1.0000     &0.9543&---&---&---\\ 
Average length                         &0.0373           & 0.3804     &0.5418&---&---&---\\ \hline
\end{tabular}
\end{center}
\end{table}

The results are summarized in Table \ref{Tg1} and Table \ref{Tg2} for the case of 
independent covariates and the case of dependent covariates, respectively.  
First of all, we can see that BCS 
worked extremely well in identifying true models, whose performance is almost perfect.
As seen in Section \ref{perform}, Bayesian Lasso can also distinguish the true and false covariates from 
 posterior samples when the coefficients of the true covariates are sufficiently large. 
 However, due to over-shrinkage, it doesn't work well when they are small. Hence,
Bayesian Lasso mis-identified some true covariates for this example.
Both Lasso and SCAD tend to select dense models, although the true covariates 
can be selected.  As mentioned previously, this is due to an inherent drawback of the 
 regularization methods. The regularization shrinks the true regression coefficients  
 toward zero. To compensate the shrinkage effect, some false covariates have to be included.
 \textcolor{black}{For Lasso, the comparison shows that the choice of $\lambda=\texttt{lambda.1se}$ alleviates the ``overselection'' issue, and leads to less estimation error for zero $\beta_j$'s and larger estimation bias for nonzero $\beta_j$'s.}
 BCS also shrinks the true regression coefficients, but it can still perform well in 
 variable selection. This is due to that BCS accounts for the uncertainty 
 of coefficient estimates in variable selection: BCS is sample-based, for which 
 different false covariates might be selected to compensate the shrinkage effect at different 
 iterations, and thus the chance of selecting false covariates can be largely eliminated by 
 averaging over different iterations. 

 Regarding parameter estimation, we note that SCAD yields a somehow better results than BCS. However,
 a direct comparison of these two methods is unfair, as the BCS tells us something more beyond 
 point estimation, e.g., credible interval. Also, BCS leads to much accurate variable selection
 as reported above. Among the Bayesian methods,  
 we can see that BCS performs much better than Bayesian Lasso, which indicates the 
 importance of posterior consistency. \textcolor{black}{We note that it is unfair to directly compare $L_1$-estimation errors of $\bbeta_{(\xi^*)^c}$ for shrinkage estimators (BCS or Bayesian Lasso) and sparse estimators (Lasso or SCAD), since the shrinkage estimators never shrink any coefficients to exactly zero. For example, in Table \ref{Tg1}, the $L_1$ error of BCS is 2.3, which is much larger than those by LASSO and SCAD. However, it actually implies that $\hat\beta_j\approx 2.3/200\approx0.011$ for each zero $\beta_j$, as BCS selected almost no false predictors. Hence, it represents a fairly successful shrinkage for the false predictors.}

 For interval estimation, de-biased Lasso produced high quality confidence intervals. 
 For both true and false covariates, it produced about 
 the same length confidence intervals, and the coverage rates of these confidence intervals were about the same as the nominal level.
 This observation is consistent with our previous
 discussion.
 For the true covariates, BCS yields almost 95\% converge; in contrast,  
 Bayesian Lasso yields a very low coverage due to the effect of over-shrinkage.
 For the false covariates, both BCS and Bayesian Lasso produced 100\% coverage with very narrow 
 credible intervals. Hence, they don't hold the long-run frequency coverage for the false predictors.
 These discoveries agree with our theoretical results.
 The de-biased Lasso yields wider intervals for the false covariates, as
 it cannot incorporate the model sparsity information 
 into the construction of confidence intervals. 

The performance of BCS for the cases of independent and dependent covariates are quite consistent, 
except that the proposed method tends to select a smaller value of $\gamma$ for the independent case 
 and, as a consequence, the posterior $L_1$-error of the false covariates tends to be larger than for 
 the dependent case.
 This is reasonable, as the high spurious correlation requires a higher
 penalty for the multiplicity adjustment.

\subsection{Simulation II: $n$=100, $p$=501}

\begin{table}
\caption{Comprehensive comparison of BCS, Bayesian Lasso (Bay-Lasso), LASSO with \texttt{lambda.min} (LASSO$_1$), LASSO with \texttt{lambda.1se} (LASSO$_2$), SCAD and
de-biased Lasso for the datasets with independent covariates, $n=100$ and $p=501$.}\label{Tg3}
\begin{center}
\begin{tabular}{cccccccc}
 \hline
 & \multicolumn{6}{c}{Methods}\\ \cline{2-7}
                                          & BCS      & Bay-Lasso & debiased-Lasso & LASSO$_1$ & LASSO$_2$ &SCAD\\ \hline
$L_1$ error of $\bbeta_{\xi^*}$           &0.2789    &2.3863     & 0.3177         & 0.7173 &0.9645 &0.2616\\ 
Standard error                            &0.0115    &0.0310     & 0.0145         & 0.0229 &0.0253 &0.0107\\\hline
$L_1$ error of $\bbeta_{(\xi^*)^{c}}$     &4.4011    &8.7190     & 50.301         & 0.9736 &0.2158 &0.3080\\ 
Standard error                            &0.0312    &0.0602     & 0.4636         & 0.0900 &0.0436 &0.0402\\\hline
\hline                                                                        
$|\hat\xi\cap\xi^*|$                      & 3             &2.1964         &---              & 3     &3      & 3    \\ 
Standard error                            & ---           &0.0436         &---              & ---   &---    & ---   \\\hline 
$|\hat\xi\cap(\xi^*)^{c}|$                & 0.0268        &0              &---              & 20.554&4.7411 & 7.0178\\ 
Standard error                            & 0.0153        &---            &---              & 1.6070&0.8629 & 0.8042\\\hline
\hline
Coverage of $\xi^*$                       & 0.9285   & 0.0208    &0.9494            &---&---&---\\
Average length                            & 0.430    & 0.7412    &0.4985            &---&---&---\\ \hline
Coverage of $(\xi^*)^{c}$                 & 1.0000   & 1.0000    &0.9517            &---&---&---\\ 
Average length                            & 0.1506   & 0.2841    &0.6038            &---&---&---\\ \hline
\end{tabular}                         
\end{center}
\end{table}

\begin{table}
\caption{Comprehensive comparison of BCS, Bayesian Lasso (Bay-Lasso), LASSO with \texttt{lambda.min} (LASSO$_1$), LASSO with \texttt{lambda.1se} (LASSO$_2$), SCAD and
de-biased Lasso for the datasets with dependent covariates, $n=100$ and $p=501$.}\label{Tg4}
\begin{center}
\begin{tabular}{ccccccc}
 \hline
 & \multicolumn{6}{c}{Methods}\\ \cline{2-7}
                                           & BCS       & Bay-Lasso & debiased-Lasso & LASSO$_1$ & LASSO$_2$ &SCAD\\ \hline
$L_1$ error of $\bbeta_{\xi^*}$            &0.3960     &3.1087     & 0.3888 & 0.8742 &1.0228  & 0.3338\\ 
Standard error                             &0.0260     &0.0300     & 0.0155 & 0.0282 &0.0223  & 0.0158\\\hline
$L_1$ error of $\bbeta_{(\xi^*)^{c}}$      &0.4288     &9.2585     &54.2889 & 1.3656 &0.5331  & 0.1424\\ 
Standard error                             &0.1076     &0.0694     & 0.4754 & 0.1045 &0.0342  & 0.0196\\\hline
\hline                                                             
$|\hat\xi\cap\xi^*|$                       & 2.9464           &1.4554     &---    & 3     &  3    & 3    \\ 
Standard error                             & 0.0213           &0.0566     &---    & ---   &---    & ---   \\\hline
$|\hat\xi\cap(\xi^*)^{c}|$                 & 0.0089           &0          &---    & 21.428&9.7324 & 6.4732\\ 
Standard error                             & 0.0089           &---        &---    & 1.3218&0.4742 & 0.8065\\\hline
\hline                                                            
Coverage of $\xi^*$                        & 0.9107    & 0.0060    & 0.9077&---&---&---\\
Average length                             & 0.5783    & 0.7498    & 0.5263&---&---&---\\ \hline
Coverage of $(\xi^*)^{c}$                  & 1.0000    & 1.0000    & 0.9316&---&---&---\\ 
Average length                             & 0.0219    & 0.2870    & 0.6142&---&---&---\\ \hline
\end{tabular}                                                     
\end{center}
\end{table}

The results are summarized in Tables \ref{Tg3} and \ref{Tg4} for 
the independent and dependent covariates, respectively. 
 As in the case with $n=80$ and $p=201$, BCS performs much better than 
 the regularization methods in variable selection, and performs 
 much better than Bayesian Lasso in all aspects of variable selection, 
 parameter estimation and interval estimation. 
 
 \vskip 0.2in
 {\noindent Before moving forward to the real application in the next section, we would like to mention that we also conduct simulations, under the same data generation scheme, for the two-Gaussian mixture prior specification, and the results are presented in the Supplementary Materials. While the two-Gaussian mixture prior also achieves near-perfect model selection performance, we find that its shrinkage effect on $\bbeta_{(\xi^*)^c}$ and its interval estimation coverage performance are inferior to those of $t$ shrinkage prior (although they are much better than Bayesian Lasso inference results). One potential reason is that the hyperparameters $m_1$, $\sigma_1^2$ and $\sigma_0^2$ are not optimally tuned. Our empirical experience shows that the value of $m_1$ has a large effect on model selection performance, and the values of $\sigma_1^2$ and $\sigma_0^2$ affect the level of posterior shrinkage and posterior normality asymptotics. However, tuning all three hyperparameters simultaneously is much more difficult in practice, than tuning only one hyperparameter of the $t$-shrinkage methods, hence is not recommended.
 }

\subsection{A Real Data Example}

 We analyzed a reduced gene expression dataset on Bardet-Biedl syndrome from  \cite{Scheetz2006}. The reduced dataset is available in the R package {\it flare} \cite{flareR}, which contains 120 samples with 201 gene expression levels. The scientific community has discovered that TRIM32 is the causal gene to Bardet-Biedl syndrome \cite{chiang2006homozygosity}. In this example, we treat the expression level of gene TRIM32 as the response variable and the expression levels of the other 200 genes as predictors. Therefore, the selected set of genes from this regression will cover the regulators of the gene TRIM32 by the consistency property of BCS. 
 
 We applied both de-biased Lasso and BCS to this regression problem.
 De-biased Lasso identified  gene 153 as the only significant covariate according to the  Bonferroni-adjusted p-values, and produced a 95\% confidence interval of  [0.024,0.072] for this gene.
 For BCS, 
 the optimal value $\hat \gamma = 0.58$ was selected, and
 the posterior exceedance probability $q_j\triangleq \pi(|\beta_j/\sigma|>a|D_n)$ was  
 used to quantify the significance of each covariate, 
 where $a$ is as defined in Section \ref{perform}.
 BCS also identified gene 153 as the most significant 
 covariate with $q_{153} =0.54$.
 Figure \ref{eye1} shows the posterior distribution of the regression coefficient of gene 153 under the 
  choice of $\hat\gamma=0.58$ as well as the confidence intervals produced by
 the two methods. The 95\% HPD 
 credible interval produced by BCS is $[-0.018, 0.018]\cup[0.064, 0.131]$, which
 is the union of two intervals
  representing the evidence against and for 
 being the true covariate, respectively.
 Note that if the true model is exactly the 153th gene, its OLS estimator will be 0.109.
 The de-biased Lasso confidence interval (represented by the dashed segment in Figure \ref{eye1}) seems a compromise between the two intervals, and it doesn't contain the OLS value 0.109.

\begin{figure}[htbp]
 \begin{center}
  \includegraphics[width=3.5in,height=2.5in]{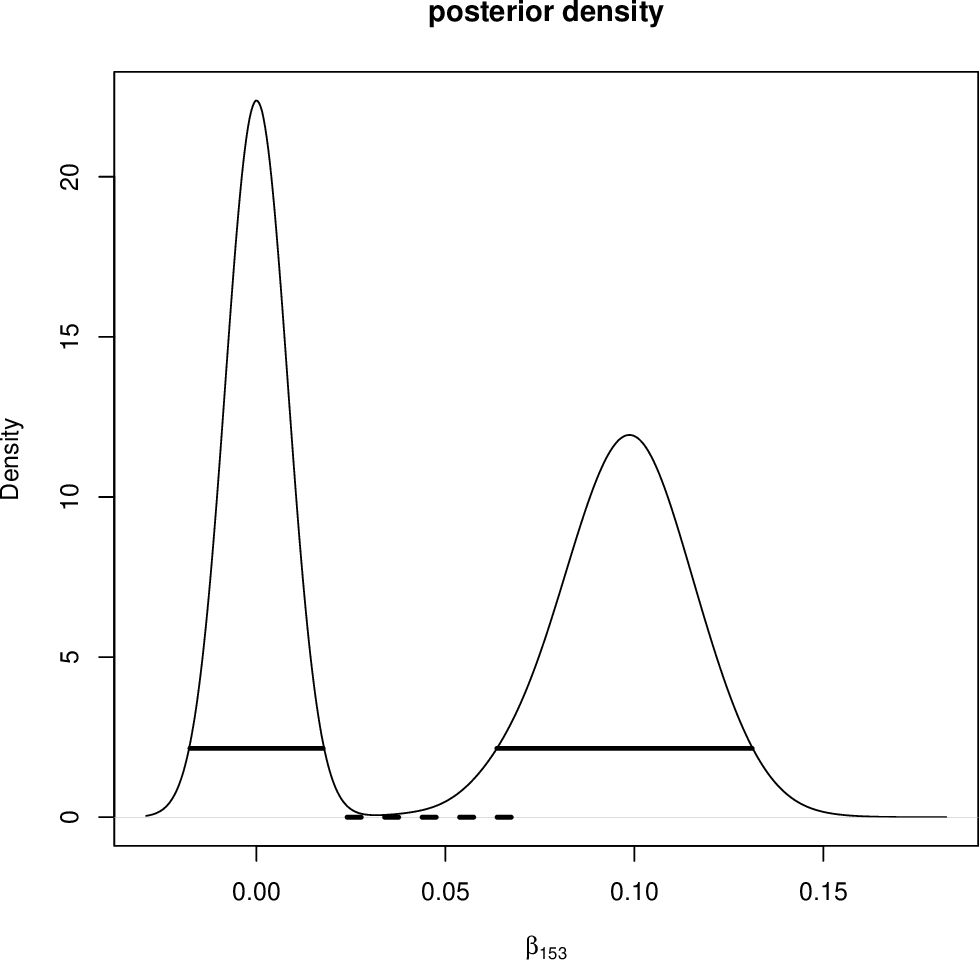}
  \caption{
  Histogram of the posterior samples of the regression coefficient of gene 153, where
  the black line shows the posterior HPD interval, and the dashed line shows the 
  de-biased Lasso confidence interval.}\label{eye1}
 \end{center}
\end{figure}

\section{Conclusion}\label{diss}

 In this paper, we have studied the posterior asymptotics under absolutely continuous priors 
 for high-dimensional linear regression. 
 We first proved that if the prior distribution is heavy-tailed and allocates a 
 sufficiently large probability mass in a very small neighborhood of zero, then the posterior consistency
 holds with a nearly optimal contraction rate.  
 More specifically, we found that any polynomial-tailed distribution 
 with a scale parameter, which decreases as $p_n$ increases, can be used as an appropriate prior 
 to derive valid Bayesian inference for high dimensional regression models.
 Note that it is not necessary for the continuous prior distribution to
 have an infinite density at zero as in the DL or horseshoe priors.

 In the literature, the local-global shrinkage prior has been widely studied, 
 especially for the normal means problem. Such prior follows 
 $\beta_j\sim N(0, \sigma^2\lambda_j^2\tau^2)$, where $\lambda_j^2$ controls the local shrinkage, and 
 $\tau^2$ controls the degree of global shrinkage. Our work verifies that 
 a sufficient condition that ensures consistency of the local-global shrinkage is to let the 
 local shrinkage parameter $\lambda_j^2$  follow some polynomial-tailed distribution, 
 and let the global shrinkage parameter $\tau^2$  deterministically decrease in the order
 $-\log(\tau^2)=O(\log p_n)$. 
 In this work, we suggest a BIC-like score posterior mean criterion for tuning the global shrinkage parameter. Although it works well for our examples,  it is still of great interest to the Bayesian community if  
 an adaptive or full Bayesian approach can be 
 developed for choosing, rather than tuning, 
 the global shrinkage parameter. 
 Such analysis has been conducted by \cite{VanSV2017} under normal means models. 
 However, there is a significant difference between  normal means models and regression models.
 For the former, one can directly analyze the marginal posterior $\pi(\beta_j|D_n)$ as  
 $\beta_j$'s are (conditionally) independent. 
 For the latter one, one needs to take into 
 account of the dependency among covariates. Empirically, the result of \cite{VanSV2017} seems not applicable to regression 
 problems. Figure \ref{shoe} shows the boxplots of the regression coefficients drawn from a posterior $\pi(\beta_j|D_n)$ constructed with a 
 horseshoe prior for the same dataset used in the toy example of Section \ref{perform}, where $\lambda_j$ is subject to a half-Cauchy prior, $\tau$ is subject to a uniform prior truncated on
 $[n^{-3/2}, n^{-1/2}]$. The plot shows that the horseshoe prior leads to many false discoveries 
 for this example. Therefore, we would note that adaptively choosing the global shrinkage parameter is nontrivial due to spurious multicollinearity caused by the curse of dimensionality.

 \begin{figure}[htbp]
 \begin{center}
  \includegraphics[width=10cm]{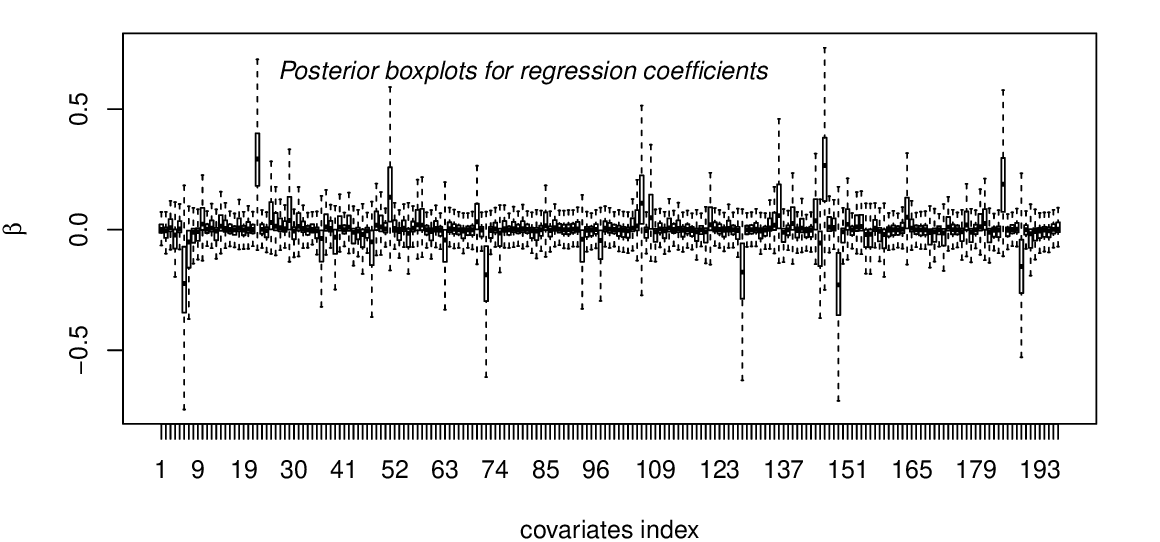}
  \caption{
  Boxplots of  $\{\beta_j\}_{j\notin\xi^*}$ simulated from a posterior distribution with a horseshoe prior for the same dataset used in Figure \ref{toy}, where the global shrinkage parameter was truncated into [$n^{-3/2}, n^{-1/2}$].}\label{shoe}
 \end{center}
\end{figure}

In this paper, we have also studied the selection consistency based on sparsified posterior,
as well as the posterior shape approximation. We proved that if the tail of the prior distribution 
is sufficiently flat, then selection is consistent and the BvM-type result holds.
This further implies that for the true covariates, the credible intervals are asymptotically equivalent to 
the oracle confidence intervals; and for the false covariates, the credible intervals 
are super-efficient. 

The theory established in this paper implies that a consistent shrinkage prior shares almost 
the same posterior asymptotic behavior with the golden standard spike-and-slab prior, see e.g. 
 \cite{CastilloSHV2015}. However, the shrinkage prior is more efficient in computation. In this paper, we used a student-$t$ prior in all numerical studies, and the Gibbs sampler was conveniently used in sampling from  posterior distributions. 
 The computation shall be further improved if a stochastic gradient MCMC algorithm is employed for simulations.
 However, for the spike-and-slab prior, a trans-dimensional MCMC sampler has to be used
 for simulations.


\Acknowledgements{Song's research is supported in part  by the NSF grant DMS-1811812. Liang's researchh is supported in part by
 the NSF grants DMS-2015498, and the NIH grants 
R01GM117597 and R01GM126089. The authors thank Professor Xuming He for his comments and recommendation on the paper, and the editor, associate editor and referees for their constructive comments on the paper.}


\Supplements{ }

\bibliographystyle{mybst2}
\bibliography{ref.bib}

\begin{appendix}

\section{Proof of the main theorem}
First, we restate the result from lemma 2.2.11 in \cite{VaartW1996} for the sake of readability.
\begin{lemma}[Bernstein's inequality]\label{lemmaa}
If $Z_1,\dots,Z_n$ are independent random variables with mean zero and satisfy that $E|Z_i|^m\leq m!M^{m-2}v_i/2$ for every 
$m>1$ and some constants $M$ and $v_i$, then
\[P(|\sum Z_i|>z)\leq 2\exp\{-z^2/2(v+Mz)\},\]
for $v\geq \sum v_i$.
\end{lemma}

As mentioned in \cite{LuoSW2014}, the conditions in Lemma \ref{lemmaa} are satisfied by the centered one-degree chi-square distribution.
\begin{lemma}\label{lemmac}
If $X$ follows $\chi_1^2$ distribution, there exists some constant C, such that for any $m\in\mathbb{N}$, we have
$E|X-E(X)|^m\leq C m!2^m$. Therefore, given any constant scale $\lambda$, $E|\lambda X-E(\lambda X)|^m\leq m!(2\lambda)^{m-2}(4C\lambda^2)$.
\end{lemma}

The following lemma (\cite{ZubkovS2013}) gives an upper bound for the tail probability of the binomial distribution.
\begin{lemma}\label{lemmad}
 For a Binomial random variable $X\sim \mbox{B}(n,v)$, for any $1<k<n-1$
 \[ 
  Pr(X\geq k+1)\leq 1- \Phi(\mbox{sign}(k-nv)\sqrt{2nH(v, k/n)}),
 \]
  where $\Phi$ is the CDF of standard Gaussian distribution and
$H(v, k/n)= (k/n)$ $\log(k/nv)+(1-k/n)\log[(1-k/n)/(1-v)]$.
\end{lemma}

We also restate Lemma 6 in \cite{Barron1998}:
\begin{lemma}\label{lemmab}
Let $B_n$ and $C_n$ be two subsets of the parameter space $\Theta$, and $\phi_n$ be the test function satisfying 
$\phi_n(D_n)\in[0,1]$ for any realization $D_n$ of the data generation. If $\pi(B_n)\leq b_n$, $E_{\theta^*}\phi_n(D_n)\leq b_n'$,
$\sup_{\theta\in C_n}E_{\theta}(1-\phi_n(D_n))\leq c_n$, where $E_\theta(\cdot)$ denotes the expectation with respect to the data generation with true parameter value
being $\theta$. Furthermore, if 
\[P^*\left\{\frac{m(D_n)}{f^*(D_n)}\geq a_n \right\}\geq 1-a_n',\]
where $f^*=f_{\theta^*}$ is the true density function, $m(D_n)=\int_{\Theta} \pi(\theta)f_{\theta}(D_n)d\theta$ is the margin probability of $D_n$. Then,
\[P^*\left(\pi(C_n\cup B_n|D_n)\geq\frac{b_n+c_n}{a_n\delta_n}\right)\leq \delta_n+b_n'+a_n',\]
for any $\delta_n$.
\end{lemma}


\begin{theorem}\label{genthm}
Consider a linear regression model (\ref{lm}) with the design matrix satisfying conditions $A_1$ and $A_2$. The prior of
$\sigma^2$ follows an inverse-Gamma distribution IG($a,b$), and the prior density of $\bbeta$ is given by 
\[\pi(\bbeta|\sigma^2)=\prod_{i=1}^{p_n} \frac{1}{\sigma}g_\lambda(\beta_i/\sigma).\] 
If there exists a positive constant $u$ such that
\begin{equation}\label{gencon}
\begin{split}
  & 1-\int_{-a_n}^{a_n}g_\lambda(x)dx\leq p_n^{-(1+u)},\\
  & - \log (\inf_{x\in[-E_n,E_n]}g_\lambda(x)) =O(\log p_n),\\
  \end{split}
\end{equation} 
hold for $a_n\asymp \sqrt{s\log p_n/n}/p_n$, then the posterior consistency holds asymptotically, i.e.,
\[ P^{*}\big\{\pi[A_n|D_n]>\exp(-c_1n\epsilon_n^2)\big\}\leq \exp(-c_2n\epsilon_n^2),\]
where $A_n=\{\mbox{At least $\tilde p$ entries of $|\bbeta/\sigma|$ is larger than } a_n\}
  \cup\{\|\bbeta-\bbeta^*\|\geq (3+\sqrt{\lambda_0})\sigma^*\epsilon_n\}\cup\{\sigma^2/\sigma^{*2}>(1+\epsilon_n)/(1-\epsilon_n)
  \mbox{ or }\sigma^2/\sigma^{*2}<(1-\epsilon_n)/(1+\epsilon_n)\}$ with
$\tilde p \asymp s$, $\epsilon_n= M\sqrt{s\log p_n/n}$ for some large constant $M$.
\end{theorem}

\begin{proof}
We apply Lemma \ref{lemmab} to prove this theorem. 
Define $C_n=A_n\backslash B_n$, where
$B_n=\{$At least $\tilde p$ entries of $|\bbeta/\sigma|$ is larger than $a_n\}$, 
$\tilde p\leq\bar p-s$, $\tilde p \prec n\epsilon_n^2$, and its specific choice will be given below.
The proof consists of three parts:

Firstly, we will show the existence of a testing function $\phi_n$ such that 
\begin{equation}\label{testfn}
 \begin{split}
  E_{(\bbeta^*,\sigma^{^*2})}(\phi_n)\leq \exp(-c_3n\epsilon_n^2),\quad \mbox{and}
  \sup_{(\bbeta,\sigma^2)\in C_n} E_{(\bbeta,\sigma^2)}(1-\phi_n)\leq \exp(-c_3'n\epsilon_n^2);
 \end{split}
\end{equation} for some positive constants $c_3$ and $c_3'$.

Secondly, we will show that for some $c_4>0$,
\begin{equation}\label{badset}
 \pi(B_n)<e^{-c_4n\epsilon_n^2}.
\end{equation}

Thirdly, we will show that 
\begin{equation}\label{priorconcern}
\lim_n P^*\left\{\frac{m(D_n)}{f^*(D_n)}\geq \exp(-c_5n\epsilon_n^2) \right\}>1-\exp\{-c_5'n\epsilon_n^2\},
\end{equation} for some positive $0<c_5<\min(c_3',c_4)$.
Therefore, the proof can be concluded by Lemma \ref{lemmab}.

{\bf Part I:}
We consider the following  testing function 
$\phi_n=\max\{\phi'_n,\tilde\phi_n\}$, where
\[
\begin{split}
\phi'_n&=\max_{\{\xi\supseteq\xi^*, |\xi|\leq \tilde p+s\}}1\{\big|\by^T(I-H_{\xi})\by/(n-|\xi|)\sigma^{*2}-1\big|\geq \epsilon_n\}, \mbox{ and }\\
\tilde\phi_n &= \max_{\{\xi\supseteq\xi^*, |\xi|\leq \tilde p+s\}}
1\{\|(\bX_\xi^T\bX_\xi)^{-1}\bX_\xi^T\by-\bbeta^*_\xi)\|\geq \sigma^*\epsilon_n\},
\end{split}
\]
and $H_\xi = \bX_{\xi}(\bX_{\xi}^T\bX_{\xi})^{-1}\bX_{\xi}^T$ is the hat matrix corresponding to $\xi$.

For any $\xi$ that satisfies $\xi\supseteq\xi^*, |\xi|\leq \tilde p+s$, we have
\begin{equation}\label{test1}
\begin{split}
   &E_{(\bbeta^*,\sigma^{*2})}1\{\big|\by^T(I-H_{\xi})\by/(n-|\xi|)\sigma^{*2}-1\big|\geq \epsilon_n\}\\
  = &Pr(|\chi^2_{n-|\xi|}-(n-|\xi|)|\geq (n-|\xi|)\epsilon_n)
  \leq \exp(-\hat c_3n\epsilon_n^2)
\end{split}
\end{equation}
for some small constant $\hat c_3$,
where $\chi^2_{p}$ denotes a chi-square distribution with degree of freedom $p$, and the last inequality follows from Bernstein inequality (Lemma \ref{lemmaa} and \ref{lemmac}) 
and the facts $\epsilon\prec1$, $s+\tilde p\prec n$.

Following similar arguments as in the proof of Lemma 1 in \cite{ArmaganDL2013}, we have that 
for any $\xi$ satisfying  $\xi\supseteq\xi^*, |\xi|\leq \tilde p+s\prec n\epsilon_n^2$,
\begin{equation}\label{test2}
\begin{split}
 &E_{(\bbeta^*,\sigma^{*2})}1\{\|(\bX_\xi^T\bX_\xi)^{-1}\bX_\xi^T\by-\bbeta^*_\xi\|\geq \sigma^*\epsilon_n|\bbeta^*,\sigma^{*2}\}\\
=&E_{(\bbeta^*,\sigma^{*2})}1\{\|(\bX_\xi^T\bX_\xi)^{-1}\bX_\xi^T\bvarepsilon\|\geq \epsilon_n\}\leq Pr(\chi^2_{|\xi|}\geq n \lambda_0\epsilon_n^2)\\
\leq&
\exp(-\tilde c_3n\epsilon_n^2),
\end{split}
\end{equation}
for  some $\tilde c_3>0$. Note that the last inequality holds due to Bernstein inequality and large value of $M$.

Combining (\ref{test1}) and (\ref{test2}),  we obtain that
\begin{equation}
\begin{split}\label{test3}
E_{(\bbeta^*,\sigma^{*2})}\phi_n
& \leq E_{(\bbeta^*,\sigma^{*2})}\sum_{\{\xi\supseteq\xi^*, |\xi|\leq \tilde p+s\}}\big(1\{\big|\by^T(I-H_{\xi})\by/(n-|\xi|)\sigma^{*2}-1\big|\geq \epsilon_n\}\\
 &\qquad+1\{\|(\bX_\xi^T\bX_\xi)^{-1}\bX_\xi^T\by-\bbeta^*_\xi)\|\geq \epsilon_n\}\big)\\
& <(\tilde p+s){p_n \choose \tilde p+s}[\exp(-c_3n\epsilon_n^2)+\exp(-c_3'n\epsilon_n^2)].
\end{split}
\end{equation}
We set $\tilde p = \lfloor\min\{\hat c_3,\tilde c_3\}n\epsilon_n^2/(2\log p_n)\rfloor$. (Since $\bar p\log p_n\succ n\epsilon_n^2$,
 $\tilde p$ always exists.) 
 Hence, we have  $\log(\tilde p+s)+ (\tilde p+s)\log p_n <(2\min\{\hat c_3,\tilde c_3\}n\epsilon_n^2)/3$,
 which leads to $E_{(\bbeta^*,\sigma^{*2})}\phi_n\leq \exp(- c_3n\epsilon_n^2)$ for some 
 fixed $c_3$.
 
Now we study $\sup_{(\bbeta,\sigma^2)\in C_n}E_{(\bbeta,\sigma^2)}(1-\phi_n)$.
Let $C_n\subset \hat C_n\cup\tilde C_n$, where
\[
\begin{split}\hat C_n=&\{
\sigma^2/\sigma^{*2}>(1+\epsilon_n)/(1-\epsilon_n)\mbox{ or }\sigma^2/\sigma^{*2}<(1-\epsilon_n)/(1+\epsilon_n)\}\\
       &\cap\{\mbox{at most $\tilde p$ entries of $|\bbeta/\sigma|$ is larger than } 
a_n\},\\
\tilde C_n=&\{\|\bbeta-\bbeta^*\|>(3+\sqrt{\lambda_0})\sigma^*\epsilon_n, \sigma^2/\sigma^{*2}\leq(1+\epsilon_n)/(1-\epsilon_n)\\
          & \mbox{ and at most $\tilde p$ entries of $|\bbeta/\sigma|$ is larger than } 
a_n\}.
\end{split}
\]
Then we have
\[
\begin{split}
\sup_{(\bbeta,\sigma^2)\in C_n}E_{(\bbeta,\sigma^2)}(1-\phi_n)
& = \sup_{(\bbeta,\sigma^2)\in C_n}E_{(\bbeta,\sigma^2)}\min\{1-\phi'_n,1-\tilde\phi_n\}\\
& \leq \max\{\sup_{(\bbeta,\sigma^2)\in \hat C_n}E_{(\bbeta,\sigma^2)}(1-\phi'_n), \sup_{(\bbeta,\sigma^2)\in \tilde C_n}E_{(\bbeta,\sigma^2)}(1-\tilde\phi_n)\}.
\end{split}\]
Let $\tilde\xi=\tilde\xi(\bbeta)=\{k: |\beta_k/\sigma|>a_n\}\cup\xi^*$,
and $\tilde\xi^c=\{1,\dots,p_n\}\backslash \tilde\xi$.
 Hence, for any
$(\bbeta,\sigma^2)\in\tilde C_n\cup \hat C_n$, $|\tilde\xi(\bbeta)|\leq \tilde p + s\leq \bar p$, 
and $\|\bX_{{\tilde\xi}^c}\bbeta_{{\tilde\xi}^c}\|\leq \sqrt{np}\|\bbeta_{{\tilde\xi}^c}\|\leq\sqrt{n}\sqrt{\lambda'_0}\sigma\epsilon_n$ given a large value of $M$.
\[
\small
\begin{split}
 &\sup_{(\bbeta,\sigma^2)\in C'_n}E_{\bbeta}(1-\phi'_n)\\
&=\sup_{(\bbeta,\sigma^2)\in C'_n}E_{(\bbeta,\sigma^2)}\min_{\xi\supseteq\xi^*, |\xi|\leq \tilde p+s}1\{\big|\by^T(I-H_{\xi})\by/(n-|\tilde\xi|)\sigma^{*2}-1\big|\leq \epsilon_n\}\\
&\leq \sup_{(\bbeta,\sigma^2)\in C'_n}E_{(\bbeta,\sigma^2)}1\{\big|\by^T(I-H_{\tilde\xi})\by/(n-|\tilde\xi|)\sigma^{*2}-1\big|\leq \epsilon_n\}\\
&=\sup_{(\bbeta,\sigma^2)\in C'_n}Pr\{|\sigma^2(\bX_{{\tilde\xi}^c}\bbeta_{{\tilde\xi}^c}/\sigma+\bvarepsilon)^T (I-H_{\tilde\xi})(\bX_{{\tilde\xi}^c}\bbeta_{{\tilde\xi}^c}/\sigma+\bvarepsilon)/[(n-|\tilde\xi|)\sigma^{*2}]-1|\leq \epsilon_n \}\\
&\leq\sup_{(\bbeta,\sigma^2)\in C'_n} Pr\{\sigma^2(\bX_{{\tilde\xi}^c}\bbeta_{{\tilde\xi}^c}/\sigma+\bvarepsilon)^T (I-H_{\tilde\xi})(\bX_{{\tilde\xi}^c}\bbeta_{{\tilde\xi}^c}/\sigma+\bvarepsilon)/[(n-|\tilde\xi|)\sigma^{*2}]\in[1-\epsilon_n, 1+ \epsilon_n] \}\\
&\leq\sup_{(\bbeta,\sigma^2)\in C'_n} Pr\{ (\bX_{{\tilde\xi}^c}\bbeta_{{\tilde\xi}^c}/\sigma+\bvarepsilon)^T (I-H_{\tilde\xi})(\bX_{{\tilde\xi}^c}\bbeta_{{\tilde\xi}^c}/\sigma+\bvarepsilon)/(n-|\tilde\xi|)\notin[1-\epsilon_n,1+\epsilon_n]\}\\
&\leq \sup_{(\bbeta,\sigma^2)\in C'_n} Pr\{ |\chi^2_{n-|\tilde\xi|}(k)-(n-|\tilde\xi|)|\geq (n-|\tilde \xi|)\epsilon_n\}\\
&\leq 
\exp(-\hat c_3'n\epsilon_n^2),
\end{split}
\]
for some $\hat c_3'>0$. 
Note that $(\bX_{{\tilde\xi}^c}\bbeta_{{\tilde\xi}^c}/\sigma+\bvarepsilon)^T (I-H_{\tilde\xi})(\bX_{{\tilde\xi}^c}\bbeta_{{\tilde\xi}^c}/\sigma+\bvarepsilon)$
follows a noncentral $\chi^2$ distribution $\chi_{n-|\tilde\xi|}(k)$ with the noncentral parameter $k =\bbeta_{{\tilde\xi}^c}^T\bX_{{\tilde\xi}^c}^T (I-H_{\tilde\xi})\bX_{{\tilde\xi}^c}\bbeta_{{\tilde\xi}^c}
/\sigma^2\leq (\sqrt{n}\sqrt{\lambda'_0}\epsilon_n/4)^2$.
Since the noncentral $\chi^2$ distribution is a sub-exponential, the last inequality follows from the Bernstein inequality as well.
Also, we have  
\[
\small
\begin{split}
 &\sup_{(\bbeta,\sigma^2)\in\tilde C_n}E_{(\bbeta,\sigma^2)}(1-\tilde\phi_n) 
  =\sup_{(\bbeta,\sigma^2)\in\tilde C_n}E_{(\bbeta,\sigma^2)}\min_{|\xi|\leq \tilde p+s}1\{\|(\bX_{\xi}^T\bX_{\xi})^{-1}\bX_{\xi}^T\by-\bbeta^*_{\xi}\|\leq \sigma^*\epsilon_n\}\\
  &\leq\sup_{(\bbeta,\sigma^2)\in\tilde C_n}E_{(\bbeta,\sigma^2)}1\{\|(\bX_{\tilde\xi}^T\bX_{\tilde\xi})^{-1}\bX_{\tilde\xi}^T\by-\bbeta^*_{\tilde\xi}\|\leq \sigma^*\epsilon_n\}\\
  &=  \sup_{(\bbeta,\sigma^2)\in\tilde C_n}Pr\{\|(\bX_{\tilde\xi}^T\bX_{\tilde\xi})^{-1}\bX_{\tilde\xi}^T\by-\bbeta^*_{\tilde\xi}\|\leq \sigma^*\epsilon_n|\bbeta,\sigma^2\}\\
  &=\sup_{(\bbeta,\sigma^2)\in\tilde C_n}Pr\{\|(\bX_{\tilde\xi}^T\bX_{\tilde\xi})^{-1}\bX_{\tilde\xi}^T\sigma\bvarepsilon+\bbeta_{\tilde\xi}+(\bX_{\tilde\xi}^T\bX_{\tilde\xi})^{-1}\bX_{\tilde\xi}^T\bX_{{\tilde\xi}^c}\bbeta_{{\tilde\xi}^c}-\bbeta^*_{\tilde\xi}\|\leq \sigma^*\epsilon_n\}\\
  &\leq\sup_{(\bbeta,\sigma^2)\in\tilde C_n}Pr\{\|(\bX_{\tilde\xi}^T\bX_{\tilde\xi})^{-1}\bX_{\tilde\xi}^T\sigma\bvarepsilon\|
  \geq \|\bbeta_{\tilde\xi}-\bbeta^*_{\tilde\xi}\|-(\bX_{\tilde\xi}^T\bX_{\tilde\xi})^{-1}\bX_{\tilde\xi}^T\bX_{{\tilde\xi}^c}\bbeta_{{\tilde\xi}^c}-\sigma^*\epsilon_n\}\\
  &=\sup_{(\bbeta,\sigma^2)\in\tilde C_n}Pr\{\|(\bX_{\tilde\xi}^T\bX_{\tilde\xi})^{-1}\bX_{\tilde\xi}^T\bvarepsilon\|
  \geq [\|\bbeta_{\tilde\xi}-\bbeta^*_{\tilde\xi}\|-\sigma^*\epsilon_n-(\bX_{\tilde\xi}^T\bX_{\tilde\xi})^{-1}\bX_{\tilde\xi}^T\bX_{{\tilde\xi}^c}\bbeta_{{\tilde\xi}^c}]/\sigma\}\\
  &\leq\sup_{(\bbeta,\sigma^2)\in\tilde C_n}Pr\{\|(\bX_{\tilde\xi}^T\bX_{\tilde\xi})^{-1}\bX_{\tilde\xi}^T\bvarepsilon\|
  \geq \epsilon_n\}\leq 
  \exp(-\tilde c_3n\epsilon_n^2),
 \end{split}
\]
where above
inequalities hold asymptotically because
$\|\bbeta_{\tilde\xi}-\bbeta^*_{\tilde\xi}\|\geq \|\bbeta-\bbeta^*\|-p_n(\sqrt{\lambda_0}\epsilon_n\sigma/p_n)$,
$\sigma^*/\sigma\geq\sqrt{(1-\epsilon_n)/(1+\epsilon_n)}$,
 and the fact that
\[
\begin{split}
&\|(\bX_{\tilde\xi}^T\bX_{\tilde\xi})^{-1}\bX_{\tilde\xi}^T\bX_{{\tilde\xi}^c}\bbeta_{{\tilde\xi}^c}/\sigma\|
\leq \sqrt{\lambda_{\max}\left((\bX_{\tilde\xi}^T\bX_{\tilde\xi})^{-1}\right)}\|\bX_{{\tilde\xi}^c}\bbeta_{{\tilde\xi}^c}\| \\
\leq&\sqrt{1/n\lambda_0}\sqrt{n\lambda'_0}\epsilon_n\sigma\leq \epsilon_n.
\end{split}
\]
Hence, (\ref{testfn}) is proved.
\vskip 0.1in
{\bf Part II:}
Define $N=|\{i:|\beta_i/\sigma|\geq a_n\}|$,
thus $N \sim \mbox{Binomial}(p_n, v_n)$, where $v_n=\int_{|x|\geq a_n} g_\lambda(x)dx$ and $g_\lambda(x)$ is the prior density
function of $\beta_i/\sigma$.
Thus $\pi(B_n) = Pr(\mbox{Binomial}(p_n, v_n)\geq \tilde p)$. By Lemma \ref{lemmad}, we have 
\[
\begin{split}
 &\pi(B_n)\leq 1-\Phi(\sqrt{2p_nH[v_n,(\tilde p-1)/p_n]})\leq\frac{\exp\{-p_nH[v_n,(\tilde p-1)/p_n]\}}{\sqrt{2\pi}\sqrt{2p_nH[v_n,(\tilde p-1)/p_n]}},\\
 &p_nH[v_n,(\tilde p-1)/p_n]\\
 = &(\tilde p-1)\log[(\tilde p-1)/(p_nv_n)]+(p_n-\tilde p+1)\log[(p_n-\tilde p+1)/(p_n-p_nv_n)].
\end{split}
\]
Therefore, to prove (\ref{badset}), it is sufficient to show that $p_nH[v_n,(\tilde p-1)/p_n] \geq O(n\epsilon_n^2)$. 
Since $1/(p_nv_n)\geq O(p_n^{u})$, $\tilde p\log p_n^u \asymp n\epsilon_n^2$ (if $M$ is sufficiently large), 
$(\tilde p-1)\log[(\tilde p-1)/(p_nv_n)] \asymp n\epsilon_n^2$, and 
$(p_n-\tilde p+1)\log[(p_n-\tilde p+1)/(p_n-p_nv_n)]\approx \tilde p-\tilde p^2/p_n\prec n\epsilon_n^2$.
Hence we have $p_nH[v_n,(\tilde p-1)/p_n] = O(n\epsilon_n^2)$.

\vskip 0.1in
{\bf Part III:}
Now we prove (\ref{priorconcern}).
Because 
\[
m(D_n)/f^*(D_n)=\int\frac{(\sigma^*)^n\exp\{-\|\by-\bX\bbeta\|^2/2\sigma^2\}}{\sigma^n\exp\{-\|\by-\bX\bbeta^*\|^2/2\sigma^{*2}\}}\pi(\bbeta,\sigma^2)d\bbeta d\sigma^2,
\]
 it is sufficient to show that 
\[
\begin{split}
&P^*\big(\pi(\{\|\by-\bX\bbeta\|^2/2\sigma^2+n\log(\sigma/\sigma^*)<\|\by-\bX\bbeta^*\|^2/2\sigma^{*2}+c_5n\epsilon_n^2/2\})\\
&\geq e^{-c_5n\epsilon_n^2/2}\big)
\geq 1-\exp\{-c_5'n\epsilon_n^2\},
\end{split}
\]
for some sufficiently small positive $c_5$.

Note that
$P^*(\Omega=\{\|\bvarepsilon\|^2\leq n (1+ \hat c_5)\mbox{ and }\|\bvarepsilon^T\bX\|_{\infty}\leq \hat c_5n\epsilon_n\})\geq 1-\exp\{-c_5'n\epsilon_n^2\}$
for some $\hat c_5$, by the properties of chi-square distribution and normal distribution, 
On the event of $\Omega$, it is easy to see that
 $\{\|\by-\bX\bbeta\|^2/2\sigma^2+n\log(\sigma/\sigma^*)<\|\by-\bX\bbeta^*\|^2/2\sigma^{*2}+c_5n\epsilon_n^2/2\}$ is a super-set of
 $\{\sigma\in[\sigma^*,\sigma^*+\eta_1\epsilon_n^2], \mbox{ and } \|(\bbeta^*-\bbeta)/\sigma\|_1<2\eta_2\epsilon_n \}$ for some small constants $\eta_1$ and $\eta_2$. 

In addition, we have
\begin{equation}\label{pp}
\begin{split}
&-\log\pi(\{\sigma\in[\sigma^*,\sigma^*+\eta_1\epsilon_n^2], \mbox{ and } \|(\bbeta^*-\bbeta)/\sigma\|_1<2\eta_2\epsilon_n \})\\
=&-\log\pi(\{0\leq\sigma^2-\sigma^{*2}\leq\eta_1\epsilon_n^2]\})-\log\pi(\{\|(\bbeta^*-\bbeta)/\sigma\|_1<2\eta_2\epsilon_n\})
\end{split}
\end{equation}

Given the fact that the inverse Gamma density is always bounded away from zero around $\sigma^{*2}$,
hence the first term in (\ref{pp}) satisfies $-\log\pi(\{0\leq\sigma^2-\sigma^{*2}\leq\eta_1\epsilon_n^2]\})\leq-\log(\eta_1\epsilon_n^2)
-\log(\min_{\sigma\in[\sigma^*,\sigma^*+\eta_1\epsilon_n^2]}\pi(\sigma^2))<\mbox{constant}+\log(1/\epsilon_n^2)\leq \delta_1 n\epsilon_n^2$,
where $\delta_1$ can be an arbitrary constant if we choose $M$ to be sufficiently large.
  
For the second term in (\ref{pp}),
\[
\begin{split}
&\{\|(\bbeta^*-\bbeta)/\sigma\|_1<2\eta_2\epsilon_n\}\supset\{|\beta_j/\sigma|\leq\eta_2\epsilon_n/p_n\mbox{ for all } j\notin\xi^*\}\\
&\cap\{\beta_j/\sigma\in[\beta_j^*/\sigma-\eta_2\epsilon_n/ s,\beta_j^*/\sigma+\eta_2\epsilon_n/ s]
\mbox{ for all }j\in\xi^*\},
\end{split}
\]
and 
\begin{equation}\label{priorfalse}
\begin{split}
&\pi(\{|\beta_j/\sigma|\leq\eta_2\epsilon_n/p_n\mbox{ for all } j\notin\xi^*\})\\
\geq&\pi(\{|\beta_j/\sigma|\leq a_n\mbox{ for all } j\notin\xi^*\})\geq (1-p_n^{-1-u})^{p_n}\rightarrow 1,
\end{split}
\end{equation}
given a large value of $M$.
For those $\beta_j^*\neq0$,
\begin{equation}\label{priortrue}
\begin{split}
&\pi\big(\{\beta_j/\sigma\in[\beta_j^*/\sigma\pm\eta_2\epsilon_n/ s]\mbox{ for all }j\in\xi^*\big \}\big)
\geq [2\eta_2\epsilon_n\inf_{x\in[-E,E]} g_\lambda(x)/ s]^s ,
\end{split},
\end{equation}
the inequality holds because $|\beta^*_j/\sigma|+\eta\epsilon_n/ s\leq E$ which is implied by
$\sigma^2<\sigma^{*2}+\eta'\epsilon_n^2$ and $|\beta_j^*/\sigma^*|\leq \gamma E$.
By (\ref{priorfalse}), (\ref{priortrue}) and condition (\ref{gencon}), (\ref{priorconcern}) holds.

\end{proof}

\begin{theorem}\label{genpred}
If all the conditions of theorem \ref{genthm} except condition $A_1$(3) hold,  
then posterior prediction for observed data is consistent, i.e.,
\[ P^{*}\big\{\pi[\|\bX\bbeta-\bX\bbeta^*\|\geq c_0\sqrt n\sigma^*\epsilon_n|D_n]<1-\exp(-c_1n\epsilon_n^2)\big\}\leq \exp(-c_2n\epsilon_n^2),\]
for some $c_0$, $c_1$ and $c_2$.
\end{theorem}
\begin{proof}
Define
$A_n=\{\mbox{At least $\tilde p$ entries of $|\bbeta/\sigma|$ is larger than } a_n\}
  \cup\{\|\bX\bbeta-\bX\bbeta^*\|\geq c_0\sqrt n\sigma^*\epsilon_n\}\cup\{\sigma^2/\sigma^{*2}>(1+\epsilon_n)/(1-\epsilon_n)
  \mbox{ or }\sigma^2/\sigma^{*2}<(1-\epsilon_n)/(1+\epsilon_n)\}$, 
$B_n=\{
\mbox{At least $\tilde p$ entries of $|\bbeta/\sigma|$ is larger than}$ $a_n\}$ and $C_n=A_n\backslash B_n$,
where $\tilde p\leq\bar p-s$, $\tilde p \prec n\epsilon_n^2$.

We still follow the three-step proof as in  Theorem \ref{genthm}. Since the proof is quite similar, the
details are omitted here. The only difference is that we now consider a slightly different testing function as 
\[
\begin{split}
\phi'_n&=\max_{\{\xi\supseteq\xi^*, |\xi|\leq \tilde p+s\}}1\{\big|\by^T(I-H_{\xi})\by/(n-|\xi|)\sigma^{*2}-1\big|\geq \epsilon_n\}, \mbox{ and }\\
\tilde\phi_n &= \max_{\{\xi\supseteq\xi^*, |\xi|\leq \tilde p+s\}}
1\{\|\bX_\xi(\bX_\xi^T\bX_\xi)^{-1}\bX_\xi^T\by-\bX_\xi\bbeta^*_\xi)\|\geq c_0\sigma^*\sqrt n\epsilon_n/3\}.
\end{split}
\]
Note that in the proof of Theorem \ref{genthm}, we need to bound the singular value of $(\bX_\xi^T\bX_{\xi})^{-1}\bX_{\xi}^T$ via condition $A_1$(3). However, in the proof of Theorem \ref{genpred}, only matrix $\bX_{\xi}(\bX_\xi^T\bX_{\xi})^{-1}\bX_{\xi}^T$ gets involved, and its eigenvalues are always bounded by 1. Thus condition $A_1$(3) is redundant.

\end{proof}

\begin{theorem}\label{vs}
Assume the conditions of Theorem \ref{genthm} hold, and let $\xi=\{j: |\beta_j/\sigma|> a_n\}$ denote a posterior subset model.
If the following conditions also hold: 
\[
\begin{split}
 &\lim\sup \sqrt{n}a_np_n\sigma^*/\sqrt{\log p_n} <k,\\
 &\min_{j\in\xi^*}|\beta_j^*|\geq M_1 \sqrt{\log p_n/n} \mbox{ for some large }M_1,\\
 &u>1+c/2+k^2/2\sigma^{*2}+2\sqrt{c'}k,\\
 &l_n=\max_{j\in\xi^*}\sup_{\substack{x_1,x_2\in \beta_j^*/\sigma^*\pm c_0\epsilon_n\\|x_1|,|x_2|\geq a_n}}\frac{g_\lambda(x_1)}{g_\lambda(x_2)},\mbox{ and } s\log l_n\prec \log p_n
\end{split}\]
for some constants $c'>1$, $c$ and sufficiently large $c_0$,
then
\[ P^*\{\pi(\xi=\xi^*|\bX,\by)>1-o(1) \}> 1-o(1).\]
\end{theorem}
\begin{proof}
For any $\bbeta_{\xi}$ which is a subvector of $\bbeta$ corresponding to $\xi$, we define
\[
\begin{split}
SSE(\bbeta_{\xi})&=\min_{\bbeta_{\xi^c}}\|\bY-\bX_{\xi}\bbeta_{\xi}-\bX_{\xi^c}\bbeta_{\xi^c}\|^2 \\
&=(\bY-\bX_{\xi}\bbeta_{\xi})^T(I-\bX_{\xi^c}(\bX_{\xi^c}^T\bX_{\xi^c})^{-1}\bX_{\xi^c}^T)(\bY-\bX_{\xi}\bbeta_{\xi}). \\
\end{split}
\]

By the consistency result in Theorem \ref{genthm}, let $A_n'$ be the set $\{\|\bbeta-\bbeta^*\|\leq c_1\epsilon_n\}\cap\{|\sigma^2-\sigma^{*2}|\leq c_2\epsilon_n\}\cap
\{\mbox{at most } c_3\sqrt{n\epsilon_n^2/\log p_n}$ entries of $\bbeta/\sigma$ is larger than $ a_n\}$, and let $\Omega_n$ be the event 
$\{\pi(A_n'|D_n)>1-\exp\{-c_4n\epsilon_n^2\}\}$, then we have $P^*(\Omega_n) > 1-e^{-c_5n\epsilon_n^2}$ for some $c_1$ to $c_5$.
All the following analysis is conditioned on the event $\Omega_n$, and we can ignore set $(A_n')^c$ in all the following posterior probability calculation.


Let $E_1 = \{\|\bbeta_1/\sigma-\bbeta_1^*/\sigma^*\|_\infty\leq c_1\epsilon_n, \|\bbeta_1/\sigma\|_\infty\geq a_n , |\sigma^2-\sigma^{*2}|\leq c_2\epsilon_n\}$,
where $\|\cdot\|_{\min} $ be the smallest absolute value of the entries of a vector, we
define 
\[
 \underline{\pi(\bbeta_1|\sigma^2)} = \inf_{(\bbeta_1,\sigma^2)\in E_1}\pi(\bbeta_1,\sigma^2)/\pi(\sigma^2), 
 \quad
\overline{\pi(\bbeta_1|\sigma^2)} = \sup_{(\bbeta_1,\sigma^2)\in E_1}\pi(\bbeta_1,\sigma^2)/\pi(\sigma^2).
\]

First we study the posterior probability $\pi(\xi=\xi^*|\bX,\by)$ up to the normalizing constant. For simplicity of notation, we use 
the subscript ``1'' to denote the true model $\xi^*$, and the subscript ``2'' to denote the rest $(\xi^*)^c$. Then
\begin{equation}
 \begin{split}
  &\int\frac{1}{\sigma^n}\exp\left\{-\frac{\|\by-\bX_1\bbeta_1-\bX_2\bbeta_2\|^2}{2\sigma^2}\right\}\pi(\bbeta,\sigma)I(\|\bbeta_2/\sigma\|_\infty\leq a_n, \|\bbeta_1/\sigma\|_{\min}\geq a_n)d\sigma^2d\bbeta\\
 \geq & \pi(\|\bbeta_2/\sigma\|_\infty\leq a_n)\int_{E_1}\inf_{\|\bbeta_2/\sigma\|_\infty\leq a_n}
 \frac{1}{\sigma^n}\exp\left\{-\frac{\|\by-\bX_1\bbeta_1-\bX_2\bbeta_2\|^2}{2\sigma^2}\right\}
  \pi(\bbeta_1,\sigma^2)d\sigma^2d\bbeta_1.
\end{split}
\end{equation}

The integral in the above inequality satisfies
\begin{equation}
 \begin{split}
 & \int_{E_1}\inf_{\|\bbeta_2/\sigma\|_\infty\leq a_n}  
 \frac{1}{\sigma^n}\exp\left\{-\frac{\|\by-\bX_1\bbeta_1-\bX_2\bbeta_2\|^2}{2\sigma^2}\right\}\pi(\bbeta_1,\sigma^2)d\sigma^2d\bbeta_1\\
 \geq & \underline{\pi(\bbeta_1|\sigma^2)}\int_{E_1}\inf_{\|\bbeta_2/\sigma\|_\infty\leq a_n}  
 \frac{1}{\sigma^n}\exp\left\{-\frac{\|\by-\bX_1\bbeta_1-\bX_2\bbeta_2\|^2}{2\sigma^2}\right\}\pi(\sigma^2)d\sigma^2d\bbeta_1\\
 = & \underline{\pi(\bbeta_1|\sigma^2)}\int_{E_1}\inf_{\|\bbeta_2/\sigma\|_\infty\leq a_n}  
 \frac{1}{\sigma^n}\exp\left\{-\frac{SSE(\bbeta_2)+(\bbeta_1-\hat\bbeta_1)^T\bX_1^T\bX_1(\bbeta_1-\hat\bbeta_1)}{2\sigma^2}\right\} \pi(\sigma^2)d\bbeta_1d\sigma^2\\
 \approx &\underline{\pi(\bbeta_1|\sigma^2)}\int_{E_1}\inf_{\|\bbeta_2/\sigma\|_\infty\leq a_n}  
 \frac{1}{\sigma^n}\exp\left\{-\frac{SSE(\bbeta_2)}{2\sigma^2}\right\}\pi(\sigma^2){(2\pi)}^{s/2}\sqrt{|\sigma^2(\bX^T_1\bX_1)^{-1}|})d\sigma^2\\
 \approx & \underline{\pi(\bbeta_1|\sigma^2)} \sqrt{|(\bX^T_1\bX_1)^{-1}|}\sqrt{2\pi}^{s}\inf_{\|\bbeta_2\|_\infty\leq a_n(\sigma^*+c_2\epsilon_n)} 
 \frac{\Gamma(a_0+(n-s)/2)}{(SSE(\bbeta_2)/2+b_0)^{a_0+(n-s)/2}}.
\end{split}
\end{equation}
where $\hat\bbeta_1 = (\bX_1^T\bX_1)^{-1}\bX_1^T(\by-\bX_2\bbeta_2)$.
The first approximation holds because  most probability mass of the normal density is in the region of 
$\{\|\bbeta_1-\hat\bbeta_1\|\leq C\sqrt{s/n}\}$, which is a subset of $E_1$ in probability, if $c_1$ is large. 
Similarly, the second approximation holds since the
distribution $IG(a_0+(n-s)/2, SSE(\bbeta_2)/2+b_0)$ puts most of its probability mass inside the region $\{|\sigma^2-\sigma^{*2}|\leq c_2\epsilon_n\}$.

Next, we study the posterior probability  $\pi(\xi=\xi'|\bX,\by)$ for any $\xi'\supset\xi^*$ up to the normalizing constant.
Similarly, we use the subscript ``1'' to denote the true model $\xi^*$, use the subscript ``2'' to denote $(\xi'\backslash\xi^*)$,
and use the subscript ``3'' to denote the rest $(\xi')^c$.
\begin{equation}
 \begin{split}
  &\int\frac{1}{\sigma^n}\exp\left\{-\frac{\|\by-\bX_1\bbeta_1-\bX_2\bbeta_2-\bX_3\bbeta_3\|^2}{2\sigma^2}\right\}\pi(\bbeta,\sigma) I(\|\bbeta_2/\sigma\|_{\min} > a_n, \|\bbeta_3/\sigma\|_\infty\leq a_n)d\sigma^2d\bbeta\\
 \lesssim & \pi(\|\bbeta_2/\sigma\|_{\min}> a_n,\|\bbeta_3/\sigma\|_\infty\leq a_n)\\
       &\times\sup_{\|\bbeta_3/\sigma\|_\infty\leq a_n,\bbeta_2}\int_{E_1}\frac{1}{\sigma^n}\exp\left\{-\frac{\|\by-\bX_1\bbeta_1-\bX_2\bbeta_2-\bX_3\bbeta_3\|^2}{2\sigma^2}\right\}\pi(\bbeta_1,\sigma)d\sigma^2d\bbeta_1\\
 \leq & \pi(\|\bbeta_2/\sigma\|_{\min}> a_n,\|\bbeta_3/\sigma\|_\infty\leq a_n)\\
       &\times\sup_{\|\bbeta_3/\sigma\|_\infty\leq a_n, \bbeta_2}\int_{E_1}\frac{1}{\sigma^n}\exp\left\{-\frac{SSE((\bbeta_2,\bbeta_3)^T)+(\bbeta_1-\tilde\bbeta_1)^T\bX_1^T\bX_1(\bbeta_1-\tilde\bbeta_1)}{2\sigma^2}\right\}  \pi(\bbeta_1,\sigma)d\sigma^2d\bbeta_1\\
 \leq &   \pi(\|\bbeta_2/\sigma\|_{\min}> a_n,\|\bbeta_3/\sigma\|_\infty\leq a_n)  \overline{\pi(\bbeta_1|\sigma^2)} \sqrt{|(\bX^T_1\bX_1)^{-1}|} (2\pi)^{s/2} \\
    & \times     \sup_{\|\bbeta_3\|_\infty\leq a_n(\sigma^*+c_2\epsilon_n)} 
 \frac{\Gamma(a_0+(n-s)/2)}{(SSE(\bbeta_3)/2+b_0)^{a_0+(n-s)/2}},
\end{split}
\end{equation}
where where $\tilde\bbeta_1 = (\bX_1^T\bX_1)^{-1}\bX_1^T(\by-\bX_2\bbeta_2-\bX_3\bbeta_3)$.

Therefore, combining the above results, we obtain that for any $\xi'\supset\xi^*$,
\begin{equation}\label{use4}
 \begin{split}
  \frac{\pi(\xi = \xi'|\bX,\by)}{\pi(\xi = \xi^*|\bX,\by)}
  \lesssim& \frac{\overline{\pi(\bbeta_1|\sigma^2)} }{\underline{\pi(\bbeta_1|\sigma^2)} }[p_n^{-(1+u)}/(1-p_n^{-(1+u)})]^{|\xi'\backslash\xi^*|}\\
  &\frac{\sup_{\|\bbeta_{(\xi^{*})^c}\|_\infty\leq a_n(\sigma^*+c_2\epsilon_n)} (SSE(\bbeta_{(\xi^{*})^c})/2+b_0)^{a_0+(n-s)/2}}
  {\inf_{\|\bbeta_{(\xi')^c}\|_\infty\leq a_n(\sigma^*+c_2\epsilon_n)} (SSE(\bbeta_{(\xi')^c})/2+b_0)^{a_0+(n-s)/2}}
 \end{split}
\end{equation}

It is easy to see that with probability larger than $1-4p_n\cdot p_n^{-c_6'}$, $\|\bX^TA\bvarepsilon\|_\infty\leq\sqrt{2c_6'n\log p_n} $
for any idempotent matrix $A$ and $c_6'>1$, and thus,
\begin{equation}\label{use1}\begin{split}
&SSE(\bbeta_{(\xi^{*})^c})= (\by-\bX_{(\xi^{*})^c}\bbeta_{(\xi^{*})^c})^T(I-P_{\bX_{\xi^{*}}})(\by-\bX_{(\xi^{*})^c}\bbeta_{(\xi^{*})^c})  \\
\leq& \sigma^{*2}\bvarepsilon^T(I-P_{\bX_{\xi^{*}}})\bvarepsilon+\|\bX_{(\xi^{*})^c}\bbeta_{(\xi^{*})^c}\|^2-2\sigma^{*}\bvarepsilon^T(I-P_{\bX_{\xi^{*}}})\bX_{(\xi^{*})^c}\bbeta_{(\xi^{*})^c} \\
\leq& \sigma^{*2}\bvarepsilon^T(I-P_{\bX_{\xi^{*}}})\bvarepsilon+\|\bX_{(\xi^{*})^c}\bbeta_{(\xi^{*})^c}\|^2+2\sigma^{*}\sqrt{2c_6'n\log p_n}\|\bbeta_{(\xi^{*})^c}\|_1;\\
  &SSE(\bbeta_{(\xi')^c} )= (\sigma^{*}\bvarepsilon-\bX_{(\xi')^c}\bbeta_{(\xi')^c})^T(I-P_{\bX_{\xi'}})(\sigma^{*}\bvarepsilon-\bX_{(\xi')^c}\bbeta_{(\xi')^c})\\
\geq& \sigma^{*2}\bvarepsilon^T(I-P_{\bX_{\xi^{'}}})\bvarepsilon-2(\sigma^{*}\bvarepsilon)^T(I-P_{\bX_{\xi'}})\bX_{(\xi')^c}\bbeta_{(\xi')^c}\\
\geq& \sigma^{*2}\bvarepsilon^T(I-P_{\bX_{\xi^{'}}})\bvarepsilon-2\sigma^{*}\sqrt{2c_6'n\log p_n}\|\bbeta_{(\xi^{'})^c}\|_1.
\end{split}\end{equation}
Let $\tilde p_n\triangleq c_3\sqrt{n\epsilon_n^2/\log p_n} $.
By the properties of the quantiles of the chi-square distribution (e.g. \cite{Inglot2010}), and Lemma \ref{lemmac},
with probability greater than $1-p_n^{-c_6}$ for any constant $c_6$,
\begin{equation}\label{use2}\begin{split}
   &\bvarepsilon^T(I-P_{\bX_{\xi^{*}}})\bvarepsilon-\bvarepsilon^T(I-P_{\bX_{\xi^{'}}})\bvarepsilon\leq \sigma^{*2}\{c_7{|\xi'\backslash\xi^*|}\log p_n\},\\
   &\bvarepsilon^T(I-P_{\bX_{\xi^{*}}})\bvarepsilon \in \sigma^{*2}(n-s)[1-c_8, 1+c_8],
  \end{split}
\end{equation}
hold for all $\xi'$ with $1<|\xi'\backslash\xi^*|\leq  \tilde p_n$,  when $n$ is sufficiently large and any $c_7>c_6+2$, $c_8>0$.

Combining (\ref{use1}), (\ref{use2}) and the fact that $\sqrt{n}a_np_n\sigma^*< k\sqrt{\log p_n}$ for large $n$,
we have  
\begin{equation}\label{use3}
\begin{split}
    &\frac{\sup_{\|\bbeta_{(\xi^{*})^c}\|_\infty\leq a_n(\sigma^*+c_2\epsilon_n)} (SSE(\bbeta_{(\xi^{*})^c})/2+b_0)^{a_0+(n-s)/2}}
  {\inf_{\|\bbeta_{(\xi')^c}\|_\infty\leq a_n(\sigma^*+c_2\epsilon_n)} (SSE(\bbeta_{(\xi')^c})/2+b_0)^{a_0+(n-s)/2}}
\leq \exp\{c_9|\xi'\backslash\xi^*|\log p_n\},
  \end{split}
\end{equation}
where $c_9$ is any constant satisfying $c_9 > c_7/2(1-c_8)+(k^2/2\sigma^{*2}+2\sqrt{2c_6'}k)/(1-c_8)$.

Further, it is easy to see that $ {\overline{\pi(\bbeta_1|\sigma^2)} }/{\underline{\pi(\bbeta_1|\sigma^2)} }\leq l_n^s$.
Combining it with (\ref{use4}) and (\ref{use3}), we get 
\begin{equation}
 \begin{split}
  \frac{\pi(\xi = \xi'|\bX,\by)}{\pi(\xi = \xi^*|\bX,\by)}
  \leq l_n^s[p_n^{-(1+u)}/(1-p_n^{-(1+u)})]^{|\xi'\backslash\xi^*|}
   \exp\{c_9|\xi'\backslash\xi^*|\log p_n\}.
 \end{split}
\end{equation}
By condition $1+u>2+c_6/2+k^2/2\sigma^{*2}+2\sqrt{2c_6'}k$, we can choose proper values of $c_7$ $c_8$ and $c_9$ such that $1+u-c_9=u'>1$, and 
\[
 \begin{split}
  \frac{\pi(\xi = \xi'|\bX,\by)}{\pi(\xi = \xi^*|\bX,\by)}
  \leq l_n^sp_n^{-u'|\xi'\backslash\xi^*|}, \mbox{ for any }\xi'\supset\xi^*.
 \end{split}
\]
Therefore, since $u'>1$,
\begin{equation}
\begin{split}\label{comb1}
&\sum_{\xi'\supset\xi^*, 1<|\xi'\backslash\xi^*|\leq \tilde p_n}\frac{\pi(\xi = \xi'|\bX,\by)}{\pi(\xi = \xi^*|\bX,\by)}
\leq  l_n^s
[(1+p_n^{-u'})^{p_n}-1]\asymp  l_n^sp_n^{-(u'-1)}.
\end{split}
\end{equation}

Final, we study the posterior probability  $\pi(\xi=\xi'|\bX,\by)$ for any $\xi'$ that such that $|\xi'\backslash\xi^*|\leq \tilde p_n$ and $\xi'$ doesn't include $\xi^*$, up to the normalizing constant.
Similarly, we use the subscript ``4'' to denote the model $(\xi^*\cap\xi')$, use the subscript ``5'' to denote $(\xi^*\backslash\xi')$, use the subscript ``2'' to denote $(\xi'\backslash\xi^*)$,
and use the subscript ``3'' to denote the rest $(\xi'\cup\xi^*)^c$.
Define $E_4 = \{\|\bbeta_4/\sigma-\bbeta_4^*/\sigma^*\|_\infty\leq c_1\epsilon_n, \|\bbeta_4/\sigma\|_\infty\geq a_n , |\sigma^2-\sigma^{*2}|\leq c_2\epsilon_n\}$,
and $\underline \pi=\inf_{x\in[-E_n,E_n]}g_\lambda(x)$. Then, 
\begin{equation}\label{otherset}
 \begin{split}
 &\frac{\pi(\xi=\xi'|\bX,\by)}{\pi(\xi=\xi'\cup\xi^*|\bX,\by)}\\
  =&\frac{\int_{|\sigma^2-\sigma^{*2}|\leq c_2\epsilon_n}\frac{1}{\sigma^n}\exp\left\{-\frac{\|\by-\bX\bbeta\|^2}{2\sigma^2}\right\}\pi(\bbeta,\sigma) I(\|(\bbeta_2,\bbeta_4)/\sigma\|_{\min} > a_n, \|(\bbeta_3,\bbeta_5)/\sigma\|_\infty\leq a_n)d\sigma^2d\bbeta}
  {\int_{|\sigma^2-\sigma^{*2}|\leq c_2\epsilon_n}\frac{1}{\sigma^n}\exp\left\{-\frac{\|\by-\bX\bbeta\|^2}{2\sigma^2}\right\}\pi(\bbeta,\sigma) I(\|(\bbeta_2,\bbeta_4,\bbeta_5)/\sigma\|_{\min} > a_n, \|\bbeta_3/\sigma\|_\infty\leq a_n)d\sigma^2d\bbeta}\\
  \lesssim & \max_{\{\bbeta_2,\bbeta_3,\|\bbeta_4\|_{\min}\geq a_n,|\sigma^2-\sigma^{*2}|\leq c_2\epsilon_n\}} \frac{p_n^{(1+u)|\xi^*\backslash\xi'|}}{
  {\sqrt{2\pi}\underline \pi}^{|\xi^*\backslash\xi'|}\sqrt{|\sigma^2(\bX_5^T\bX_5)^{-1}|}
  }\frac{\max_{\|\bbeta_5/\sigma\|_{\infty}\leq a_n} \exp\left\{-\frac{\|\by-\bX\bbeta\|^2}{2\sigma^2}\right\}  }
  {\max_{\|\bbeta_5/\sigma\|_{\min}\geq a_n} \exp\left\{-\frac{\|\by-\bX\bbeta\|^2}{2\sigma^2}\right\}}.\\
\end{split}
\end{equation}
It is not different to see that in probability, uniformly for all $\xi'$, $\bbeta_2,\bbeta_3,\|\bbeta_4\|_{\min}\geq a_n$ and $|\sigma^2-\sigma^{*2}|\leq c_2\epsilon_n$, we have
\[
\begin{split}
    &\max_{\|\bbeta_5/\sigma\|_{\infty}\leq a_n} {\|\by-\bX\bbeta\|^2}-\max_{\|\bbeta_5/\sigma\|_{\min}\geq a_n} {\|\by-\bX\bbeta\|^2} \\
\geq& \max_{\|\bbeta_5/\sigma\|_{\infty}\leq a_n} {\|\by-\bX\bbeta\|^2}- {\|\by-\bX_5\bbeta_5^*-\bX_2\bbeta_2-\bX_3\bbeta_3-\bX_4\bbeta_4\|^2}\geq M'|\xi^*\backslash\xi'|\log p_n
\end{split}
\]
for some $M'$ if $M_1$ (appeared in the beta-min condition) is sufficiently large, and condition $A_1(3)$ holds.

Given sufficiently large $M'$, uniformly for all $\xi'$, (\ref{otherset}) reduces to 
\[
\begin{split}
 &\frac{\pi(\xi=\xi'|\bX,\by)}{\pi(\xi=\xi'\cup\xi^*|\bX,\by)}\leq p_n^{-M''|\xi^*\backslash\xi'|},
 \end{split}
\]
for some $M''>1$. This further implies that 
\begin{equation}\label{comb2}
     \frac{\pi(\xi\mbox{ doesn't includes }\xi^*, |\xi\backslash\xi^*|\leq \tilde p_n|\bX,\by)}{\pi(\xi\supset\xi^*, |\xi\backslash\xi^*|\leq \tilde p_n|\bX,\by)}\leq
     (1+p_n^{-M''})^s-1=o(1).
\end{equation}

Combine (\ref{comb1}) and (\ref{comb2}), we conclude that with probability 1-$o(1)$, $\pi(\xi=\xi^*|\bX,\by) > 1- o(1)$.
\end{proof}

\begin{theorem}[BvM theorem]\label{BvM}
Under the conditions of Theorem \ref{vs}, 
$a_n\prec(1/p_n)\sqrt{1/(ns\log p_n)}$, and
$\lim_{n \to \infty} s\log l_n =0$, we have
\[\begin{split}
&\|\pi(\bbeta,\sigma^2|\bX,\by)
-\phi(\bbeta_{\xi^*};\hat\bbeta_{\xi^*}, \sigma^2(\bX_{{\xi^*}}^T\bX_{{\xi^*}})^{-1})
 \prod_{j\notin\xi^*} \pi(\beta_j|\sigma^2)
ig(\sigma^2, (n-s)/2, \hat\sigma^2(n-s)/2)\|_{TV}
\rightarrow 0  
\end{split}
\]
in probability, where $\phi$ denotes  the density function of a multivariate normal distribution,  $ig$ denotes the density function of an inverse gamma distribution, and 
$\hat\bbeta_{\xi^*}$ and $\hat\sigma^2$ are the MLEs of $\bbeta_{\xi^*}$ and $\sigma^2$, respectively, given data $(\by,\bX_{\xi^*})$.
\end{theorem}
\begin{proof}
Let $\btheta = (\bbeta_{\xi^*},\sigma^2)^T$, $\btheta'=\bbeta_{(\xi^*)^c}$ and let
$\pi_0(\btheta)$ denote the normal-inverse gamma distribution $\phi(\bbeta_{\xi^*};$ $\hat\bbeta_{\xi^*}, \sigma^2(\bX_{{\xi^*}}^T\bX_{{\xi^*}})^{-1})
ig(\sigma^2, (n-s)/2, \hat\sigma^2(n-s)/2)$, and 
\[\begin{split}
&\pi_1(\btheta) = C\frac{1}{\sigma^n}\exp\left\{  -\frac{\|\by-\bX_{\xi^*}\bbeta_{\xi^*}\|^2}{2\sigma^2}\right\}\pi(\sigma^2)\\
&\pi_2(\btheta)  = \prod_{j\in\xi^*}\frac{\pi(\beta_j|\sigma^2)}{\pi(\beta_j^*|\sigma^{*2})},\\
&\pi_3(\btheta,\btheta') = \exp\left\{-\frac{\|\by-\bX\bbeta\|^2-\|\by-\bX_{\xi^*}\bbeta_{\xi^*}\|^2}{2\sigma^2}\right\}
 \prod_{j\notin\xi^*}\pi(\beta_j|\sigma^2),
\end{split}\]
where $C$ normalizes $\pi_1$.
Thus, we have the posterior $\pi(\bbeta,\sigma^2|\bX,\by) \propto\pi_1\pi_2\pi_3$.

It is trivial to see that $\pi_1$ is exactly a normal-inverse-gamma distribution, i.e.
$\sigma^2\sim IG((n-s)/2+a_0, \hat\sigma^2(n-s)/2+b_0)$, and the conditional distribution of $\bbeta_{\xi^*}$ follows
$\bbeta_{\xi^*}|\sigma^2\sim N(\hat\bbeta_{\xi^*}, \sigma^2(\bX_{{\xi^*}}^T\bX_{{\xi^*}})^{-1})$,
where $\hat\btheta=(\hat\bbeta_{\xi^*},\hat\sigma^2)$.
Furthermore, as long as $n-s\rightarrow\infty$, it is not difficult to show that 
$\|IG((n-s)/2, \hat\sigma^2(n-s)/2)-IG((n-s)/2+a_0, \hat\sigma^2(n-s)/2+b_0)\|_{TV}\rightarrow 0$ with
dominating probability, i.e.,
$\|\pi_1(\btheta)-\pi_0(\btheta)\|_{TV}=o_p(1)$.


Let $\Omega_1=\{\|\bbeta_{\xi^*}-\bbeta^*_{\xi^*}\|\leq \sigma^*\epsilon_n \mbox{ and }|\sigma^2-\sigma^{*2}|<c_4\epsilon_n\}$. 
By the conditions of the theorem, if $\btheta\in\Omega_1$, then $|\pi_2-1|\leq |l_n^s-1|\rightarrow 0$. Therefore,
\[
\begin{split}
 &\int_{\Omega_1} |\pi_1(\btheta)\pi_2(\btheta) - \pi_0(\btheta)|d\btheta
\leq   \int_{\Omega_1} |\pi_1\pi_2 - \pi_1|d\btheta+\int_{\Omega_1} |\pi_1(\btheta) - \pi_0(\btheta)|d\btheta\\
\leq& \max_{\Omega_1} |\pi_2(\btheta) -1|+\int_{\Omega_1} |\pi_1(\btheta) - \pi_0(\btheta)|d\btheta = o_p(1).
\end{split}
\]

Let $\bvarepsilon(\bbeta_{\xi^*}) = \by-\bX_{\xi^*}\bbeta_{\xi^*}$, $\Omega_2= \{(\btheta,\btheta')\in \Omega_1, \|\beta_j/\sigma\|\leq a_n, \forall j\notin \xi^*\}$.
For any $(\btheta,\btheta')^T\in\Omega_2$, 
$\|\bvarepsilon(\bbeta_{\xi^*})\|$ $\in[\|\sigma^*\bvarepsilon\|\pm \sigma^*\sqrt{|\xi^*|}\sqrt n\epsilon_n]$,
and $|\|\bvarepsilon(\bbeta_\xi^*)\|^2-\|\bvarepsilon(\bbeta_\xi^*)-\bX_{\xi^{*c}}\bbeta_{\xi^{*c}}\|^2|\leq 
\|\bX_{\xi^{*c}}\bbeta_{\xi^{*c}}\|^2+2\bvarepsilon(\bbeta_\xi^*)^T\bX_{\xi^{*c}}\bbeta_{\xi^{*c}}$ $\leq
na_n^2p_n^2+2(\bvarepsilon+\bX_{\xi^*}(\bbeta_{\xi^*}^*-\bbeta_{\xi^*}))^T\bX_{\xi^{*c}}\bbeta_{\xi^{*c}}\leq
na_n^2p_n^2+O(\sqrt{n\log p_n}a_np_n)+O(\sqrt n\epsilon_n\sqrt n a_np_n)$ in probability. Since $na_np_n\prec 1/\epsilon_n$,
$|\|\bvarepsilon(\bbeta_\xi^*)\|^2-\|\bvarepsilon(\bbeta_\xi^*)-\bX_{\xi^{*c}}\bbeta_{\xi^{*c}}\|^2|=o_p(1)$.
Therefore, 
\[
\begin{split}
&\int_{\Omega_2} [\pi_3(\btheta,\btheta') - \prod_{j\notin \xi^*}\pi(\beta_j|\sigma^2)]d\btheta' \\
\leq&\int_{\Omega_2} \left|\exp\left\{-\frac{\|\by-\bX\bbeta\|^2-\|\by-\bX_{\xi^*}\bbeta_{\xi^*}\|^2}{2\sigma^2}\right\}-1\right|\prod_{j\notin \xi^*}\pi(\beta_j|\sigma^2)d\btheta'\\ 
\leq&|\exp[o_p(1)/(2\sigma^{*2}-c_4\epsilon_n)]-1|\int_{\Omega_3} \prod_{j\notin \xi^*}\pi(\beta_j|\sigma)d\btheta'
= o_p( 1).
\end{split}
\]

Combining the above inequalities, we have
\[
\begin{split}
&\int_{\Omega_2} |\pi_1\pi_2\pi_3(\btheta,\btheta') -\pi_0(\btheta)\prod_{j\notin \xi^*}\pi(\beta_j|\sigma^2)|d\btheta'd\btheta \\
\leq&\int_{\Omega_2} \pi_1\pi_2(\btheta)|\pi_3(\btheta,\btheta') - \prod_{j\notin \xi^*}\pi(\beta_j|\sigma^2)|d\btheta'd\btheta +\int_{\Omega_2} |\pi_1\pi_2(\btheta)- \pi_0(\btheta)|\prod_{j\notin \xi^*}\pi(\beta_j|\sigma^2)d\btheta' d\btheta\\
\leq & o_p(1)\int_{\Omega_1}\pi_1\pi_2(\btheta)d\btheta+
\int_{\Omega_1} |\pi_1\pi_2(\btheta)- \pi_0(\btheta)|d\btheta
=o_p(1).\\
\end{split}
\]

By Theorem \ref{genthm} and \ref{vs}, with high probability,
$\int_{\Omega^c_2}\pi(\btheta,\btheta'|D_n)\rightarrow 0$.
Also it is not difficult to verify that $\int_{\Omega^c_2}\pi_0(\btheta)\prod_{j\notin \xi^*}\pi(\beta_j|\sigma^2)d\btheta'd\btheta=o_p(1)$.
Therefore we conclude that 
\[\int|\pi(\btheta,\btheta'|D_n)-\pi_0(\btheta)\prod_{j\notin \xi^*}\pi(\beta_j|\sigma^2)|d\btheta'd\btheta=o_p(1).\]
\end{proof}

\noindent{\bf Proof of Theorem \ref{betaprior}}
\begin{proof}
It is sufficient to show that $g(\beta_i/\lambda_n)/{\lambda_n}$ satisfies condition (\ref{gencon}). Assume that 
$\underline c x^{-r}<g(x)<\bar c x^{-r}$ for sufficiently large $x$. Then
\[
\begin{split}
 \int_{a_n}^{\infty}g(x/\lambda_n)/{\lambda_n}dx=
\int_{a_n/\lambda_n}^{\infty}g(x)dx\leq \bar c\frac{1}{r-1} \{a_n/\lambda_n\}^{-(r-1)}.
\end{split}
\]
Given $\lambda_n\leq a_np_n^{-(u+1)/(r-1)}$ for some $u>0$, 
\[\bar c\frac{1}{r-1} \{a_n/\lambda_n\}^{-(r-1)} \leq c\frac{1}{r-1} p_n^{-1-u}\prec \frac{1}{2}p_n^{-1-u'},\]
where $0<u'<u$. Hence $1-\int_{-a_n}^{a_n}g(x/\lambda_n)/{\lambda_n}dx\leq p_n^{-(1+u')}$, i.e. the first inequality of (\ref{gencon}) holds.
\[
 \begin{split}
  -\log (\inf_{x\in[-E_n,E_n]}g(x/\lambda_n)/{\lambda_n})& = -\log  (\inf_{x\in[-E_n/\lambda,E_n/\lambda_n]}g(x)/{\lambda_n})\\ & \leq -\log  (\underline c (E_n/\lambda_n)^{-r}/{\lambda_n})
 =-\log \underline c +(r+1)\log(1/\lambda_n) + r\log(E_n).
 \end{split}
\]
Given that  $\log(E_n)\asymp \log p_n$, $-\log \lambda_n= O(\log p_n)$,
the second inequality of (\ref{gencon}) holds.




\end{proof}

\noindent{\bf Proof of Theorem \ref{thmssvs}}
\begin{proof}
 We first verify the condition (\ref{gencon}).
 Let $g_\lambda(x) = m_0\phi(x; 0 ,\sigma_0^2)+ m_1\phi(x; 0, \sigma_1^2)$.
Then
\[
\begin{split}
&1-\int_{-a_n}^{a_n}g_\lambda(x)dx=2[m_0(1-\Phi(a_n/\sigma_0))+m_1(1-\Phi(a_n/\sigma_1))]\\
\leq&
 m_1+2m_0(1-\Phi(a_n/\sigma_0))\leq m_1+ \frac{\sqrt 2}{a_n\sqrt \pi/\sigma_0}\exp\{-a_n^2/2\sigma_0^2\}\leq 1/p_n^{1+u'},
\end{split}
\] for some $0<u'\leq u$.
By the conditions, we also have
\[
\begin{split}
&-\log (\inf_{x\in[-E_n,E_n]}g_\lambda(x)) \leq - \log (m_1\inf_{x\in[-E_n,E_n]}\phi(x/\sigma_1))\\
 =&C+ (1+u)\log p_n + E_n^2/(2\sigma_1^2)+ \log \sigma_1
 \asymp \log p_n.
\end{split}
\]
 
 Next, we study the flatness of $l_n$. When $ E\geq x\geq a_n$,
\[
\begin{split}
&\frac{(1-m_1)\sigma_1\exp\{-x^2/2\sigma_0^2\}}{m_1\sigma_0\exp\{-x^2/2\sigma_1^2\} }= \frac{(1-m_1)}{m_1}\exp\{-\frac{x^2}{2\sigma_0^2}-\log\sigma_0+\frac{x^2}{2\sigma_1^2}+\log\sigma_1\}\\
\leq&\frac{(1-m_1)}{m_1}\exp\{-\frac{a_n^2}{2\sigma_0^2}-\log\sigma_0+\frac{E_n^2}{2\sigma_1^2}+\log\sigma_1\}\rightarrow 0.
\end{split}
\]
Note that the above convergence result holds since 
 $E_n^2/\sigma_1^2+\log\sigma_1\asymp \log p_n$
and $\sigma_0= O(a_n/\log p_n)$.
Hence ,
 \[
 \frac{g_\lambda(x)}{m_1\phi(x; 0, \sigma_1^2)} = 1 + \frac{1-m_1}{m_1}\frac{\sigma_1\exp\{-x^2/2\sigma_0^2\}}{\sigma_0\exp\{-x^2/2\sigma_1^2\}}\rightarrow 1.
\]
Therefore, we have 
\[
\begin{split}
l_n & =\max_{j\in\xi^*}\sup_{\substack{x_1,x_2\in \beta_j^*/\sigma^*\pm c_0\epsilon_n\\|x_1|,|x_2|\geq a_n}}\frac{g_\lambda(x_1)}{g_\lambda(x_2)}\\
& \asymp\max_{j\in\xi^*}\sup_{\substack{x_1,x_2\in \beta_j^*/\sigma^*\pm c_0\epsilon_n\\|x_1|,|x_2|\geq a_n}}\phi(x_1/\sigma_1)/\phi(x_2/\sigma_1)\\
&\leq \max_{j\in\xi^*}\sup_{x_1,x_2\in \beta^*_j/\sigma^*\pm c'\epsilon_n}\exp\{(x_1^2-x_2^2)/2\sigma_1^2\}\\
& =
\max_{j\in\xi^*}\exp\{2(\beta_j^*+c'\epsilon_n)c'\epsilon_n/\sigma_1^2\},
\end{split}
\]
which implies  $s\log l_n \leq O(sE_n\epsilon_n)/\sigma_1^2$.
The proof can be concluded by applying Theorems \ref{genthm}, \ref{vs} and \ref{BvM}.


\end{proof}

\end{appendix}

\end{document}


\setcounter{table}{0}
\renewcommand{\thetable}{S\arabic{table}}
\setcounter{figure}{0}
\renewcommand{\thefigure}{S\arabic{figure}}
\setcounter{equation}{0}
\renewcommand{\theequation}{S\arabic{equation}}
\setcounter{lemma}{0}
\renewcommand{\thelemma}{S\arabic{lemma}}
\setcounter{theorem}{0}
\renewcommand{\thetheorem}{S\arabic{theorem}}

\thispagestyle{empty}
\title{Supplementary Material for ``Nearly optimal Bayesian Shrinkage for High Dimensional Regression''}

 \author{Qifan Song and Faming Liang
\thanks{ 
Q. Song is Assistant Professor (email: qfsong@purdue.edu) and F. Liang is Professor (email: fmliang@purdue.edu),
 Department of Statistics,
  Purdue University, West Lafayette, IN 47907.
 }
 }

\date{}

\maketitle

\section{Inconsistency of Bayesian Lasso }

Bayesian Lasso imposes a Laplace prior on the regression coefficients $\bbeta$,  i.e., 
 \[
 g_{\lambda_n}(\beta_j)=(\lambda_n/2)\exp\{-\lambda_n|\beta_j|\}, \quad \mbox{for $j=1,\ldots,p_n$},
 \]
 where $\lambda_n$ is 
the scale parameter which may depend on $(n,p_n)$.
If $\sigma^*$ is known, then the  maximum {\it a posteriori} (MAP) estimator of Bayesian Lasso is exactly the
frequentist Lasso estimator.
In the literature, \cite{ArmaganDL2013} showed that under $p_n=o(n)$, Bayesian Lasso can attain $L_2$ consistency if $\lambda_m = O(\sqrt p_n \log n)$
(but its contraction rate is not optimal).
\cite{CastilloSHV2015,BhattacharyaPPD2015} showed that under normal means models, 
Bayesian Lasso is at best suboptimal, although the posterior consistency can still be attained.  
Note that there still remains noticeable difference between the normal means model and 
 the regression model; the design matrix of the former is $I$, and that of the latter is usually row-iid only. 
Hence, it is not trivial to extend the result of the normal means model to the high-dimensional regression model.
In what follows we conduct a simple simulation study to examine 
 the performance of Bayesian Lasso under a high dimensional setting. Let $p_n = \lfloor n^{1.5} \rfloor$, $\lambda = 2\sigma\sqrt{2.2n\log p}$ 
with $n$ increasing from 38 to 75, where $\lambda$ is chosen such that the frequentism consistency is achieved \cite{ZhangH2008}.
For each pair of $(n,p)$, we generated 32 independent datasets and estimated their respective $L_2$ and $L_1$ 
 errors of the $\bbeta$ estimator.  The results are summarized in Figure \ref{error}. 
Under the above choice of $\lambda_n$, 
the plot shows a potential trend of convergence of regression coefficients in $L_2$ errors,
but clearly no convergence in the $L_1$ norm.
The difference between the $L_1$ and $L_2$ convergence is due to that $L_1$ convergence is  much stronger as 
implied by the inequality $\|\bbeta-\bbeta^*\|_{2}\leq \|\bbeta-\bbeta^*\|_{1}\leq \sqrt{p_n}\|\bbeta-\bbeta^*\|_{2}$.
Moreover, $L_2$ convergence is not adequate for prediction consistency.
For any new $x_0$ with bounded $\|x_O\|_1$,
$|x_0^T\bbeta-x_0^T\bbeta^*|$ can be as large as $O(\sqrt p\epsilon)$.
However, under $L_1$ convergence, $|x_0^T\bbeta-x_0^T\bbeta^*|=O(\|\bbeta-\bbeta^*\|_1)=o(1)$.

\begin{figure}[htbp]
 \begin{center}
  \includegraphics[width=15cm]{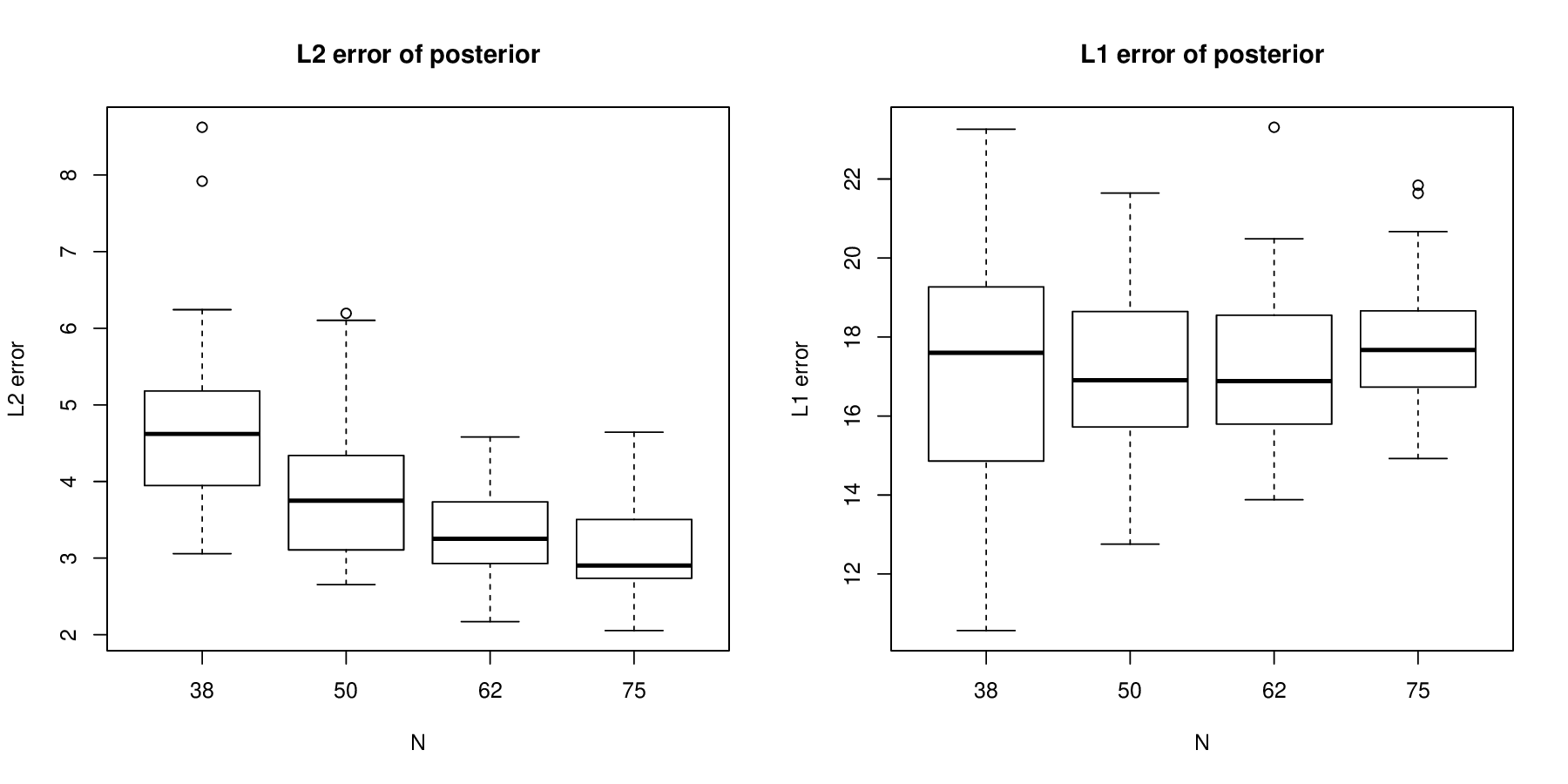}
  \caption{Convergence of Bayesian Lasso: (left) Box-plots of 
     $E_{\pi(\bbeta|D_n)}\|\bbeta-\bbeta^*\|_2$ (left) and box-plots of
     $E_{\pi(\bbeta|D_n)}\|\bbeta-\bbeta^*\|_1$ (right) over 32 replications. }\label{error}
 \end{center}
\end{figure}

To theoretically study posterior contraction for Bayesian Lasso, we impose the following condition additional to  $A_1$.
\begin{enumerate}[label=$A_1'$]
 \item \label{eig2}:  The rank of $\bX$ is $n$ and $\bX^T\bX$ has $n$ positive eigenvalues ${n\lambda_1},\dots, {n\lambda_n}$. 
 Furthermore, there exist some constant $C_0$ such that $C_0 p_n<n\lambda_i$ for all $1\leq i\leq n$.
\end{enumerate}
Since $\sum n\lambda_i \leq np_n$, this condition essentially requires that the condition number of the design matrix $\bX$ is bounded
in order by $\sqrt{n}$. 
When $p_n/n\geq C$ for some sufficiently large constant $C$, if $\bX$ has 
 iid sub-Gaussian entries (Theorem 1.1 of \cite{RudelsonV2009})
or the rows of $\bX$ are independently isotropic sub-Gaussian random vectors, 
 and the columns of $\bX$ are almost surely normalized [Theorem 
5.58 of \cite{Vershynin2012}], then condition \ref{eig2} holds with a dominating probability.

\begin{theorem}\label{blassoin}
 Assume that condition $A_1$ holds,  $\sigma^*$ is known, and a Laplace prior
 is imposed on the regression coefficients $\beta_j$'s, i.e.,  
 $\pi(\bbeta)=\prod_{j} (\lambda_n/2\sigma^*)\exp\{-\lambda_n|\beta_j|/\sigma^*\}$, where the tuning parameter $\lambda_n$  deterministically increases to $\infty$.
\begin{enumerate}
 \item If $\lim\inf p_n/n>1$ and $\lambda_n\epsilon \leq \delta_0 \sqrt p_n$ for some small $\delta_0$, then
 \begin{equation} \label{blasso2}
     \pi(\|\bbeta-\bbeta^*\|<\epsilon|D_n)=o_p(1), \mbox{ when $\bbeta^* = 0$}. 
\end{equation}
 \item If $A_1'$ hold, $p_n\succ n^2$, and  $\epsilon>0$ is any fixed sufficiently small positive number, then 
\begin{equation} \label{blasso1}
   \sup_{\|\bbeta^*\|_0\leq 1}\pi(\|\bbeta-\bbeta^*\|_1<\epsilon|D_n)\rightarrow 0, \quad a.s..
\end{equation}
regardless of the choice of the increasing order of
$\lambda_n$.
\end{enumerate}
\end{theorem}

Equation (\ref{blasso2}) implies that, in $L_2$-norm, Bayesian Lasso is at best suboptimal. 
To explain this further, let us assume that the true model size is $s=1$ and $p_n\geq n\log p_n$.
If $L_2$-consistency holds with a (near-)optimal rate as $\epsilon\preccurlyeq\sqrt{s\log p_n/n}$, 
then we must have 
$\lambda_n\geq O( \sqrt{n p_n/s\log p_n})$, which is much larger than the optimal choice of frequentism Lasso. 
With such an increasing rate of $\lambda_n$, 
by the fact that $\|\bX\bvarepsilon\|_\infty = O(\sqrt{n\log p_n})$ with a dominating probability (Bernstein inequality),
$\|\bX^T\bX\bbeta^*\|_\infty = O(n) $, and the KKT condition of LASSO, it is not difficult to show
that the MAP of Bayesian 
Lasso remains at $\bbeta = 0$ rather than inside the neighborhood of true $\bbeta^*$. This hence contradicts to the posterior consistency, since the negative logarithm of the posterior density is convex for Bayesian Lasso.
Equation (\ref{blasso1}) shows that when the design matrix is well-conditioned, the $L_1$-consistency can never be retained. In other
words, the predictive consistency fails for some newly observed $\bx$.
The failure of Bayesian Lasso is mainly due to its exponential tail. 
 If a hyperprior is imposed on $\lambda$, which changes the tail shape of the prior
 of $\bbeta$, then certain type of consistency can be achieved.

\begin{proof}
Without loss of generality, throughout the proof, we assume that the known error term $\sigma^*=1$.

We first prove equation (\ref{blasso1}).
Assume that $\bbeta^* = (\beta_1^*, 0,\dots, 0)$ and $\Delta\bbeta = \bbeta-\bbeta^*$.
For some small fixed $\epsilon$ and some $c_1>1$,
\[
\begin{split}
&\frac{\pi(\|\bbeta-\bbeta^*\|_1\leq \epsilon|D_n)}{\pi(\|\bbeta\|_1\leq \epsilon|D_n)} = \frac{\int_{\|\Delta\bbeta\|_1\leq \epsilon}\exp\{-\|\bvarepsilon-\bX\Delta\bbeta\|^2/2\}\pi(\bbeta^*+\Delta\bbeta)d\Delta\bbeta}
{\int_{\|\bbeta\|_1\leq \epsilon}\exp\{-\|\bX\bbeta^*+\bvarepsilon-\bX\bbeta\|^2/2\}\pi(\bbeta)d\bbeta}\\
\leq&\max_{\|\bbeta\|_1\leq \epsilon}\frac{\exp\{-\|\bvarepsilon-\bX\bbeta\|^2/2\}\pi(\bbeta^*+\bbeta)}
{\exp\{-\|\bX\bbeta^*+\bvarepsilon-\bX\bbeta\|^2/2\}\pi(\bbeta)}\\
\leq&\max_{\|\bbeta\|_1\leq \epsilon}\exp\{(\bbeta^{*T}\bX^T\bX\bbeta^*+2\|\bX\bbeta^*\|\|\bvarepsilon-\bX\bbeta\|)/2\}\max_{|\beta_1|\leq \epsilon_n}\frac{ \pi(\beta_1+\beta_1^*)}{\pi(\beta_1)}\\
\stackrel{as}{\leq}& \exp\{[n\beta_1^{*2}+2\sqrt n\beta_1^{*}(c_1\sqrt n+\sqrt n\epsilon)]/2 \}\max_{|\beta_1|\leq \epsilon_n}\frac{ \pi(\beta_1+\beta_1^*)}{\pi(\beta_1)}.
\end{split}
\]
Since $\pi(\beta_1+\beta_1^*)/\pi(\beta_1) = \exp\{-\lambda (|\beta_1+\beta_1^*|-|\beta_1|) \}\leq \exp\{-\lambda(|\beta_1^*|-2|\beta_1|)\}$,
the above inequality converges to zero if $\lambda\succ n$ and  $\beta_1^*>2\epsilon$.
 
Next, let's assume that  $\bbeta^* = 0$. Since the rank of the design matrix $\bX$ is $n$,
there exists a $p_n$ by $p_n$ orthogonal matrix $\Gamma$ such that the first $p_n-n$ columns of $\bX\Gamma$
is 0. We denote $\bX\Gamma = \bZ=[\bZ_1,\bZ_2]$.
 
 For some constant $c_1<1$,
 \[
 \begin{split}
 &\pi(\|\bbeta\|_1\leq \epsilon|D_n)=C \int_{\|\bbeta\|_1\leq \epsilon}\exp\{-\|\bvarepsilon-\bX\bbeta\|^2/2\}\pi(\bbeta)d\bbeta\\
\stackrel{a.s.}{\leq}&C V_{p_n}(\epsilon)\pi(0)\exp\{-(c_1\sqrt n -\sqrt n \epsilon)^2/2\}\\
=&C\frac{2^{p_n}\epsilon^{p_n}}{p_n!}(\lambda/2)^{p_n}\exp\{-(c_1-\epsilon)^2n/2\},
\end{split}\]
 where $V_{p_n}$ is the volume of $p_n$-dimensional $L_1$ ball, and $C$ is the normalizing constant.
 
 Let $\bbeta = \Gamma \bu$, $\bu=(u_1,\dots,u_{p_n})^T$, $\bu_1 = (u_1,\dots,u_{p_n-n})^T$, $\bu_2 = (u_{p_n-n+1},\dots,u_{p_n})^T$, 
 $\bu_1' = (u_1,\dots,u_{p_n-n}$, $0,\dots, 0)$ and $\bu_2' = \bu-\bu_1'$,
 \begin{equation}\label{lasso2}
 \begin{split}
 &\pi(\|\bbeta\|_1> \epsilon|D_n)=C \int_{\|\bbeta\|_1\leq \epsilon}\exp\{-\|\bvarepsilon-\bX\bbeta\|^2/2\}\pi(\bbeta)d\bbeta\\
 \stackrel{a.s.}{\geq}&
  \int_{\|\bvarepsilon-\bX\bbeta\|\leq \sqrt{n}/2}C\exp\{-\|\bvarepsilon-\bX\bbeta\|^2/2\}\pi(\bbeta)d\bbeta\\
  \geq &C\exp\{-n/4\} \int_{\|\bvarepsilon-\bX\bbeta\|\leq \sqrt{n}/2}\pi(\bbeta)d\bbeta
  =C\exp\{-n/4\} \int_{\|\bvarepsilon-\bZ\bu\|\leq \sqrt{n}/2}\pi(\Gamma\bu)d\bu\\
  =&C\exp\{-n/4\} \int_{\|\bvarepsilon-\bZ_2\bu_2\|\leq \sqrt{n}/2}\int_{\bu_1}\pi(\Gamma\bu)d\bu_1d \bu_2\\
  \geq & C\exp\{-n/4\} \int_{\|\bvarepsilon-\bZ_2\bu_2\|\leq \sqrt{n}/2}\int_{\bu_1}(\lambda/2)^{p_n}\exp\{-\lambda[\|\Gamma\bu_1'\|_1+
   \|\Gamma\bu_2'\|]\}d\bu_1d \bu_2\\
  \geq & C\exp\{-n/4\}(\lambda/2)^{p_n} \int_{\|\bvarepsilon-\bZ_2\bu_2\|\leq \sqrt{n}/2}\exp\{-\lambda\|\Gamma\bu_2'\|\}d\bu_2
   \int_{\bu_1}\exp\{-\lambda\|\Gamma\bu_1'\|_1\}d\bu_1.
  \end{split}
 \end{equation}
 
 We have that $ \int_{\bu_1}\exp\{-\lambda\|\Gamma\bu_1'\|_1\}d\bu_1\geq  \int_{\bu_1}\exp\{-\lambda\sqrt p_n\|\bu_1'\|_1\}d\bu_1
 =[2/\lambda\sqrt p_n]^{p_n-n}$.
 
 For any $\bu_2$ satisfying $\|\bvarepsilon-\bZ_2\bu_2\|\leq \sqrt{n}/2$, we have $\|\bZ_2\bu_2\|\in(c_2\sqrt n,  c_3\sqrt n) $ almost surely for $c_2<0.5<1<c_3$.
 The singular values (i.e., diagonal elements) of $\bZ_2$: $\sqrt {n \lambda_1}, \dots, \sqrt {n \lambda_n}$, which are also nonzero singular values of $\bX$, are larger than
 $\lambda_1'\sqrt p_n$,
 thus $\|\bu_2\|\leq c_4\sqrt{n/p}$ and $\exp\{-\lambda\|\Gamma\bu_2'\|\}\geq \exp\{-\lambda\sqrt p ( c_4\sqrt{n/p})\} = 
 \exp\{-c_4\lambda\sqrt n\}$.
 Also, the volume of $\{\|\bvarepsilon-\bZ_2\bu_2\|\leq \sqrt{n}/2\}$
  is 
 \[\frac{\pi^{n/2}}{\Gamma(n/2+1)}\prod_i \frac{\sqrt{n}/2}{\sqrt {n \lambda_i} }
 \leq \frac{\pi^{n/2}}{\Gamma(n/2+1)}\frac{\sqrt{n}^n/2^n}{\sqrt {p_n}^n},\]
 where the inequality holds since the geometric mean is always smaller than the quadratic mean.
  
  Combining the above results with inequality (\ref{lasso2}),
  we have
  \[
  \begin{split}
  \pi(\|\bbeta\|_1> \epsilon) 
   \stackrel{a.s.}{\geq} &C
   \frac{(\sqrt{n}/2)^n}{{\sqrt p_n}^n} \frac{\pi^{n/2}}{\Gamma(n/2+1)\sqrt{p_n^{p_n-n}}}(\lambda/2)^n\exp\{-n/4\}
   \exp\{-c_4\lambda\sqrt n\}.
  \end{split}
  \]
  
  Therefore,
  \[
  \begin{split}
  &\frac{ \pi(\|\bbeta\|_1<\epsilon)}{ \pi(\|\bbeta\|_1> \epsilon)}
  \stackrel{a.s}{\leq}\frac{(\epsilon_n\lambda\sqrt p_n)^{p_n}4^n\Gamma(n/2+1)}{(\sqrt {n\pi}\lambda)^np_n!}\exp\{n/4+c_4\lambda\sqrt n - (c_1-\epsilon)^2n/2\}.
  \end{split}
  \]
 By sterling approximation, $p_n! = O(\sqrt{p_n}p_n^{p_n}e^{-p_n})$, $\Gamma(n/2+1) = O(\sqrt{n/2}(n/2e)^{n/2})$, we have 
 \[
 \log\frac{ \pi(\|\bbeta\|_1<\epsilon)}{ \pi(\|\bbeta\|_1> \epsilon)}\stackrel{a.s}{\leq}
 p_n\log (\lambda\epsilon e/(\sqrt p_n))- n\log\lambda +\log(\sqrt n)-\log(\sqrt p_n)+O(n)+c_4\lambda\sqrt n,
 \]
 if $\lambda\prec \sqrt{p_n}$, the above term goes to $-\infty$.
 In summary, if $p_n\succ n^2$,  the $L_1$-consistency cannot be obtained for both $\bbeta^*=0$ and $\bbeta^*=(\beta_1^*, 0,\dots, 0).$

To prove equation (\ref{blasso2}), 
\[
\begin{split} 
&\frac{\pi(\|\bbeta\|_2\leq \epsilon|D_n)}{\pi(\|\bbeta\|_2> \epsilon|D_n)}\\
 =& \frac{\int_{\|\bu_2\|_2\leq\epsilon}\int_{\|\bu_1\|_2\leq 
\epsilon-\|\bu_2\|_2}\exp\{-\|\bvarepsilon-\bZ_2\bu_2\|^2/2\}\pi(\Gamma\bu)d\bu_1\bu_2}
{\int_{\bu_2}\int_{\|\bu_1\|_2> 
\epsilon-\|\bu_2\|_2}\exp\{-\|\bvarepsilon-\bZ_2\bu_2\|^2/2\}\pi(\Gamma\bu)d\bu_1\bu_2}\\
\leq&\max_{\|\bu_2\|_2\leq \epsilon}\frac{\int_{\|\bu_1\|_2\leq 
\epsilon-\|\bu_2\|_2}\exp\{-\|\bvarepsilon-\bZ_2\bu_2\|^2/2\}\pi(\Gamma\bu)d\bu_1}
{\int_{\|\bu_1\|_2> \epsilon-\|\bu_2\|_2}\exp\{-\|\bvarepsilon-\bZ_2\bu_2\|^2/2\}\pi(\Gamma\bu)d\bu_1}\\
=& \max_{\|\bu_2\|_2\leq \epsilon}\frac{\int_{\|\bu_1\|_2\leq \epsilon-\|\bu_2\|_2}\pi(\Gamma\bu)d\bu_1}
{\int_{\|\bu_1\|_2> \epsilon-\|\bu_2\|_2}\pi(\Gamma\bu)d\bu_1}\\
\leq & \max_{\|\bu_2\|_2\leq \epsilon}\frac{\int_{\|\bu_1\|_2\leq \epsilon-\|\bu_2\|_2}1d\bu_1}
{\int_{\|\bu_1\|_2> \epsilon-\|\bu_2\|_2}\exp(-\lambda\|\Gamma\bu\|_1)d\bu_1} \\
\leq & \max_{\|\bu_2\|_2\leq \epsilon}\frac{[(\epsilon-\|\bu_2\|_2)\sqrt \pi]^{p_n-n}/\Gamma(1+(p_n-n)/2)}
{\exp(-\lambda\sqrt{p_n-n}\|\bu_2\|_2)\int_{\|\bu_1\|_2> \epsilon-\|\bu_2\|_2}\exp(-\lambda\sqrt{p_n-n}\|\bu_1\|_2)d\bu_1} \\
=& \max_{\|\bu_2\|_2\leq \epsilon}\frac{[(\epsilon-\|\bu_2\|_2)\lambda\sqrt{p_n-n}])^{p_n-n}}
{\exp(-\lambda\sqrt{p_n-n}\|\bu_2\|_2)(p_n-n)\int_{t> \lambda\sqrt{p_n-n}(\epsilon-\|\bu_2\|_2)}\exp(-t)t^{p_n-n-1}dt} .
\end{split}
\]

Note that $\lambda\epsilon< \delta_0 \sqrt p_n$  and the median of the above gamma function is $O(p_n-n)$. When
$\delta_0$ is sufficiently small,
the above quantity is smaller than
\[
\frac{(\epsilon\lambda\sqrt{p_n-n})^{p_n-n}\exp(\epsilon\lambda\sqrt{p_n-n})}{(p_n-n)\Gamma(p_n-n)/2}=
O(\frac{(e\epsilon\lambda/\sqrt{p_n-n})^{p_n-n}\exp(\epsilon\lambda\sqrt{p_n-n})}{\sqrt{p_n-n}}),
\]
which is of order $o_p(1)$ if $\delta_0$ is small enough.
 \end{proof}

 \section{Additional illustrations for the toy example}
 
 This section provides some additional illustrations for the theoretical results obtained in Section 4.

 Figure \ref{toy2} shows the posterior boxplots produced by BCS with two different values of $\gamma$,
 either larger or smaller than the optimal value $\hat\gamma$.  
 It indicates that different values of $\gamma$ lead to different degrees of 
 posterior concentration for the false covariates, while the posterior distribution
 of the true covariates are unchanged.  This phenomenon agrees with our BvM approximation result.
 However, for Bayesian Lasso, Figure \ref{toy2} indicates that 
 neither increasing nor decreasing the value of $\lambda$ won't remedy the posterior inconsistency.

\begin{figure}[htb]
 \begin{center}
  \includegraphics[width=15cm]{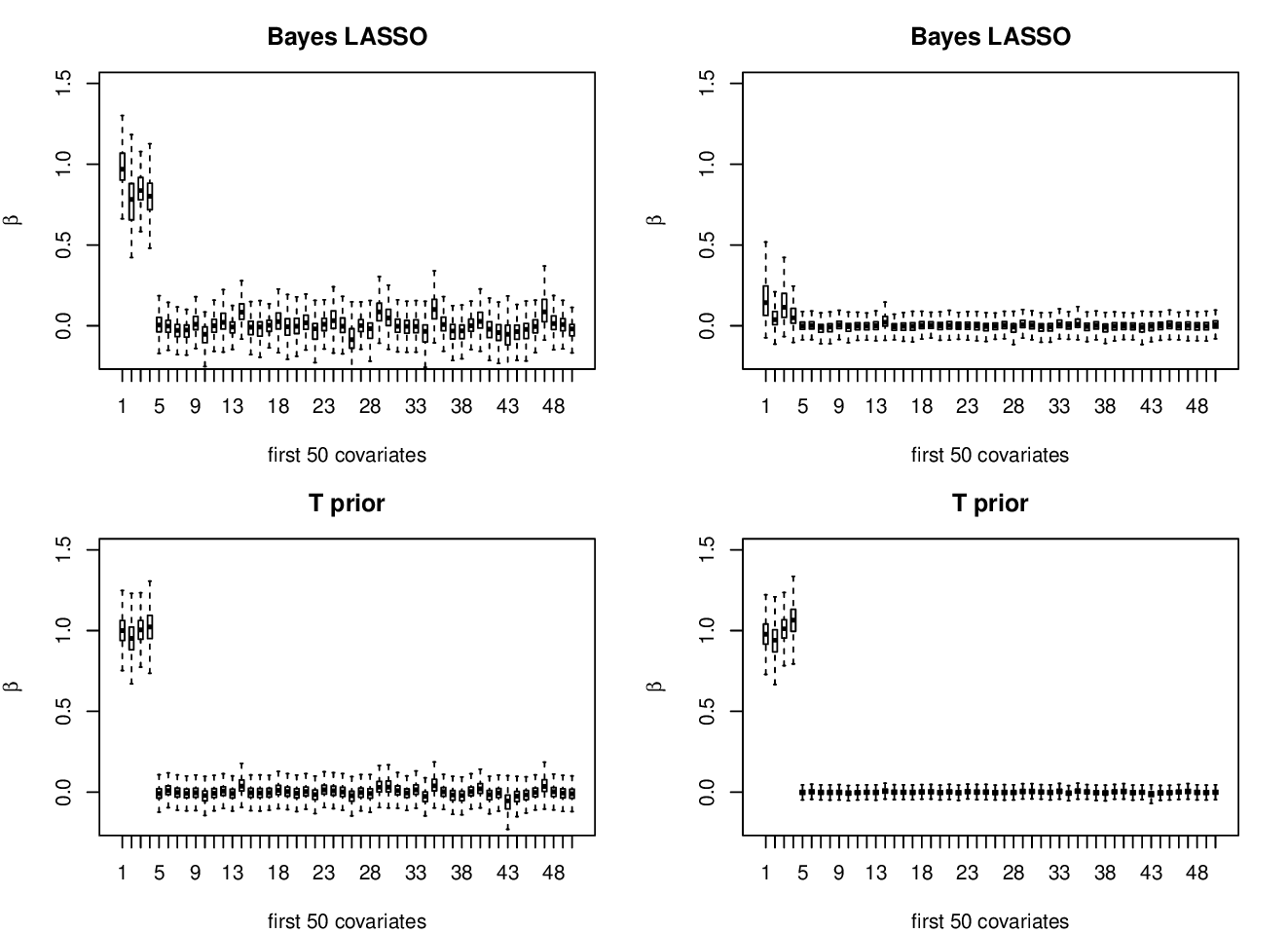}
  \caption{Sensitivity analysis of Bayesian Lasso and BCS to the scale parameter: 
  The left and right panels were produced with larger and smaller scale parameters, 
  respectively. 
  }\label{toy2}
 \end{center}
\end{figure}

Figure \ref{toy3} shows the confidence intervals by de-bias Lasso. Apparently, as shown in the plot, the true and false covariates have about the same width confidence intervals.  

\begin{figure}[htbp]
 \begin{center}
  \includegraphics[width=10cm]{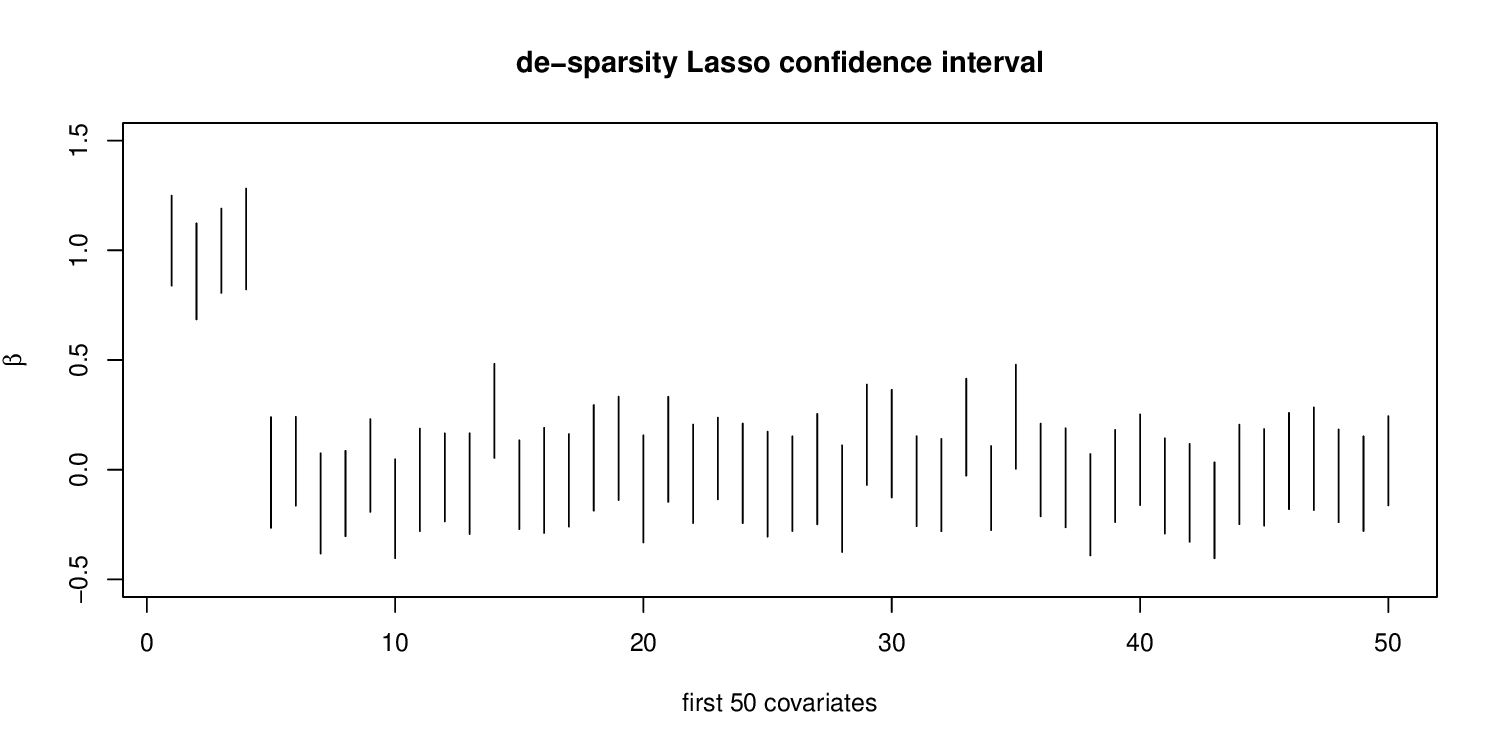}
  \caption{
  95\% marginal confidence intervals produced by de-biased Lasso 
  for the first 50 covariates.}\label{toy3}
 \end{center}
\end{figure}

\section{Simulation study of BCS under $\lambda_n\sim \mathcal {C}^{+}(0,1)$}
In the main text, we suggest to tune the global shrinkage parameter $\lambda_n$ such that it attains the minimal ``BIC-like score''. In the Bayesian literature, a popular way to tune the global shrinkage parameter is to impose a hyper prior on it. For example, one may let $\lambda_n$ be subject to a half-Cauchy prior \cite[The horseshoe,][]{CarvalhoPS2010}, i.e., $\lambda_n\sim \mathcal {C}^{+}(0,1)$. In such a way, it is not necessary to conduct multiple posterior simulations (under different levels of $\lambda_n$). However, our numerical studies show that this choice leads to inferior Bayesian inference.

We conducted additional simulations under the same settings of data generation as in Sections 5.1 and 5.2 with $t$ shrinkage and $\lambda_n\sim \mathcal {C}^{+}(0,1)$. The simulation results were reported in Table \ref{Tg1}. It is worth to note that the prior distribution for each $\beta_j/\sigma$ is $\beta_j/\sigma\sim t_3*\lambda_n$, where $\lambda_n\sim \mathcal {C}^{+}(0,1)$ and $t_3$ denotes the $t$-distribution with df=3. As a consequence, our proposed model selection rule doesn't work anymore, since the thresholding value $a$ (for $\pi(|\beta_j/\sigma|>a)=1/p_n$) can be of an order of  hundreds and the null model is always selected. Hence, in this simulation, we follow \cite{VanSV2017_2} to select predictors based on the posterior credible intervals: A predictor is selected if its 95\% credible interval excludes 0. 

As shown by Table S1, the $L_1$-error of $\bbeta$ is larger than those obtained by BCS (reported in Tables 1-4 of the main text). Although the selection of models and the coverage of credible intervals are similar to those by BCS (reported in Tables 1-4 of the main text), the credible intervals are much wider. The comparisons suggest that the half-Cauchy prior over $\lambda_n$ leads to  insufficient prior shrinkage and less accurate posterior concentration. In addition, the histograms of posterior samples obtained in the simulations do not follow the approximation (2.8) in the main text, which suggests that the asymptotic posterior shape result (i.e. Theorem 2.4) doesn't hold with the half-Cauchy global shrinkage prior.

\begin{table}
\caption{Simulation results of BCS under $\lambda_n\sim \mathcal {C}^{+}(0,1)$.}\label{Tg1}
\begin{center}
\begin{tabular}{ccccc}
 \hline
           & $n=80$, $p=201$ independent            & $n=80$, $p=201$        & $n=100$, $p=501$   & $n=100$, $p=501$  \\
           &independent&dependent &independent&dependent \\ \hline
$L_1$ error of $\bbeta_{\xi^*}$          &0.4213          &0.6550            & 0.3558  &0.5701\\ 
Standard error                           &0.0178          &0.0301            & 0.0180  &0.0296\\ \hline
$L_1$ error of $\bbeta_{(\xi^*)^{c}}$    &8.0597          &11.0509           &16.7710  &19.5830\\ 
Standard error                           &0.0860          &0.1197            & 0.0844  &0.1158\\\hline
\hline
$|\hat\xi\cap\xi^*|$                     & 3              &2.8929            &2.9821 &2.9107      \\ 
Standard error                           & ---            &0.0294            &0.0126 &0.0270  \\\hline
$|\hat\xi\cap(\xi^*)^{c}|$               & 0.2857         &0.3571            &0      &0.0089 \\ 
Standard error                           & 0.0587         &0.0536            &---    &0.0089 \\\hline
\hline                                                     
Coverage of $\xi^*$                      & 0.9077   & 0.8809    &0.9494  &0.9018  \\
Average length                           & 0.6144   & 0.8458    &0.5834  &0.8203  \\ \hline
Coverage of $(\xi^*)^{c}$                & 0.9983   & 0.9979    &1.0000  &0.9999  \\ 
Average length                           & 0.3557   & 0.4756    &0.2460  &0.3240  \\ \hline
\end{tabular}
\end{center}
\end{table}

\section{Simulation study of BCS under two-Gaussian mixture prior}
Theorem 3.2 in the main text shows that a two-Gaussian mixture prior, under proper hyperparameter settings, achieves posterior consistency, model selection consistency and posterior asymptotically normality. The following simulation aims to validate our theorem. 
This simulation study is under the same settings of data generation as in Sections 5.1 and 5.2 with the prior specification:
\begin{equation}
  \beta_j/\sigma \sim (1-\xi_j)\mbox{N}(0,\sigma_0^2)+\xi_j\mbox{N}(0,\sigma_1^2),
\quad \xi_j\sim\mbox{Bernoulli}(m_1),
\end{equation}
where we choose $m_1=1/p_n^{1.7}$, $\sigma_0^2=1/(np_n)$ and $\sigma_1^2=p_n^{1.5}$.

We evaluate the $L_1$ estimation error of posterior mean estimator and the coverage of posterior marginal percentile credible intervals. As for the model selection results, as discussed in the main text, we choose the threshold $a$ such that $\pi(|\beta_j/\sigma|>a) = 1/p_n$, and the Bayesian model selection estimator is $\hat\xi=\{j: \pi(|\beta_j/\sigma|>a|D_n)>0.5\}$. For convenience, $a$ can be approximated by $a\approx \sigma_0\Phi^{-1}(1-1/2p_n)$. The results are summarized in Table \ref{Tg2}.

From Table \ref{Tg2}, we observe that: On the one hand, this Bayesian procedure almost perfectly selects the true model. On the other hand, its shrinkage effect on the false predictors is not as strong as $t$-prior, yielding a larger $L_1$ error of $\bbeta_{(\xi^*)^c}$; and its coverage performance of the true predictors is not satisfactory. We believe it is because the hyperparameters are not optimally tuned, especially for the values of $\sigma_1^2$ and $\sigma_0^2$. However, tuning all three hyperparameters simultaneously (i.e., $m_1$, $\sigma_1^2$ and $\sigma_0^2$) is usually not feasible in statistical training. Even though, we can see that the performance of the two-Gaussian mixture prior is still much better than Bayesian Lasso result.

 \begin{table}
\caption{Simulation results of BCS under two-Gaussian-mixture prior.}\label{Tg2}
\begin{center}
\begin{tabular}{ccccc}
 \hline
           & $n=80$, $p=201$ independent            & $n=80$, $p=201$        & $n=100$, $p=501$   & $n=100$, $p=501$  \\
           &independent&dependent &independent&dependent \\ \hline
$L_1$ error of $\bbeta_{\xi^*}$          &0.3497          &0.5329            & 0.2604  &0.3812\\ 
Standard error                           &0.0309          &0.0583            & 0.0191  &0.0411\\ \hline
$L_1$ error of $\bbeta_{(\xi^*)^{c}}$    &0.1185          &0.0993            & 9.0376  &8.9761\\ 
Standard error                           &0.0067          &0.0075            & 0.0465  &0.0368\\\hline
\hline
$|\hat\xi\cap\xi^*|$                     & 2.9285         &2.8571            &2.9643  &2.9196      \\ 
Standard error                           & 0.0275         &0.0420            &0.0176  &0.0287  \\\hline
$|\hat\xi\cap(\xi^*)^{c}|$               & 0.0982         &0.0536            &0.0178  &0.0089 \\ 
Standard error                           & 0.0282         &0.0215            &0.0125  &0.0089 \\\hline
\hline                                                     
Coverage of $\xi^*$                      & 0.7411         & 0.6904           &0.7261  &0.7351  \\
Average length                           & 0.3646         & 0.4151           &0.2949  &0.3522  \\ \hline
Coverage of $(\xi^*)^{c}$                & 0.9998         & 1.0000           &1.0000  &1.0000  \\ 
Average length                           & 0.0256         & 0.0245           &0.0131  &0.0130  \\ \hline
\end{tabular}
\end{center}
\end{table}

\vskip 0.2in
\bibliographystyle{plain}
\bibliography{ref}